\documentclass[12pt, twoside,reqno]{amsart}

\usepackage{amsmath,amsthm,amssymb,amsfonts,epsfig,color,graphicx,enumerate,enumitem,accents}
\usepackage{rotating,multirow,color,enumerate,url,mathrsfs,mathrsfs}
\usepackage{arydshln}
\usepackage{bbm}
\usepackage{cases}
\usepackage{esint}
\usepackage{gensymb}
\usepackage{epstopdf}
\usepackage{comment}
\usepackage{leftidx}
\usepackage{graphicx}
\usepackage{mathtools}
\usepackage{enumerate}
\usepackage[caption=false]{subfig}
\usepackage{amsaddr}
\usepackage[normalem]{ulem}
\usepackage{tabularx}
\usepackage{lineno}
\usepackage{booktabs}

\newtheorem{theorem}{Theorem}[section]
\newtheorem{corollary}[theorem]{Corollary}
\newtheorem{lemma}[theorem]{Lemma}
\newtheorem{proposition}[theorem]{Proposition}
\newtheorem{conjecture}[theorem]{Conjecture}

\theoremstyle{definition}

\newtheorem{example}[theorem]{Example}

\numberwithin{equation}{section}

\newcommand{\R}{\mathbb{R}}

\newcommand{\Rd}{\R^d}

\newcommand{\weakstar}{{\overset{\ast}{\rightharpoonup}}}
\newcommand{\mbf}[1]{{\mathbf{#1}}}

\newcommand{\opnorm}[1]{{\left\vert\kern-0.25ex\left\vert\kern-0.25ex\left\vert #1 
		\right\vert\kern-0.25ex\right\vert\kern-0.25ex\right\vert}}

\def\O{\Omega}
\newcommand{\Ob}{{\overline{\Omega}}}
\newcommand{\bO}{{\partial\Omega}}

\newcommand{\argu}{{\,\cdot\,}}

\newcommand{\Mes}{{\mathcal{M}}}

\newcommand{\pairing}[1]{{\left \langle #1 \right \rangle}}
\newcommand{\norm}[1]{\Arrowvert #1 \Arrowvert}
\newcommand{\abs}[1]{{\left \lvert #1 \right \rvert}}

\newcommand{\eps}{\varepsilon}

\newcommand{\one}{{{\bf 1}
		\kern-0,28em \rm l}}

\begin{document}

	\title{The Isotropic Material Design methods with the cost expressed by the $L^p$-norm}

	

	\author{Karol Bo{\l}botowski, S{\l}awomir Czarnecki, Tomasz Lewi\'{n}ski}


	\date{\today}
	
	\begin{abstract} 
		The paper concerns the problem of minimization of the compliance of linear elastic structures made of an isotropic material. The bulk and shear moduli are the design variables, both viewed as non-negative fields on the design domain. The design variables are subject to the isoperimetric condition which is the upper bound on the two kinds of  $L^p$-norms of the elastic moduli. The case of $p=1$ corresponds to the original concept of the Isotropic Material Design (IMD) method proposed in the paper: S. Czarnecki, Isotropic material design, Computational Methods in Science and Technology, 21 (2), 49–64, 2015. In the present paper the IMD method will be extended by assuming the $L^p$-norms-based cost conditions. In each case the optimum design problem is reduced to the pair of mutually dual problems of the mathematical structure of a theory of elasticity of an isotropic body with non-linear power-law type constitutive equations. The state of stress determines the optimal layouts of the bulk and shear moduli of the least compliant structure. The new methods proposed deliver the upper estimates for the optimal compliance predicted by the original IMD method.
\end{abstract}

\maketitle


	\dedicatory{}
	
	
	\maketitle

	\vskip1cm
	
	
\section{Introduction}

\subsection{On  designing  isotropic inhomogeneity in stiffness optimization}

The contemporary 3D printing methods make it possible to fabricate structural elements of smoothly graded elastic properties, see \cite{zhang2022,feng2021,xu2020,khan2024,montemurro2023}. The designed inhomogeneity of elastic moduli serves for attaining appropriate stiffness characteristics of the structure keeping the prescribed lightness condition. The designs made of latticed materials working in the stretching dominated mode can be fabricated such that they possess macroscopically isotropic properties. The material of the planar structural elements can be made macroscopically isotropic by introducing of the lattice having the rotational centrosymmetry with respect to the $120\degree$ angle keeping the same areas of cross sections of the ligaments (i.e. interconnecting bars) within a unit cell or representative volume element (RVE).  The material of the spatial elements can also be made isotropic by tiling their domain by the lattice RVE discovered in \cite{gurtner2014stiffest,gurtner2014cor} composed of two families of ligaments of different axial stiffnesses. It is worth noting that, till now, 3D isotropic lattices of constant axial stiffness of the ligaments have not been discovered.

The effective stiffnesses of lattices are proportional to the relative density $c$ of the material, see \cite{christensen1986}. If the ligaments are connected by thin membranes then the effective stiffnesses are weaker and are of order $c^a$, where  the parameter $a$ assumes values between 2 and 3,  see \cite{christensen1986}. Thus, the lattices are characterized by the best stiffness to weight ratio or can be viewed as the minimizers of the problem of maximization of stiffness with the cost condition expressed by the given volume condition. Thus, the mentioned theoretical results by \cite{christensen1986} justify assuming the unit cost as a linear combination of the effective elastic moduli of the latticed material; such a unit cost will measure the density of mass of the lattice material.
By virtue of lattice realizations of the RVE in modelling  effective isotropy, one can assume that the optimal designs are macroscopically isotropic. Due to isotropy the orientation of RVE will not be designed by additional angular variables which would determine the directions of the ligaments. Thus, the assumption of isotropy of the lattice material simplifies both the optimization process as well as the process of fabrication. This is done at the sacrifice of increasing the cost of the design, since the anisotropic models are better designs as directing the ligaments along trajectories of principal stresses. On the other hand, the isotropic designs will be better suited to respond to multiple loads.
The isotropy is characterized by two fields: the bulk modulus $k(x)$  and  the shear modulus $\mu (x)$,    $x$ is a point of the feasible domain $\Omega $; its dimension $d$ equals 2 or 3. The bulk and shear moduli are closely related to the eigenvalues of the Hooke tensor $\bf{C}$. In the planar problem ($d=2$) the modulus $2k$ is the single eigenvalue (of multiplicity equal 1) and $2\mu $ is the eigenvalue of multiplicity equal 2. In the spatial case ($d=3$) the modulus $3k$ is a single  eigenvalue (i.e. of multiplicity equal 1) while $2\mu $ is an eigenvalue of multiplicity equal 5. Thus, the eigenvalues of tensor $\mbf{C}$ are:
\begin{enumerate} [label={\alph*)}]
	\item 	Case of $d=2$:  $2k$, $2\mu $, $2\mu$,
	\item	 Case of $d=3$:  $3k$, $2\mu $, $2\mu $, $2\mu $, $2\mu $, $2\mu $.
\end{enumerate}
Each eigenvalue corresponds to an eigenstate. The first eigenvalue corresponds to the hydrostatic state. The eigenvalue $2\mu $ corresponds to:
\begin{enumerate} [label={\alph*)}]
	\item 	Case of $d=2$:  two eigenmodes of pure shear,
	\item	 Case of $d=3$:   five eigenmodes of pure shear.
\end{enumerate}
The structure of the corresponding eigentensors is discussed in \cite{rychlewski1984} and \cite{walpole1984}. 
Assume that all the stress modes are given the same values of the weights. Then the trace of Hooke tensor, or $\rm{tr} \,\bf{C}$, can be assumed as the unit cost of the design:
\begin{enumerate} [label={\alph*)}]
	\item 	Case of $d=2$:  $2k + 2\mu + 2\mu = \rm{tr} \,\bf{C}$,
	\item	 Case of $d=3$:   $3k+2\mu +2\mu +2\mu +2\mu +2\mu = \rm{tr}\, \bf{C}$,
\end{enumerate}
or, equivalently,
\begin{equation}
	\label{eq:beta}
	{\rm{tr}}\,{\bf{C}} = dk + 2{\beta ^2}\mu ,  \qquad	  \beta  = {\left( {\frac{1}{2}d(d + 1) - 1} \right)^{\frac{1}{2}}}
\end{equation}   		 

The fields $k, \mu $ are natural design variables since they only obey the independent conditions $k \ge 0$, $\mu  \ge 0$ assuring the density of elastic energy being non-negative, the same for both: $d=2$ and $d=3$ cases. 

\subsection{On the Free Material -- and Isotropic Material Design methods}

The aim of the present paper is to set and solve the problem of constructing the optimum distribution of both the elastic moduli of isotropy within a given design domain $\Omega $ to minimize  the compliance of the body subject to a given static load. The compliance is the work done by the load. Minimization of the compliance means maximization of the stiffness of the structure.  For simplicity, the kinematic loads are absent. The unit cost of the design will be at each point $x$ of the domain $\Omega $ assumed as the value of the trace of the Hooke tensor, see \eqref{eq:beta}. 
Let us recall now the formulation of the Isotropic Material Design (IMD) problem: 
construct the non-negative fields $k,\mu $ within a given feasible $d$-dimensional domain $\Omega $  to minimize the compliance under the cost condition
\begin{equation}
	\label{eq:classical_cost}
	\int_\Omega  {\left[ {dk + 2{\beta ^2}\mu } \right]dx \le \Lambda }
\end{equation}                                                                            
where  $\Lambda $ represents the given highest possible total cost of the design. Note that here $dk$ means $2k$ if  $d=2$  and $3k$ if $d=3$, whilst $dx$ stands for the integration with respect to the Lebesgue measure.
The IMD method has been originally proposed in \cite{czarnecki2015a} and \cite{czarnecki2015b}. The IMD method has been then developed in the papers \cite{czarnecki2020recovery,czarnecki2024} and \cite{czarnecki2017a} in which also the underlying microstructures have been pointwise constructed by solving the inverse homogenization problems. Its extension towards multiple load problem has been given in \cite{czarnecki2017b,lewinski2021}. The IMD method can also be extended to the elasto-plasticity within the Ilyushin-Hencky-Nadai setting, see \cite{czarnecki2021isotropic}.

The IMD method is an isotropic counterpart of the Free Material Design (FMD) method proposed in \cite{ringertz1993} and developed in \cite{bendsoe1994,zowe1997,kocvara2008}, see the review paper by \cite{haslinger2010}. Within FMD all the components of the Hooke tensor $\bf{C}$ are design variables, while the unit cost is either given as trace or as the Frobenius norm of the Hooke tensor. In case of a single load case, both FMD approaches with the mentioned unit costs lead to the same result. 

The present work refers to those papers on FMD in which no additional pointwise conditions on the unit cost are imposed. Thus, it can degenerate to zero as well as can blow up in some regions. In terms of functional analysis, the cost condition \eqref{eq:classical_cost} alone implies that the set of feasible fields $k,\mu$ is merely bounded in $L^1$. Since the latter space is not reflexive, there is no compactness result to claim existence of a solution in a function space. As a consequence, the search of the optimal moduli must be extended to the space of positive measures; the material can concentrate on certain surfaces ($d=3$) or curves ($d=2$ and $d=3$). This version of the FMD method needs the measure-theoretic tools to justify the setting as well as to interpret the final results, see \cite{bolbotowski2022b}.

The solutions to the FMD method, corresponding to the case of a single load, are singular: only one component of the Hooke tensor becomes non-zero, irrespective of the form of the cost condition, i.e. whether the trace of $\mbf{C}$ or the Frobenius norm of $\mbf{C}$ is taken as the unit cost. If the cubic symmetry is imposed, then also one of the three elastic moduli vanishes everywhere, see \cite{czubacki2015}.  The above inevitable singularity is eliminated only by imposing isotropy, i.e. it is not present within the IMD method. However, due to lack of any constraints apart from the conditions of non-negativity of the bulk and shear moduli there appear subdomains where one of the moduli vanishes and, in some problems, there appear subdomains where both the moduli vanish, hence these domains are simply voids, which changes the feasible domain into a material design domain. Developed in \cite{czarnecki2024} the inverse homogenization methods make it possible to design smoothly graded latticed 2D materials in all the subdomains, even there where one of the moduli vanishes. One can construct latticed microstructures whose effective Poisson ratio varies within the whole range: $[-1,1]$ in 2D and $[-1,1/2]$ in 3D case, see Figs. 5-13  in \cite{czarnecki2024} showing the relevant RVEs. Let us stress: for the inhomogeneous layouts of the moduli $k(x), \mu(x)$ one can construct an underlying microstructure whose effective properties are isotropic and smoothly, with appropriate accuracy, model the optimal solutions, see Fig.  17  in \cite{czarnecki2024} showing such underlying structure. Nowadays such structures can be printed, which makes the IMD-like optimum design methods directly applicable.

The weak assumptions $k \ge 0,{\rm{ }}\mu  \ge 0$ are the great virtue of the IMD method since the domains where both the moduli vanish are cut out from the feasible domain, which makes the IMD method one of the methods of optimization of structural topology. On the other hand, the IMD method requires solving rather difficult problem of minimization of a functional with linear growth, which necessitates special numerical approaches, see \cite{czarnecki2015a,czarnecki2015b} and \cite{bolbotowski2021}.

The fundamental mathematical problem to be solved in the IMD method is stress-based, with the functional of linear growth. Consequently, its dual involves a non-smooth functions and the optimality conditions linking both the formulations are not invertible. Although the dual problem involves displacements, it cannot be solved with using the relevant finite element codes, since they require the constitutive equations expressed by certain functions. One of the motivation of the present work is to introduce modifications into the original IMD method to make it solvable with using available FEM codes for non-linear elasticity.

\subsection{Two modifications of the cost conditions  of the IMD method}

The present paper puts forward two new versions of the IMD method. Both the modifications proposed concern the formula for the cost of the design. In place of the formula \eqref{eq:classical_cost} we propose the following alternative cost definitions.
\begin{enumerate} [label={\Roman*.}]
	\item 	Bounding the $L^p$-norm of the vector function $(2k, 2\mu ,2\mu)$, if $d=2$,  and of the vector function
$(3k, 2\mu ,2\mu,2\mu,2\mu,2\mu )$   if $d=3$. The respective optimum design method will be named: vp-IMD.
	\item	Bounding the $L^p$-norm of the scalar function $dk + 2{\beta ^2}\mu $. The respective optimum design method will 
	be named: sp-IMD.
\end{enumerate}

It occurs that the mathematical structure of the problems of the methods vp-IMD and sp-IMD are similar to those known from the non-linear elasticity with the power law constitutive equations. Thus, these problems can be solved with using the FEM codes for non-linear elasticity. In this paper, however, the special stress-based numerical method will be developed. On the other hand, the methods do not give any tools to create the voids, the predicted fields of the moduli cannot vanish identically in some subdomain od the feasible domain; the cutting out property of the IMD method is lost.

It is worth mentioning that recently a wide family of convex costs were studied in \cite{buttazzo2023} where the material distribution of a heat conductor is modeled by a single scalar non-negative function. The two costs proposed above are a particular case of the costs studied therein. The novelty here consists in addressing the problem of optimal design in elasticity, as well as in working with a vectorial variable $(k,\mu)$.

It will be shown that the optimal compliances predicted by the  vp-IMD method estimate the IMD results from the lower side. In the examples discussed the results of the sp-IMD occur to be the upper estimates of the IMD results.

The proposed methods: vp-IMD and sp-IMD lead to specific  results in the extreme cases of $p$ tending to infinity. The optimal moduli predicted by vp-IMD occur to satisfy the equality $d\hat{k} = 2\hat{\mu}$ and become homogeneous within the feasible domain. Thus, for both cases of $d=2$ and $d=3$  the optimal Poisson ratio  $\hat{\nu}$ vanishes identically, which is a known property of a cork. The optimal moduli predicted by sp-IMD occur to satisfy the equality $d\hat{k} + 2{\beta ^2}\hat{\mu} = \rm{const}$ which means that the field of the trace of Hooke’s tensor becomes homogeneous within the feasible domain. Thus, even in this extreme case one of the optimal moduli (bulk or shear) can vanish (but not both of them). Then, the optimal Poisson ratio attains either the lower bound $-1$  or the upper bound:  $\frac{1}{2}$ in case of $d=3$ and 1 in case of $d=2$.

Let us display the modifications I, II  corresponding to the vp-IMD and sp-IMD methods   explicitly.

\noindent \textit{I.	Formulation of the cost condition in the method vp-IMD}

The condition \eqref{eq:classical_cost} is modified to the form
\begin{equation}
	\label{eq:vp_cost}
	{\left( {\int_\Omega  {\left[ {{{\left( {dk} \right)}^{p}} + \left( {\tfrac{1}{2}d(d + 1) - 1} \right){{\left( {2\mu } \right)}^{p}}} \right]dx} } \right)^{\frac{1}{{p}}}} \le \Lambda 
\end{equation}                                                                      
 
Consider a scalar function {$f\in {L^p}\left( \Omega  \right)$} and a vector function:  {${\bf{f}} = \left( {{f_1},...,{f_n}} \right) \in {\left( {{L^p}\left( \Omega  \right)} \right)^n}$}.  The expressions 
\begin{equation}
	{\left\| f \right\|_p} = {\bigg( {\int_\Omega  {{{\left| {{f_{}}} \right|}^p}dx} } \bigg)^{1/p}},     \qquad                {\left\| {\bf{f}} \right\|_p} = {\bigg( {\int_\Omega  {\sum\limits_{i = 1}^n {{{\left| {{f_i}} \right|}^p}dx} } } \bigg)^{1/p}}
\end{equation}                                         
are the  $L^p$-norms of these functions  {if $p \in [1,\infty)$ (case $p=\infty$ requires the notion of essential supremum)}. The left hand side of \eqref{eq:vp_cost} can be interpreted as the $L^p$-norm of the vector function ${\bf{f}} = \left( {dk,\beta^{2/p} 2\mu } \right)$ with $n=2$. The condition \eqref{eq:vp_cost} may be re-written as
\begin{equation}
	{{\left\| {\left( {dk,\beta^{2/p} 2\mu } \right)} \right\|_{p}} \le \Lambda.}
\end{equation}

\noindent \textit{II.	Formulation of the cost condition in the method sp-IMD}

The condition \eqref{eq:classical_cost} is modified to the form
\begin{equation}
	\label{eq:sp_cost}
	{\left( {\int_\Omega  {{{\left[ {\left( {dk} \right) + \left( {\frac{1}{2}d(d + 1) - 1} \right)\left( {2\mu } \right)} \right]}^{p}}dx} } \right)^{\frac{1}{{p}}}} \le \Lambda
\end{equation}
 The left hand side of \eqref{eq:sp_cost} can be interpreted as the $L^p$-norm of the scalar function $f = dk + {\beta ^2}2\mu $, where $\beta $  is given by \eqref{eq:beta}, and written as
\begin{equation}
	{\left\| {\left( {dk + {\beta ^2}2\mu } \right)} \right\|_{p}} \le \Lambda 
\end{equation}

It will occur that the  generalized IMD problems with the cost condition \eqref{eq:vp_cost}  will lead to a static problem of an isotropic body with the power law type constitutive equations.  These models of elasticity are the subject of the study \cite{castaneda1998}. The cost condition (1.6) will lead to a different setting of non-linear elasticity.

\subsection{Adopted conventions and notation}
\label{ssec:notation}

Throughout the paper a conventional notation is applied. The design domain $\Omega$ will be a bounded open subset of $\Rd$ with a Lipschitz continuous boundary. In case of $d=3$ the domain is parametrized by the Cartesian system $\left( {{x_1},{x_2},{x_3}} \right)$ with the orthogonal basis ${{\bf{e}}_i}$, $i=1,2,3$;  ${{\bf{e}}_i} \cdot {{\bf{e}}_j} = {\delta _{ij}}$, where $ \cdot $  is the scalar product in \({\mathbb{R}^d}\). The Euclidean norm of \({\bf{p}} \in {\mathbb{R}^d}\) is defined by $\left\| {\bf{p}} \right\| = \sqrt {{\bf{p}} \cdot {\bf{p}}} $.  The set of second rank symmetric tensors is denoted by $E_s^2$. 
The identity tensor in $E_s^2$ is ${\bf{I}} = \sum_{i=1}^d {{\bf{e}}_i} \otimes {{\bf{e}}_i}$; repetition of indices implies summation. The scalar product of ${\boldsymbol{\sigma}},\;{\boldsymbol{\eps}} \in E_s^2$ is defined by: ${\boldsymbol{\sigma}} \cdot {\boldsymbol{\eps}} =\sum_{i,j=1}^d {\sigma _{ij}}\,{\varepsilon _{ij}}$. The already introduced Euclidean norm of ${\boldsymbol{\sigma}} \in E_s^2$ is defined by $\left\| {\boldsymbol{\sigma}} \right\| = \sqrt {{\boldsymbol{\sigma}} \cdot {\boldsymbol{\sigma}}} $. Any tensor from the set $E_s^2$ can  be decomposed as follows: 
\begin{equation}
	{\boldsymbol{\sigma}} = \frac{1}{d}\left( {{\rm{tr}}\,{\boldsymbol{\sigma}}} \right){\bf{I}}\, + {\rm{dev}}{\boldsymbol{\sigma}}
\end{equation}      						
where ${\rm{tr}}\,{\boldsymbol{\sigma}} = \sum_{i=1}^d {\sigma _{ii}}$. Note that for ${\boldsymbol{\sigma}},\;{\boldsymbol{\eps}} \in E_s^2$

\begin{equation}
	\label{eq:dot_decomp}
	{\boldsymbol{\sigma}} \cdot {\boldsymbol{\eps}} = \left( {{\rm{Tr }}\,{\boldsymbol{\sigma}}} \right)\left( {{\rm{Tr }}\,{\boldsymbol{\eps}}} \right) + {\rm{dev}}\;{\boldsymbol{\sigma}} \cdot {\rm{dev}}\,{\boldsymbol{\eps}}
\end{equation}
where we agree that $${\rm{Tr}}\,{\boldsymbol{\sigma}} := \frac{{\rm{tr}}\,{\boldsymbol{\sigma}}}{\sqrt d}. $$

By $L^p(\Omega)$ we will understand the Lebesgue space of $p$-integrable measurable function on $\Omega$;  its vector valued and symmetric tensor valued variants will be denoted by $L^p(\Omega;\R^n) \equiv (L^p(\Omega))^n$ and $L^p(\Omega;E_s^2)$, respectively. Let us recall the H\"{o}lder inequality. Let  ${\bf{f}} = \left( {{f_1},...,{f_n}} \right) $ , ${\bf{g}} = \left( {{g_1},...,{g_n}} \right)$ be any measurable vectorial functions  on the domain \(\Omega \). For any $p \in [1,\infty]$ let $p'$ be the H\"{o}lder conjugate exponent, i.e. the one satisfying 
\begin{equation}
	\label{eq:Holder_conj}
	\frac{1}{p} + \frac{1}{{p'}} = 1
\end{equation}  						
If $p =1$, then by convention $p'=\infty$ and \textit{vice versa}. The H\"{o}lder inequality reads
\begin{equation}
	\label{eq:Holder_basic}
	\int_\Omega  {\sum\limits_{i = 1}^n {\abs{{{f_i}{g_i}}} \,dx} }  \le {\left\| {\bf{f}} \right\|_p}{\left\| {\bf{g}} \right\|_{p'}}.
\end{equation}
Assuming that $p\in (1,\infty)$, the above inequality is an equality if and only if for some numbers $\alpha,\beta \geq 0$, not both zero, we have $\alpha \abs{f_i(x)}^p = \beta \abs{g_i(x)}^{p'}$ for all $i$ and a.e. $x$.

For $s \in [1,\infty]$ the vector valued Sobolev space will be denoted by $W^{1,s}(\Omega;\Rd)$. For any function $\mbf{v} \in W^{1,s}(\Omega;\Rd)$ (virtual displacement function) the strain tensor field $\boldsymbol{\eps}(\mbf{v}):\O \to E_s^2$ will be defined as a symmetrical part of the weak gradient, i.e.
\begin{equation}
    \boldsymbol{\eps}(\mbf{v}) := \frac{1}{2} \big( \nabla \mbf{v} + (\nabla \mbf{v})^\top \big),
\end{equation}
being a function in $L^s(\O;\Rd)$.

We conclude with some basic notions in convex analysis.
For an extended real convex functional $F:X \to \R \cup \{+\infty\}$, where $X$ is a Banach space, by the Fenchel-Legendre transform $F^*$ we understand the functional on the dual space $X^*$ given by the formula
\begin{equation}
	F^*(x^*) := \sup_{x \in X} \Big\{ \pairing{x^*,x} - F(x) \Big\} \qquad \forall \,x^*\in X^*,
\end{equation}
where $\pairing{\argu,\argu}$ is the duality bracket. If $Y$ is another Banach space, then, for a continuous linear operator $\Lambda: X \to Y$, by $\Lambda^*: Y^* \to X^*$ we understand the adjoint operator that is defined through the equality $\pairing{x,\Lambda^* y^*} :=\pairing{\Lambda x, y^*}$. Finally, if $\opnorm{\argu}$ is a norm on $X$, by $\opnorm{\argu}_*$ we denote the dual norm:
\begin{equation}
	\opnorm{x^*}_* = \sup_{x \in X} \Big\{ \pairing{x,x^*} \ \Big| \ \opnorm{x} \leq 1 \Big\}.
\end{equation}

\bigskip

\section{Useful analytical results}

Minimization over the design variables will be performed analytically with using the minimization results given in the present section. Assume that $a_i(x) \geq 0$, $i=1,\ldots,n$, are given non-negative measurable functions on $\Omega$. We will be looking for   non-negative  functions ${\bf{u}} = \left( {{u_1},...,{u_n}} \right)$ on \(\Omega \), i.e. ${u_i}(x) \geq 0, \ i = 1,...,n$, which minimize the convex functional
\begin{equation}
	\label{eq:fun_au}
	\int_\Omega  {\sum\limits_{i = 1}^n {\frac{{{a_i}( x )}}{{{u_i}( x )}} \, dx} }  
\end{equation}
under a constraint on an $L^p$-norm of $\mbf{u}$ for $p \in [1,\infty)$. In the two subsections two variants of these norms will be considered, each suited to different variants of the cost condition defining the methods vp-IMD or sp-IMD.

Beforehand, we should make sense of the ratio $\frac{{{a_i}( x )}}{{{u_i}( x )}}$ when one or both functions are zero:
\begin{equation}
	\label{eq:quotient_convention}
	\frac{{{a_i}( x )}}{{{u_i}( x )}} := \begin{cases}
		\frac{{{a_i}( x )}}{{{u_i}( x )}} & \text{if } u_i(x) >0, \\
		0 & \text{if } u_i(x) =0, \ a_i(x) =0,\\
		+\infty & \text{if } u_i(x) =0, \ a_i(x) >0.
	\end{cases}
\end{equation}
Naturally,  for the functional \eqref{eq:fun_au} to be finite, the last case cannot hold for a subset of $\Omega$ with positive measure.

\subsection{An auxiliary minimization result for the vp-IMD method}

In this subsection we solve a problem of minimizing with respect to  non-negative  functions ${\bf{u}} = \left( {{u_1},...,{u_n}} \right)$ on \(\Omega \), i.e. ${u_i}(x) \geq 0$, $i = 1,...,n$, which satisfy
\begin{equation}
	\label{eq:up_Lambda}
	\norm{\mbf{u}}_p \leq \Lambda.
\end{equation}
In particular, $\mbf{u}$ must be an element of $(L^p(\O))^n$. Note that  in case of $p = 1$ the result below has been given and proved in \cite{lewinski2021}. 

\begin{lemma}
	\label{lem:vp}
	For $p\in [1,\infty)$ and a  vector $\mbf{a}=\mbf{a}(x)$ of non-negative measurable functions $a_i(x) \geq 0$, $i=1,\ldots,n$, the following variational equality holds true:
	\begin{equation}
		\label{eq:analytical_I}
		\mathop {\min }_{\substack{{\bf{u}} \geq \mbf{0} \\ \norm{\mbf{u}}_p \leq \Lambda}} \int_\Omega  {\sum\limits_{i = 1}^n {\frac{{{a_i}\left( x \right)}}{{{u_i}\left( x \right)}}dx} }  = \frac{1}{\Lambda }  \norm{\sqrt{\mbf{a}}}_r^2,
	\end{equation} 
	where $\sqrt{\mbf{a}}(x) = \big(\sqrt{a_1(x)}, \ldots , \sqrt{a_n(x)} \big)$, and $$r= \frac{2p}{p+1}  \in (1,2).$$
	
	Moreover, assuming that $\sqrt{a_i} \in L^r(\Omega)$ for every $i$ and that at least one function $a_i$ is not a.e. zero, the problem \eqref{eq:analytical_I} admits a unique minimizer
	\begin{equation}
		\label{eq:uopt_I}
		\hat{u}_i ( x ) = \Lambda \left(\frac{\sqrt{a_i(x)}}{\norm{\sqrt{\mbf{a}}}_r} \right)^{2-r}
	\end{equation}
	which satisfies the condition \eqref{eq:up_Lambda} sharply.
\end{lemma} 
\begin{proof}
	Let us define the functions
	\begin{equation}
		{f_i} = {\left( {\frac{{{a_i}}}{{{u_i}}}} \right)^\frac{r}{2}}, \qquad {g_i} = {\left( {{u_i}} \right)^{\frac{r}{2}}}.
	\end{equation} 
	Since $u_i$ is to be a competitor in \eqref{eq:analytical_I}, it is not restrictive to assume that $\frac{{{a_i}( x )}}{{{u_i}( x )}} < +\infty$ for a.e. $x$. Thus,  H\"{o}lder inequality \eqref{eq:Holder_basic} can be applied, and we choose the exponents $s = \frac{2}{r} = \frac{p+1}{p} > 1$ and $s' = \frac{2}{2-r} =  p+1$. We find that:
	\begin{equation}
		{f_i}{g_i} = (\sqrt{a_i})^r,  \qquad  {\left( {{f_i}} \right)^{s}} = \frac{a_i}{u_i} , \qquad  {\left( {{g_i}} \right)^{s'}} = {\left( {{u_i}} \right)^{p}}.
	\end{equation}  
	Then,
	\begin{equation}
		{\left\| {\bf{f}} \right\|_s} = {\left( {\int_\Omega  {\sum\limits_{i = 1}^n {\frac{{{a_i}}}{{{u_i}}}dx} } } \right)^{\frac{r}{2}}},  \qquad
		{\left\| {\bf{g}} \right\|_{s'}} = {\left( {\int_\Omega  {\sum\limits_{i = 1}^n {{{\left( {{u_i}} \right)}^p}dx} } } \right)^\frac{1}{p+1}} =\norm{\mbf{u}}_p^{p/(p+1)}.
	\end{equation}
	and, after raising the H\"{o}lder inequality $	\int_\Omega  {\sum_{i = 1}^n {\left| {{f_i}{g_i}} \right|dx} }  \le {\left\| {\bf{f}} \right\|_s}{\left\| {\bf{g}} \right\|_{s'}}$ to the power $s =\frac{2}{r} = \frac{p+1}{p}$ on both sides, we get the inequalities
	\begin{equation}
		 {\left( {\int_\Omega  {\sum_{i = 1}^n {{{\left( {\sqrt{a_i}\left( x \right)} \right)}^{r}}} dx} } \right)^{2/r}}  \leq  \norm{\mbf{u}}_p  {\int_\Omega  {\sum\limits_{i = 1}^n {\frac{{{a_i}}}{{{u_i}}}dx} } } \leq \Lambda\, {\int_\Omega  {\sum\limits_{i = 1}^n {\frac{{{a_i}}}{{{u_i}}}dx} } }.
	\end{equation}
	Assuming that $\mbf{a}$ is not identically zero, the first inequality is an equality if and only if for some $\alpha> 0 $ there holds $g_i^{s'} =\alpha f_i^s$ or, equivalently, $u_i = \alpha (\sqrt{a_i})^{2/(p+1)}$. The second inequality is an equality if and only if $\norm{\mbf{u}}_p = \Lambda$ which determines the constant $\alpha$, and ultimately leads to the formula \eqref{eq:uopt_I}.
\end{proof}

\subsection{An auxiliary minimization result for the sp-IMD method}

Let  $p > 1$ and let  ${a_i}(x) > 0,{\rm{ }}i = 1,...,n$ be given on the domain  \(\Omega \). This time we are looking for the vector functions  ${\bf{u}} = \left( {{u_1},...,{u_n}} \right)$ on \(\Omega \)  such that  ${u_i}(x) > 0$, $i = 1,...,n$ and
\begin{equation}
	\label{eq:sumup_Lambda}
	{\Big\| {\sum_{i = 1}^n {{u_i}} } \Big\|_p} \le \Lambda 
\end{equation}

\begin{lemma}
	\label{lem:sp}
	For $p\in [1,\infty)$ and a  vector  $\mbf{a}=\mbf{a}(x)$ of non-negative measurable functions $a_i(x) \geq 0$, $i=1,\ldots,n$, the following variational equality holds true:
	\begin{equation}
		\label{eq:analytical_II}
		\mathop {\min }_{\substack{{\bf{u}} \geq \mbf{0} \\ \norm{\sum_i u_i}_p \leq \Lambda}} \int_\Omega  {\sum_{i = 1}^n {\frac{{{a_i}\left( x \right)}}{{{u_i}\left( x \right)}}dx} }  = \frac{1}{\Lambda }  {\Big\| {\sum_{i = 1}^n {\sqrt {{a_i}} } } \Big\|_{r}^2}
	\end{equation} 
	where again $r= \frac{2p}{p+1}$.
	Moreover, assuming that $\sqrt{a_i} \in L^r(\Omega)$ for every $i$ and that at least one function $a_i$ is not a.e. zero, the problem \eqref{eq:analytical_I} admits a unique minimizer
	\begin{equation}
		\label{eq:uopt_II}
		\hat{u}_i ( x ) =  \frac{\Lambda }{{{{\left\| {\sum\nolimits_{i = 1}^n {\sqrt {{a_i}} } } \right\|_{r}^{2-r}}}}}\frac{{\sqrt {{a_i}\left( x \right)} }}{{{{\Big( {\sum_{i = 1}^n {\sqrt{{a_i}\left( x \right)} } } \Big)}^{r-1}}}}
	\end{equation}
	which satisfies the condition \eqref{eq:sumup_Lambda} sharply.
\end{lemma} 
\begin{proof}
	The proof will be similar in spirit, yet instead of using H\"{o}lder inequality for vectorial functions, we will combine H\"{o}lder for scalar functions, and Schwarz inequality (H\"{o}lder for $p=2$) in $\R^n$:
	\begin{align}
		 &{\int_\Omega \Big(  {\sum_{i = 1}^n {{{{\sqrt{a_i(x)}} }}}} \Big)^r dx }  = \int_\Omega \Big( \sum_{i = 1}^n \sqrt{\tfrac{a_i(x)}{u_i(x)}} \sqrt{u_i}(x)  \Big)^r dx \leq \int_\Omega \bigg(\sum_{i = 1}^n \frac{a_i(x)}{u_i(x)}\bigg)^{\frac{r}{2}} \bigg(\sum_{i = 1}^n u_i(x)\bigg)^{\frac{r}{2}} dx \\
		 & \qquad  \leq \bigg( \int_\Omega \sum_{i = 1}^n \frac{a_i(x)}{u_i(x)}\, dx\bigg)^{\frac{r}{2}}  \bigg( \int_\Omega \Big(\sum_{i = 1}^n u_i(x) \Big)^p\, dx\bigg)^{\frac{1}{p+1}} \leq \bigg( \int_\Omega \sum_{i = 1}^n \frac{a_i(x)}{u_i(x)}\, dx\bigg)^{\frac{r}{2}}  \Lambda^{\frac{p}{p+1}}.
	\end{align}
	The first inequality is Schwarz inequality in $\R^n$. The second one is
	 H\"{o}lder inequality for scalar functions on $\O$ written for
	 exponents $s = \frac{2}{r} = \frac{p+1}{p} > 1$ and $s' = \frac{2}{2-r} =  p+1$; note that $s' r/2 = p$. Finally, the third inequality is simply the constraint \eqref{eq:sumup_Lambda}.
	 
	After raising this chain of inequalities to the power $\frac{2}{r} = \frac{p+1}{p}$ we arrive at the inequality $\leq$ in the place of equality \eqref{eq:analytical_II}. We can easily check that $u_i = \hat{u}_i$ satisfies it as equality, and that ${\left\| {\sum\nolimits_{i = 1}^n {{\hat{u}_i}} } \right\|_p} = \Lambda$. This approves of  \eqref{eq:analytical_II} and of optimality of $\hat{u}_i$. The uniqueness of $\hat{u}_i$ can be showed by studying the conditions under which the Schwarz and the H\"{o}lder inequality become equalities.
\end{proof}

\section{The problem of compliance minimization and its stress based reformulation}

\subsection{Preliminaries on linear elasticity}

Let us fix a real number $p \in (1,\infty)$, with $p=1$ and $p=\infty$ excluded. Apart from $p$, the crucial role will be played by the  exponent
\begin{equation}
	r = \frac{2p}{p+1}
\end{equation}
and its H\"{o}lder conjugate $r' = \frac{r}{r-1} = \frac{2p}{p-1}$. Note that $r \in (1,2)$, whilst $r'$ ranges between 2 and $+\infty$.

For a bounded domain $\Omega \subset \R^d$ of Lipschitz regular boundary $\bO$, consider an elastic isotropic body of moduli  $k,\mu$ being non-negative functions in $L^p(\Omega)$. 
The body is clamped on a subset $\Gamma_0$ of the boundary $\bO$; the set $\Gamma_0$ is assumed to be of non-zero $d-1$ dimensional measure. The definitions of the functional spaces to be given below are dictated by the mathematical framework of the herein considered optimization problem, rather than by the framework of the classical linear elasticity. Accordingly, by the space of virtual displacements we will understand
\begin{equation}
	\mbf{V}^{r'}\!(\Omega) := \Big\{ \mbf{v} \in  W^{1,r'}\!(\Omega;\Rd) \ \Big| \ \mbf{v} = 0 \ \ \text{$ds$-a.e. on  $\Gamma_0$}  \Big\}
\end{equation}
where $ds$-a.e. means almost everywhere with respect to the $d-1$ dimensional measure on the boundary $\bO$. $\mbf{V}^{r'}\!(\Omega)$ is a closed subspace of  $W^{1,r'}\!(\Omega;\Rd)$, and thus a reflexive Banach space, see Theorem 6.3-4 in \cite{ciarlet2000a} (the proof easily extends to any exponent larger than two). It is worth noting that the functions $\mbf{v} \in \mbf{V}^{r'} \!(\Omega)$ exhibit higher regularity than generic functions in $H^1(\Omega;\Rd)$ owing to $r' >2$. In particular, due to the Morrey embedding theorem, $\mbf{v}$ is continuous in dimension $d=2$ for any $p \in (1,\infty)$ and in dimension $d=3$ provided that $p <3$ (note that then $r' > 3$).

The load  will be represented by the virtual work functional ${f}$ being an element in the  dual space of $W^{1,r'}(\Omega;\Rd)$. In the view of the foregoing remarks we can assert that:
\begin{enumerate}
	\item [a)] in the case $d=3$, $p \geq 3$ for the load one can take, for instance, the functionals of the form
	\begin{equation}
		{f}(\mbf{v}) = \int_\Omega \mbf{q}(x) \cdot\mbf{v}(x) \, dx + \int_{\bO \backslash \Gamma_0} \mbf{t}(x) \cdot \mbf{v}(x)\, ds
	\end{equation}
	where $\mbf{q} \in L^r(\Omega;\Rd)$ and $\mbf{t} \in L^r(\bO;\Rd)$ are the intensities of the body forces and, respectively, the tractions on the boundary;
	\item[b)] whenever $d=2$, $p \in (1,\infty)$ or $d=3$, $p <3$, we can add the point forces $\mbf{F}_i \in \Rd$; namely,
		\begin{equation}
		{f}(\mbf{v})= \int_\Omega \mbf{q}(x) \cdot\mbf{v}(x) \, dx + \int_{\bO \backslash \Gamma_0} \mbf{t}(x) \cdot \mbf{v}(x)\, ds + \sum_{i=1}^n \mbf{F}_i \cdot \mbf{v}(x_i),
	\end{equation}
	where $x_i \in \Ob$.
\end{enumerate}
The above examples of loads ${f}$ do not saturate all the possibilities, see \cite[Theorem 3.9]{adams2003} for the full characterization of the space $\big(W^{1,r'}(\Omega;\Rd)\big)^*$.

By the virtual stress fields we will understand functions ${\boldsymbol{\tau}} = ({\tau _{ij}})$ in the Lebesgue space $L^{r}(\Omega;E_s^2)$ which equilibrate the load ${f}$ in the weak sense, i.e.
\begin{equation}
		\label{eq:stat_adm}
		\int_\Omega  {{\boldsymbol{\tau}} \cdot \,} {\boldsymbol{\eps}}\left( {\bf{v}} \right)dx  = {f}(\mbf{v}) \qquad \forall \,  \mbf{v} \in \mbf{V}^{r'}\!(\Omega),
\end{equation}
where ${\boldsymbol{\eps}}\left( {\bf{v}} \right) = \left( {{\varepsilon _{ij}}\left( {\bf{v}} \right)} \right)$, $i,j = 1,..,d$, see Section \ref{ssec:notation}. The set of statically admissible  stress fields $\boldsymbol{\tau}\in L^{r}(\Omega;E_s^2)$ satisfying \eqref{eq:stat_adm}  is denoted by $\Sigma_{{f}} ( \Omega )$. It is straightforward to show that $\Sigma_{{f}} ( \Omega )$ is weakly closed in $L^r(\Omega;E_s^2)$.
Under the assumption that $\Gamma_0$ is of non-zero $d-1$ dimensional measure, the set  $\Sigma_{{f}} ( \Omega )$ is non-empty. This will be demonstrated via duality arguments in Section \ref{sec:duality}. 

We say that displacements $\mbf{u} \in \mbf{V}^{r'}\!(\Omega)$ and stresses $\boldsymbol{\sigma} \in \Sigma_{{f}} ( \Omega )$ solve the linear elasticity problem for the body of the elastic moduli $k,\mu \in L^p(\Omega;\R_+)$ if the constitutive law of isotropy is met:
\begin{equation}
	\label{eq:const}
	{\rm{Tr}} \,{\boldsymbol{\sigma}} = dk \,{\rm{Tr }}\,{\boldsymbol{\eps}}({\bf{u}}), \qquad {\rm{dev}}\,{\boldsymbol{\sigma}} = 2\mu\, {\rm{dev}}\,{\boldsymbol{\eps}}({\bf{u}}) \qquad \text{a.e. in $\O$}.
\end{equation}
The compliance of such a body is classically understood by the virtual work of the load, i.e. $\mathcal{C}(k,\mu) = {f}(\mbf{u})$. However, for general $k,\mu \in L^p(\Omega;\R_+)$, which may not be uniformly bounded from below and above,  there is no guarantee that solution $\mbf{u}$ exists in $W^{1,r'}\!(\O;\Rd)$ or in any other reasonable functional space.

To by-pass this issue, an equivalent definition of compliance through elastic energy can be employed.  By using the formulae \eqref{eq:const} the density of energy reads $W = \frac{1}{2}\boldsymbol\sigma\cdot \boldsymbol{\eps}(\mbf{u})$ can be expressed in stresses only, i.e. $	W({\boldsymbol{\sigma}} ) = \frac{{\rm{1}}}{{\rm{2}}}\big( {\frac{1}{{dk}}{{\left( {{\rm{Tr }}{\boldsymbol{\sigma}}} \right)}^2} + \frac{1}{{2\mu }}{{\left\| {{\rm{dev}}\;{\boldsymbol{\sigma}}} \right\|}^2}} \big)$. The Castigliano theorem states that the unknown stress $\boldsymbol\sigma$  is the minimizer of the functional $\Sigma_{{f}}(\O) \ni \boldsymbol{\tau} \mapsto \int_\O W(\boldsymbol{\tau})\, dx$ and the minimum  equals half the compliance, see \cite{duvant1976}. Ultimately, the stress-based formula for the compliance can be written down:
\begin{equation}
	\label{eq:comp}
	\mathcal{C}(k,\mu ) := \inf_{{\boldsymbol{\tau}} \in \Sigma_{{f}} \left( \Omega  \right)} \int_\Omega  {\left( {\frac{1}{{dk}}{{\left( {{\rm{Tr }}{ \, \boldsymbol{\tau}}} \right)}^2} + \frac{1}{{2\mu }}{{\left\| {{\rm{dev}}\,{\boldsymbol{\tau}}} \right\|}^2}} \right)dx}.
\end{equation} 
This formula is meaningful for any functions $k,\mu \in L^p(\Omega;\R_+)$. Note that it allows singular $k,\mu$, e.g. that are zero on subdomains of $\Omega$, see the convention \eqref{eq:quotient_convention}. It may result in an infinite value $\mathcal{C}(k,\mu ) = + \infty$, but these scenarios will be naturally avoided by optimization. The next result warrants the well posedness of the optimization problem stated in the subsequent sections:  
\begin{proposition}
	\label{prop:functional}
	The compliance functional $\mathcal{C}: (L^p(\Omega;\R_+))^2\to \R_+ \cup \{+\infty\}$ is convex and weakly lower semi-continuous.
\end{proposition}
\begin{proof}
	By the classical duality \cite{ekeland1999} the displacement-based definition of compliance can be derived: 
	\begin{equation}
        \label{eq:comp_disp}
		\mathcal{C}(k,\mu ) = \sup\limits_{\mbf{v} \in \mbf{V}^{r'} \!(\Omega)} \left\{ 2{f}(\mbf{v}) - \int_\Omega  {\left( {{{dk}}{{\big( {{\rm{Tr }}{\, \boldsymbol{\eps}(\mbf{v})}} \big)}^2} + {{2\mu }}{{\left\| {{\rm{dev}}\,{\boldsymbol{\eps}(\mbf{v})}} \right\|}^2}} \right)dx} \right\}.
	\end{equation} 
	Since $\mbf{v} \in W^{1,r'}\!(\Omega;\Rd)$, the function $\boldsymbol{\eps}(\mbf{v})$ is $r' = \frac{2p}{p-1} = 2p'$ -integrable. As a result, for  any fixed $\mbf{v}$, the mapping $(k,\mu) \mapsto 2{f}(\mbf{v})- \int_\Omega  {\big( {{{dk}}{{\big( {{\rm{Tr }}{\, \boldsymbol{\eps}(\mbf{v})}} \big)}^2} + {{2\mu }}{{\left\| {{\rm{dev}}\,{\boldsymbol{\eps}(\mbf{v})}} \right\|}^2}} \big)dx}$ is a linear continuous functional on $(L^p(\Omega;\R))^2$. In particular, it is convex and weakly lower semi-continuous, and the latter properties are preserved under taking the supremum with respect to $\mbf{v}$.
\end{proof}

\subsection{The vp-IMD method}

The vp-IMD variant of the compliance minimization method herein developed entails the admissible set of moduli that satisfy the cost condition \eqref{eq:vp_cost}, that is:
\begin{equation}
	{M_{vp}}(\Omega ) := \Big\{ (k,\mu) \in (L^p(\Omega;\R_+))^2 \ \Big\vert \ \| (dk,\beta^{\frac{2}{p}}2\mu) \|_p \leq \Lambda \Big\},
\end{equation}
recall that $p \in (1,\infty)$. Note that ${M_{vp}}(\Omega)$ is the intersection of the half-space $(L^p(\Omega;\R_+))^2$ and $B_{L^p}(0,\Lambda)$ being the standard (up to the scaling parameters $d, \beta$) closed ball in the Banach space $L^p(\Omega;\R^2)$ of radius $\Lambda$.

The isotropic material design problem in its vectorial variant can be readily posed:
\begin{equation}
	\label{eq:vp-IMD}
	(\mathrm{\text{vp-IMD}}) \qquad \qquad\qquad  \qquad\qquad  \hat{\mathcal{C}}_{vp} =  \mathop {\min }\limits_{(k,\mu ) \in {M_{vp}}\left( \Omega  \right)} \mathcal{C}(k,\mu ). \qquad \qquad  \qquad\qquad  \qquad\qquad 
\end{equation}
It is a convex optimization problem that is well posed: solution $(\hat{k},\hat{\mu}) \in M_{vp}(\Omega)$ exists, and it is unique. Existence follows directly by virtue of the weak lower semi-continuity of $\mathcal{C}$ established in Proposition \ref{prop:functional} and  by the weak compactness of ${M_{vp}}(\Omega)$, which  for $p \in (1,\infty)$ is due to the reflexivity of $L^p$. To show uniqueness, we observe that $\mathcal{C}(t k, t\mu) = \frac{1}{t} \, \mathcal{C}(k, \mu)$ for $t>0$, see  the definition \eqref{eq:comp}. Therefore, assuming that $\hat{\mathcal{C}}_{vp} \neq 0$, any solution $(\hat{k},\hat{\mu})$ must lie on the sphere $\partial B_{L^p}(0,\Lambda)$. Since the ball $B_{L^p}(0,\Lambda)$ is a strictly convex set, the uniqueness can be thus deduced from the convexity of $\mathcal{C}$.

Observe that, since the reflexivity of $L^1$ is lacking, the existence of solutions in (vp-IMD) fails to hold in general for $p=1$, which calls for a measure-theoretic relaxation mentioned in the introduction, see \cite{bolbotowski2022b} and Section \ref{ssec:IMD_recovery} below in this paper.

Let us now proceed to reformulating the problem to its stress-based form. Incorporating the formula \eqref{eq:comp} for the compliance $\mathcal{C}(k,\mu)$ into the problem (vp-IMD) we arrive at a double minimization problem.  Since the set $\Sigma_{{f}} \left( \Omega  \right)$ does not depend on $(k,\mu)\in {M_{vp}}(\Omega)$, the order of the minimization may be swapped; hence,
\begin{equation}
	\label{eq:min_swap}
	\hat{\mathcal{C}}_{vp} = \inf_{(k,\mu ) \in M_{vp}(\Omega)} \inf_{{\boldsymbol{\tau}} \in \Sigma_{{f}} \left( \Omega  \right)} \int_\Omega  {\left( {\frac{1}{{dk}}{{\left( {{\rm{Tr }}{ \, \boldsymbol{\tau}}} \right)}^2} + \frac{1}{{2\mu }}{{\left\| {{\rm{dev}}\,{\boldsymbol{\tau}}} \right\|}^2}} \right)dx} = \mathop {\inf }\limits_{{\boldsymbol{\tau}} \in \Sigma_{{f}} \left( \Omega  \right)} {F_{vp}}\left( {\boldsymbol{\tau}} \right),
\end{equation}   					
where
\begin{equation}
	\label{eq:Fp}
	{F_{vp}}({\boldsymbol{\tau}}) := \inf \limits_{(k,\mu ) \in {M_{vp}}\left( \Omega  \right)} \int_\Omega  {\left( {\frac{1}{{dk}}{{\left( {{\rm{Tr }}\,{\boldsymbol{\tau}}} \right)}^2} + \frac{1}{{2\mu }}{{\left\| {{\rm{dev}}\,{\boldsymbol{\tau}}} \right\|}^2}} \right)dx}.
\end{equation}

Assuming that $\boldsymbol{\tau} \in \Sigma_f(\Omega) \subset  L^r(\Omega;E_s^2)$ is fixed, Lemma \ref{lem:vp} can be now employed to explicitly solve the problem  defining $F_{vp}(\boldsymbol{\tau})$.
To fit into the framework of the lemma, for $n=2$  let us introduce the non-negative  functions 
\begin{equation}
	\label{eq:a1a2}
	{a_1} = {\left( {{\rm{Tr }}\,{\boldsymbol{\tau}}} \right)^2}, \qquad {a_2} = \beta^{2/p}{\left\| {{\rm{dev}}\;{\boldsymbol{\tau}}} \right\|^2}
\end{equation}
and perform the change of variables:
\begin{equation}
	{u_1} = dk{\rm{ }}, \qquad {u_2} = 2\beta^{2/p}\mu.
\end{equation}
Since $\boldsymbol{\tau} \in L^r(\Omega;E_s^2)$, it is clear that $\sqrt{a_i} \in  L^r(\Omega)$. Applying Lemma \ref{lem:vp} directly yields 
\begin{equation}
	\label{eq:Fvp_formula}
	{F_{vp}}\left( {\boldsymbol{\tau}} \right) = \frac{1}{\Lambda }  \norm{\sqrt{\mbf{a}}}_r^2= \frac{1}{\Lambda }{\Big(r 	{\widetilde{F}_{vp}}\left( {\boldsymbol{\tau}} \right) \Big)^{\frac{2}{r}}}, 
\end{equation}
where we defined:
\begin{equation}
{\widetilde{F}_{vp}}\left( {\boldsymbol{\tau}} \right) := \frac{1}{r} \int_\Omega  {\Big( {{{\left| {{\rm{Tr }}{\boldsymbol{\tau}}} \right|}^{r}} +\beta^{2-r} {{\left\| {{\rm{dev}}\,{\boldsymbol{\tau}}} \right\|}^{r}}} \Big)dx}
\end{equation}
(the factor $\frac{1}{r}$ is introduced for the sake of the forthcoming duality arguments).
Moreover, assuming that $\boldsymbol{\tau} \neq 0$, the moduli 
$(k,\mu)$ solve \eqref{eq:Fp} if and only if
\begin{equation}
	\label{eq:kmu_tau_vp}
	d{{k} } = \frac{\Lambda }{{{{\big( r{{{\widetilde{F}_{vp}}\left( {\boldsymbol{\tau}} \right) }} \big)}^{{1}/{{p}}}}}}{\left| {{\rm{Tr }}\,{{\boldsymbol{\tau}} }} \right|^{\frac{2}{p+1}}}, \qquad 2\beta^{\frac{2}{p}}{{\mu}} =  \frac{\Lambda }{{{{\big( r{{{\widetilde{F}_{vp}}\left( {\boldsymbol{\tau}} \right) }} \big)}^{{1}/{{p}}}}}}{\left\| {{\beta^{\frac{1}{p}}}{\rm{dev}}\,{{\boldsymbol{\tau}}}} \right\|^{\frac{2}{p+1}}} \qquad \text{a.e. in $\Omega$}.
\end{equation}	  
Note that, several times  above, we have made manipulations using the relation between $p$ and $r$.

Going back to the chain of equalities \eqref{eq:min_swap}, the formula for minimal compliance can be recast by exploiting \eqref{eq:Fvp_formula}:
\begin{equation}
	\label{eq:Cvp}
	\hat{\mathcal{C}}_{vp} = \frac{1}{\Lambda }{\left( {{r Z_r}} \right)^{\frac{2}{r}}},
\end{equation}
where 
\begin{align}
	\label{eq:Zr}
	{Z_r} := \mathop {\min }\limits_{{\boldsymbol{\tau}} \in \Sigma_{{f}} \left( \Omega  \right)} \frac{1}{r} \int_\Omega  {\Big( {{{\left| {{\rm{Tr }}\,{\boldsymbol{\tau}}} \right|}^{r}} + \beta^{2-r}{{\left\| {{\rm{dev}}\,{\boldsymbol{\tau}}} \right\|}^{r}}} \Big)dx}.
\end{align}	
Note that $({\widetilde{F}_{vp}}\left( \argu \right))^{1/r}$ is a norm on $L^r(\Omega;E_s^2)$ that is equivalent to its standard norm. Moreover, it can be checked that ${\widetilde{F}_{vp}}\left( \argu \right)$ is strictly convex. Since $\Sigma_{{f}} \left( \Omega  \right)$ is a closed affine subspace of $L^r(\Omega;E_s^2)$,  the existence in \eqref{eq:Zr} follows together with uniqueness.

The problem \eqref{eq:Zr} may be deemed the stress-based formulation of the problem (vp-IMD). Indeed, we have showed that the formulae \eqref{eq:kmu_tau_vp} constitute a one-to-one  correspondence between the unique solution $(\hat{k},\hat{\mu})$ of the problem \eqref{eq:vp-IMD} and the unique solution $\hat{\boldsymbol\tau}$ of the problem \eqref{eq:Zr}. The results of this subsection are summarized below:

\begin{theorem}
	\label{thm:stress_vp}
	Let $\hat{\boldsymbol{\tau}} \in \Sigma_f(\Omega)$ be the unique solution of the stress-based problem \eqref{eq:Zr}, which always exists. Then, a pair $(\hat{k},\hat{\mu})\in M_{vp}(\Omega)$ is the unique solution of the (vp-IMD) problem \eqref{eq:vp-IMD} if and only if
	\begin{equation}
	\label{eq:opt_kmu_tau_vp}
	d{\hat{k}} = \frac{\Lambda }{{{{\left( {{r Z_r}} \right)}^{\frac{1}{{p}}}}}}{\left| {{\rm{Tr }}\,{\hat{\boldsymbol{\tau}} }} \right|^{\frac{2}{p+1}}}, \qquad 2{\hat\mu} = \frac{\Lambda }{{{{\left( {{r Z_r}} \right)}^{\frac{1}{{p}}}}}}{\Big\| { \frac{1}{\beta} \,{\rm{dev}}\,{\hat{\boldsymbol{\tau}} }} \Big\|^{\frac{2}{p+1}}} \qquad \text{a.e. in $\Omega$}.
\end{equation}	 
\end{theorem}

\bigskip

The optimal moduli ${\hat{k}},{\hat{\mu}}$   given by (3.14)   determine the optimal Young modulus $\hat{E}$  and the optimal Poisson ratio $\hat{\nu}$ according to the rules (see Sec. 7.2.4 in \cite{Lewinski2019}):

a)	Case of  $d=2$:
\begin{equation}
	\label{eq:Enu2D}
	\hat{E} = 2{\left( {\frac{1}{{2\hat{k}}} + \frac{1}{{2\hat{\mu}}}} \right)^{ - 1}}, \qquad \hat{\nu} = \frac{{\hat{k} - \hat{\mu}}}{{\hat{k} + \hat{\mu}}};
\end{equation}  

b)	Case of  $d=3$:
\begin{equation}
	\label{eq:Enu3D}
	\hat{E} = 3{\left( {\frac{1}{{3\hat{k}}} + \frac{1}{{2\hat{\mu}}} + \frac{1}{{2\hat{\mu}}}} \right)^{ - 1}}, \qquad \hat{\nu} = \frac{{3\hat{k} - 2\hat{\mu}}}{{2(3\hat{k} + \hat{\mu})}}.
\end{equation}	
Let us note that the Young modulus vanishes if one of the moduli  ${\hat{k}},{\hat\mu}$ vanishes, since this modulus is proportional to the harmonic mean of $2\hat{k},2\hat{\mu}$ ($d=2$) or  $3\hat{k},2\hat{\mu },2\hat{\mu}$ ($d=3$).
The conditions $\hat{k} \ge 0, \hat{\mu} \ge 0$ introduce the known bounds on the Poisson ratio $\hat{\nu}$:   in case of $d=2$ we have: $ - 1 \le \hat{\nu} \le 1$, while in case of $d=3$ the range is narrower: $ - 1 \le \hat{\nu} \le 1/2$. Thus, we expect that  the optimal material representing the optimum design  will be partially auxetic, i.e. in certain subdomains the Poisson ratio will assume negative values; the corresponding microstructures assume then chiral shapes, see \cite{czarnecki2024}.

\bigskip

\subsection{The sp-IMD method}

In the sp-IMD variant of the optimization method the set of admissible moduli distributions involves the cost condition \eqref{eq:sp_cost}:
\begin{equation}
	{M_{sp}}(\Omega ) := \Big\{ (k,\mu) \in (L^p(\Omega;\R_+))^2 \ \Big\vert \ \|d k+ \beta^2 2\mu \|_p \leq \Lambda \Big\},
\end{equation}
and the associated minimum compliance problem follows:
\begin{equation}
	\label{eq:sp-IMD}
	(\mathrm{\text{sp-IMD}}) \qquad \qquad\qquad  \qquad\qquad  \hat{\mathcal{C}}_{sp} =  \mathop {\min }\limits_{(k,\mu ) \in {M_{sp}}\left( \Omega  \right)} \mathcal{C}(k,\mu ). \qquad \qquad  \qquad\qquad  \qquad\qquad 
\end{equation}
Note that, for $(k,\mu) \in L^p(\Omega;\R^2)$ (the functions are not necessarily non-negative), $\big\|d\abs{k}+ \beta^2 2\abs{\mu} \big\|_p$  is a norm that induces the ball $\widetilde{B}_{L^p}(0,\Lambda)$. This new norm is equivalent to the norm $\| (dk,\beta^{\frac{2}{p}}2\mu) \|_p$ that defines the set $M_{vp}(\Omega)$. Therefore, the compactness arguments from the previous subsection can be reproduced, and existence of solutions of the problem (sp-IMD) follows. However, it can be easily checked that the ball $\widetilde{B}_{L^p}(0,\Lambda)$ is not a strictly convex set. As a result, the uniqueness in (sp-IMD) cannot be claimed.

The passage to the stress-based form of the (sp-IMD) problem can be carried out similarly. Swapping the order of minimization leads to
\begin{equation}
	\hat{\mathcal{C}}_{sp} = \mathop {\inf }\limits_{{\boldsymbol{\tau}} \in \Sigma_{{f}} \left( \Omega  \right)} {F_{sp}}\left( {\boldsymbol{\tau}} \right),
\end{equation}   					
where
\begin{equation}
	\label{eq:Fsp}
	{F_{sp}}({\boldsymbol{\tau}}) := \inf \limits_{(k,\mu ) \in {M_{sp}}\left( \Omega  \right)} \int_\Omega  {\left( {\frac{1}{{dk}}{{\left( {{\rm{Tr }}\,{\boldsymbol{\tau}}} \right)}^2} + \frac{1}{{2\mu }}{{\left\| {{\rm{dev}}\,{\boldsymbol{\tau}}} \right\|}^2}} \right)dx}.
\end{equation}
The next step is solving the above problem analytically, this time by employing Lemma \ref{lem:sp} for  ${a_1} = {\left( {{\rm{Tr }}{\boldsymbol{\tau}}} \right)^2}, \  {a_2} = \beta^{2}{\left\| {{\rm{dev}}\;{\boldsymbol{\tau}}} \right\|^2}$ and ${u_1} = dk, \ {u_2} = 2\beta^{2}\mu$. We find that
\begin{equation}
	\label{eq:Fsp_formula}
	{F_{sp}}\left( {\boldsymbol{\tau}} \right) = \frac{1}{\Lambda }  \Big\|\sum_{i=1}^2\sqrt{{a_i}} \Big\|_r^2= \frac{1}{\Lambda }{\Big(r	{\widetilde{F}_{sp}}( {\boldsymbol{\tau}} ) \Big)^{\frac{2}{r}}}, 
\end{equation} 
where
\begin{equation}
		\widetilde{F}_{sp}( \boldsymbol{\tau}) :=\frac{1}{r} \int_\Omega  {{{\Big( {{{\left| {{\rm{Tr }}\,{{\boldsymbol{\tau}}}} \right|}^{}} + \beta {{\left\| {{\rm{dev}}\,{{\boldsymbol{\tau}} }} \right\|}^{}}} \Big)}^{r}}dx}. 
\end{equation}
For non-zero $\boldsymbol{\tau}$, the moduli $(k,\mu)$ solve \eqref{eq:Fsp} if and only if a.e. in $\Omega$ the equalities below hold true:
\begin{equation*}
	d{k } = \frac{\Lambda }{{{{\big( r{{{\widetilde{F}_{vp}}\left( {\boldsymbol{\tau}} \right) }} \big)}^{\frac{1}{{p}}}}}}\frac{{{{\left| {{\rm{Tr }}{{\boldsymbol{\tau}} }} \right|}^{}}}}{{{{\left( {\left| {{\rm{Tr }}{{\boldsymbol{\tau}} }} \right| + \beta \left\| {\,{\rm{dev}}\;{{\boldsymbol{\tau}} }} \right\|} \right)}^{\frac{p-1}{{p+ 1}}}}}},{\rm{           }} \qquad 2\beta^2 {\mu  } = \frac{\Lambda }{{{{\big( r{{{\widetilde{F}_{vp}}\left( {\boldsymbol{\tau}} \right) }} \big)}^{\frac{1}{{p}}}}}}\frac{{\beta \left\| {{\rm{dev}}\;{{\boldsymbol{\tau}} }} \right\|}}{{{{\left( {\left| {{\rm{Tr }}{{\boldsymbol{\tau}}}} \right| + \beta \left\| {{\rm{dev}}\;{{\boldsymbol{\tau}} }} \right\|} \right)}^{\frac{p-1}{{p + 1}}}}}}.
\end{equation*}  
The formula for the minimum compliance follows:
\begin{equation}
	\hat{\mathcal{C}}_{sp} = \frac{1}{\Lambda }{\left( {{r Y_r}} \right)^{\frac{2}{r}}},
\end{equation}
where $Y_r$ is the minimal energy in the stress-based reformulation of the (sp-IMD) problem:
\begin{equation}
	\label{eq:Yr}
	{Y_r} :=  \mathop {\min }\limits_{{\boldsymbol{\tau}} \in \Sigma_{{f}} \left( \Omega  \right)} \frac{1}{r} \int_\Omega  {{{\Big( {{{\left| {{\rm{Tr }}\,{{\boldsymbol{\tau}}}} \right|}^{}} + \beta {{\left\| {{\rm{dev}}\,{{\boldsymbol{\tau}} }} \right\|}^{}}} \Big)}^{r}}dx}.
\end{equation}   
Similarly as in the previous subsection, the functional $({\widetilde{F}_{sp}}\left( \argu \right))^{1/r}$ is equivalent to the  standard norm in $L^r(\Omega;E_s^2)$, and existence of solutions in \eqref{eq:Yr} can be readily deduced. Nonetheless, the ${\widetilde{F}_{sp}}\left( \argu \right)$ is not strictly convex; hence, we cannot be sure that solutions of \eqref{eq:Yr} are unique. To sum up:

\begin{theorem}
	\label{thm:stress_sp}
	Let $\hat{\boldsymbol{\tau}} \in \Sigma_f(\Omega)$ be a (possibly non-unique) solution of the stress-based problem \eqref{eq:Yr}, which always exists. Then, the following $(\hat{k},\hat{\mu})$ is a (possibly non-unique) solution of the (sp-IMD) problem \eqref{eq:sp-IMD}:
\begin{equation}
	\label{eq:opt_kmu_tau_sp}
	d{\hat{k} } = \frac{\Lambda }{{{{\left( {{r Y_r}} \right)}^{\frac{1}{{p}}}}}} \frac{{{{\left| {{\rm{Tr }}\,{\hat{{\boldsymbol{\tau}}} }} \right|}^{}}}}{{{{\left( {\left| {{\rm{Tr }}\,{\hat{{\boldsymbol{\tau}}} }} \right| + \beta \left\| {{\rm{dev}}\,{\hat{{\boldsymbol{\tau}}} }} \right\|} \right)}^{\frac{p-1}{{p+ 1}}}}}},\qquad 2 \hat{\mu  } = \frac{\Lambda }{{{{\left( {{r Y_r}} \right)}^{\frac{1}{{p}}}}}}\frac{{\frac{1}{\beta} \left\| {{\rm{dev}}\,{\hat{{\boldsymbol{\tau}}} }} \right\|}}{{{{\left( {\left| {{\rm{Tr }}{\hat{{\boldsymbol{\tau}}}}} \right| + \beta \left\| {{\rm{dev}}\,{\hat{{\boldsymbol{\tau}}} }} \right\|} \right)}^{\frac{p-1}{{p + 1}}}}}}.
\end{equation}  
Conversely, for every pair $(\hat{k},\hat{\mu})$ that solves the problem (sp-IMD), there exists a solution $\hat{\boldsymbol{\tau}}$ of \eqref{eq:Yr} that satisfies \eqref{eq:opt_kmu_tau_sp} a.e. in $\Omega$. 
\end{theorem}

The rules of retrieving the optimal Young modulus $\hat{E}$ and the optimal Poisson ration $\hat{\nu}$ are identical to those described in the previous subsection.

\section{Displacement-based dual problems and optimality conditions }
\label{sec:duality}

Both from the theoretical point of view  as well as from the point of view of forming numerical methods it is indispensable not to consider the  problems  (3.19), (3.30) solely, but complement them by the dual formulations along with the relevant optimality conditions. The present section is aimed at deriving them. By analogy with elasticity problems these dual settings will be named: displacement-based formulations.

\subsection{The general duality result}

The result to be put forward will provide the common duality framework for both the problems (vp-IMD) and (sp-IMD). It employs classical duality arguments on reflexive Banach spaces.

\begin{theorem}
	\label{thm:general_duality}
	Let $r \in (1,\infty)$, and $\frac{1}{r} + \frac{1}{r'} = 1$.
	For any norm $\opnorm{\argu}$ on the space of symmetric matrices $E_s^2$ the following zero-gap equality holds true:
	\begin{equation}
		\label{eq:duality}
		\min_{\boldsymbol{\tau} \in \Sigma_{{f}}(\Omega)} \int_\Omega \frac{1}{r} \opnorm{\boldsymbol{\tau}}^r  dx = \max_{\mbf{v} \in \mbf{V}^{r'}\!(\Omega) } \left\{ {f}(\mbf{v})- \int_\Omega \frac{1}{r'} \opnorm{\boldsymbol{\eps}(\mbf{v})}_*^{r'} \!dx\right\},
	\end{equation}
        where $\opnorm{\argu}_*$ is the dual norm, see Section \ref{ssec:notation}.
	In particular, the minimum and the maximum are finite and attained.  
	
	Moreover, the functions $\hat{\mbf{v}} \in	\mbf{V}^{r'}\!(\Omega)$ and $\hat{\boldsymbol{\tau}} \in \Sigma_{{f}}(\Omega)$ solve the respective problems if and only if
	\begin{equation}
		\label{eq:opt_cond_norm}
			\hat{\boldsymbol{\tau}}\cdot \boldsymbol{\eps}(\hat{\mbf{v}}) = 
	\frac{1}{r} \opnorm{\hat{\boldsymbol{\tau}}}^r +\frac{1}{r'} \opnorm{\boldsymbol{\eps}(\hat{\mbf{v}})}_*^{r'} \qquad \text{a.e. in $\Omega$}.
	\end{equation}
\end{theorem}
\begin{proof} Consider lower semi-continuous convex functionals 
	$F:\mbf{V}^{r'}\!(\Omega) \to \R \cup \{+\infty\}$,  $G: L^{r'}\!(\Omega;E_s^2) \to \R \cup \{+\infty\}$ and a linear continuous operator $\Lambda:\mbf{V}^{r'}\!(\Omega) \to L^{r'}\!(\Omega;E_s^2)$. Then, if there exists $\mbf{v}_0$ such that $F(\mbf{v}_0) < +\infty$ and $G$ is continuous at $\Lambda \mbf{v}_0$, the equality holds:	
	\begin{equation}
		\label{eq:abstract_duality}
		\sup _{\mbf{v} \in \mbf{V}^{r'}\!(\Omega)} \Big \{ - F (\mbf{v}) - G (\Lambda \mbf{v})   \Big \} = \inf
		_{\boldsymbol{\tau} \in L^r(\Omega;E_s^2)} \Big \{ F^* (- \Lambda ^* \boldsymbol{\tau})+ G^* (\boldsymbol{\tau}) \Big \},
	\end{equation}
	while the minimum is attained provided that the value function is finite. For the proof see  Chapter III of \cite{ekeland1999}. Let us put
	$F (\mbf{v}) = -{f}(\mbf{v})$, \ $G(\boldsymbol{\epsilon}) = \int_\Omega g(\boldsymbol{\epsilon}) dx $, \ $\Lambda \mbf{v} = \boldsymbol{\eps}(\mbf{v})$, where the integrand $g: E_s^2 \to \R_+$ reads $g(\boldsymbol{\epsilon})= \frac{1}{r'} \opnorm{\boldsymbol{\epsilon}}_*^{r'}$. The prerequisites for \eqref{eq:abstract_duality} is trivial to check as $F,G$ are continuous functionals on the respective Banach spaces. It is easy to check that
	\begin{equation}
		F^* (- \Lambda ^* \boldsymbol{\tau}) = \begin{cases}
			0 & \text{if} \ \ \int_\Omega  {{\boldsymbol{\tau}} \cdot \,} \boldsymbol{\eps}\left( {\bf{v}} \right)dx = f(\mbf{v})  \quad \forall\, \mbf{v} \in \mbf{V}^{r'}\!(\Omega), \\
			+\infty & \text{otherwise}.
		\end{cases}
	\end{equation}	
	Therefore, the finiteness of the minimized functional implies that $\boldsymbol{\tau} \in \Sigma_f(\Omega)$. 
	Meanwhile, by the Rockafellar theorem \cite{rockafellar1968},
	$G^*(\boldsymbol{\tau}) = \int_\Omega g^*(\boldsymbol{\tau}) \,dx$. Moreover, $g^*(\boldsymbol{\tau})  = \frac{1}{r} \opnorm{\boldsymbol{\tau}}^r$, see Corollary 15.3.1 in \cite{rockafellar1970}. Readily, the abstract duality \eqref{eq:abstract_duality} furnishes the equality \eqref{eq:duality}.
	
	Checking existence of the minimizer goes as in the case of problems \eqref{eq:Zr} or \eqref{eq:Yr}. To see that the maximum is attained, we use Korn's inequality to show that $\big( \int_\Omega  \opnorm{\boldsymbol{\eps}(\mbf{v})}_*^{r'} \!dx \big)^{1/r'}$ is a norm on the space $\mbf{V}^{r'}\!(\Omega)$ that is equivalent to the standard norm on $W^{1,r'}\!(\Omega;\Rd)$, see Theorem 6.3-4 in \cite{ciarlet2000a}. Then, thanks to the continuity of the linear functional $f$, the maximized functional is coercive on the reflexive space $\mbf{V}^{r'}\!(\Omega)$. The existence of the maximizer follows, see \cite{ekeland1999}.
	
	To show the optimality condition \eqref{eq:opt_cond_norm} let us take any functions $\mbf{v} \in \mbf{V}^{r'}\!(\Omega)$ and $\boldsymbol{\tau} \in \Sigma_f(\Omega)$. Since $f(\mbf{v}) =  \int_\Omega  {{\boldsymbol{\tau}} \cdot \,} \boldsymbol{\eps}( {\bf{v}} )dx$, from the equality \eqref{eq:duality} follows that $\mbf{v}, \boldsymbol{\tau}$ are optimal if and only if $	\int_\Omega  {{\boldsymbol{\tau}} \cdot \,} \boldsymbol{\eps}( {\bf{v}} )dx  =  \int_\Omega g( \boldsymbol{\eps}( {\bf{v}} )) \,dx  + \int_\Omega g^*(\boldsymbol{\tau}) \,dx.$
	Meanwhile, ${{\boldsymbol{\tau}} \cdot \,} \boldsymbol{\eps}( {\bf{v}} )\leq g^*(\boldsymbol{\tau})+ g ( \boldsymbol{\eps}( {\bf{v}} ))$ pointwisely in $\Omega$ by the very definition of the Fenchel-Legendre conjugate $g^*$. Therefore, $\boldsymbol{\tau}$, $\mbf{v}$ are optimal if and only if \eqref{eq:opt_cond_norm} holds true a.e. in $\Omega$.
\end{proof}

\subsection{Duality in the vp-IMD variant}

To develop the duality scheme for the stress-based formulation of the (vp-IMD) problem, the general duality result in Theorem \ref{thm:general_duality} needs to be particularized for the norm 
\begin{equation}
    \label{eq:norm_vp}
	\opnorm{\boldsymbol{\tau}} = \Big( {{{\left| {{\rm{Tr }}\,{\boldsymbol{\tau}}} \right|}^{r}} +\beta^{2-r} {{\left\| {{\rm{dev}}\,{\boldsymbol{\tau}}} \right\|}^{r}}} \Big)^{\frac{1}{r}}.
\end{equation}
Indeed, the minimization problem in \eqref{eq:duality} becomes \eqref{eq:Zr} exactly. To compute the dual norm $\opnorm{\argu}_*$ and to characterize the extremality condition \eqref{eq:opt_cond_norm}, we exploit the formula \eqref{eq:dot_decomp} to provide the following chain of estimates for arbitrary $\boldsymbol{\tau}, \boldsymbol{\eps} \in E_s^2$:
\begin{align*}
	\boldsymbol{\tau} \cdot \boldsymbol{\eps} & = \mathrm{Tr}\,\boldsymbol{\tau} \, \mathrm{Tr}\,\boldsymbol{\eps} + \mathrm{dev} \, \boldsymbol{\tau} \cdot \mathrm{dev} \, \boldsymbol{\eps}\\
	& \leq\abs{\mathrm{Tr}\,\boldsymbol{\tau}} \, \abs{\mathrm{Tr}\,\boldsymbol{\eps}} + \norm{\mathrm{dev} \, \boldsymbol{\tau}}  \norm{\mathrm{dev} \, \boldsymbol{\eps}}\\
	&  = \abs{\mathrm{Tr}\,\boldsymbol{\tau}} \, \abs{\mathrm{Tr}\,\boldsymbol{\eps}} + \beta^{\frac{r-2}{r}}\norm{\mathrm{dev} \, \boldsymbol{\tau}}  \,\beta^{\frac{2-r}{r}}\norm{\mathrm{dev} \, \boldsymbol{\eps}}\\
	& \leq \frac{1}{r}\abs{\mathrm{Tr}\,\boldsymbol{\tau}}^r +  \frac{1}{r'} \abs{\mathrm{Tr}\,\boldsymbol{\eps}}^{r'} + \frac{1}{r}  \beta^{{r-2}}\norm{\mathrm{dev} \, \boldsymbol{\tau}}^r+ \frac{1}{r'} \beta^{{r'-2}}\norm{\mathrm{dev} \, \boldsymbol{\eps}}^{r'}\\
	& = \frac{1}{r} \opnorm{\boldsymbol{\tau}}^{r} + \frac{1}{r'}\Big( \abs{\mathrm{Tr}\,\boldsymbol{\eps}}^{r'} +\beta^{{r'-2}}\norm{\mathrm{dev} \, \boldsymbol{\eps}}^{r'}  \Big)
\end{align*}
The second inequality combines two Young inequalities. The first inequality above is an equality if and only if $\alpha_1\mathrm{Tr}\, \boldsymbol{\tau} = \beta_1\mathrm{Tr}\, \boldsymbol{\eps}$  and $\alpha_2 \,\mathrm{dev}\, \boldsymbol{\tau} = \beta_2 \,\mathrm{dev}\, \boldsymbol{\eps}$ for some $\alpha_i,\beta_i \geq 0$. The Young inequalities are sharp if and only if  $\abs{\mathrm{Tr}\,\boldsymbol{\tau}}^r  = \abs{\mathrm{Tr}\,\boldsymbol{\eps}}^{r'}$ and $\beta^{{r-2}}\norm{\mathrm{dev} \, \boldsymbol{\tau}}^r = \beta^{{r'-2}}\norm{\mathrm{dev} \, \boldsymbol{\eps}}^{r'}$. Combining the two conditions, above we have equalities if and only if the power law-type equations hold true:
\begin{equation}
	\label{eq:saturatuing_tau_vp}
	\mathrm{Tr}\,\boldsymbol{\eps} = \abs{\mathrm{Tr}\,\boldsymbol{\tau}}^{r-2} \mathrm{Tr}\,\boldsymbol{\tau}, \qquad  \mathrm{dev}\,\boldsymbol{\eps} = \big(\beta\,\norm{\mathrm{dev}\,\boldsymbol{\tau}}\big)^{r-2} \mathrm{dev}\,\boldsymbol{\tau}
\end{equation}
or, equivalently, $		\mathrm{Tr}\,\boldsymbol{\tau} = \abs{\mathrm{Tr}\,\boldsymbol{\eps}}^{r'-2} \mathrm{Tr}\,\boldsymbol{\eps},$ and $\mathrm{dev}\,\boldsymbol{\tau} = \big(\beta\,\norm{\mathrm{dev}\,\boldsymbol{\eps}}\big)^{r'-2} \mathrm{dev}\,\boldsymbol{\eps}$.

Next, based on the chain above, we observe that the convex function $\boldsymbol{\eps} \mapsto  \frac{1}{r'}\big( \abs{\mathrm{Tr}\,\boldsymbol{\eps}}^{r'} +\beta^{{r'-2}}\norm{\mathrm{dev} \, \boldsymbol{\eps}}^{r'}  \big)$ is an upper bound on the function  $\boldsymbol{\eps} \mapsto \boldsymbol{\tau} \cdot \boldsymbol{\eps}  -  \frac{1}{r} \opnorm{\boldsymbol{\tau}}^{r} $ whatever is $\boldsymbol{\tau} \in E_s^2$. Since this bound is saturated for $\boldsymbol{\tau}$ as in \eqref{eq:saturatuing_tau_vp}, the bound must be non-other than the Fenchel-Legendre transform of $ \frac{1}{r} \opnorm{\argu}^{r}$, i.e. it must be the function $ \frac{1}{r'} \opnorm{\argu}_*^{r'}$. We thus find that
\begin{equation}
	 \opnorm{\boldsymbol{\eps}}_* = \Big( \abs{\mathrm{Tr}\,\boldsymbol{\eps}}^{r'} +\beta^{{r'-2}}\norm{\mathrm{dev} \, \boldsymbol{\eps}}^{r'}  \Big)^{\frac{1}{r'}}
\end{equation}
and, as a result, \eqref{eq:saturatuing_tau_vp} characterizes the optimality condition in \eqref{eq:opt_cond_norm} when $\boldsymbol{\eps} = \boldsymbol{\eps}(\mbf{v})$. Readily, duality for the (vp-IMD) follows as a particular case of Theorem \ref{thm:general_duality}:

\begin{corollary}
	\label{cor:dual_vp}
	Consider the stress-based formulation of the (vp-IMD) problem:
	\begin{align}
		\label{eq:primal_vp}
		{Z_r} = \mathop {\min }\limits_{{\boldsymbol{\tau}} \in \Sigma_{{f}} \left( \Omega  \right)} \frac{1}{r} \int_\Omega  {\Big( {{{\left| {{\rm{Tr }}\,{\boldsymbol{\tau}}} \right|}^{r}} + \beta^{2-r}{{\left\| {{\rm{dev}}\,{\boldsymbol{\tau}}} \right\|}^{r}}} \Big)dx}.
	\end{align}	
	Then, $Z_r < +\infty$, and its  dual displacement-based problem reads
	\begin{equation}
		\label{eq:dual_vp}
		Z_r = \max_{\mbf{v} \in \mbf{V}^{r'}\!(\Omega) } \left\{ {f}(\mbf{v})- \frac{1}{r'}  \int_\Omega \Big( \abs{\mathrm{Tr}\,\boldsymbol{\eps}(\mbf{v})}^{r'} +\beta^{{r'-2}}\norm{\mathrm{dev} \, \boldsymbol{\eps}(\mbf{v})}^{r'}  \Big) dx\right\}
	\end{equation}
	and attains a unique solution.
	Moreover, admissible functions $\hat{\boldsymbol{\tau}} \in \Sigma_f(\Omega)$ and $\hat{\mbf{v}} \in \mbf{V}^{r'}\!(\Omega)$ solve the respective problems if and only if the power law holds true:
	\begin{equation}
			\label{eq:eps_tau_vp}
			\mathrm{Tr}\,\boldsymbol{\eps}(\hat{\mbf{v}}) = \abs{\mathrm{Tr}\,\hat{\boldsymbol{\tau}}}^{r-2} \mathrm{Tr}\,\hat{\boldsymbol{\tau}}, \qquad
			\mathrm{dev}\,\boldsymbol{\eps}(\hat{\mbf{v}}) = \big(\beta\,\norm{\mathrm{dev}\,\hat{\boldsymbol{\tau}}}\big)^{r-2} \mathrm{dev}\,\hat{\boldsymbol{\tau}} \qquad \text{a.e. in } \Omega.
	\end{equation}
\end{corollary}

It should be noted that uniqueness of solutions $\hat{\mbf{v}}$ does not follow directly from Theorem \eqref{thm:general_duality}, as it is not guaranteed for any choice of the norm $\opnorm{\argu}$ (the induced unit sphere must be smooth). Here the uniqueness may be deduced from the strict concavity of the maximized functional. Alternatively, since we know that $\hat{\boldsymbol{\tau}}$ is unique (see Theorem \ref{thm:stress_vp}), uniqueness of $\hat{\mbf{v}}$ is a consequence of the power law \eqref{eq:eps_tau_vp} and of the fact that $\boldsymbol{\eps}(\mbf{v}) \equiv 0 $ implies $\mbf{v} \equiv 0$ for $\mbf{v} \in \mbf{V}^{r'}\!(\Omega)$.

It should be stressed that only in the result above we have a proof that $Z_r$ is finite or, equivalently, $\Sigma_f(\Omega)$ is non-empty. This  also implies finiteness of the minimum compliance $\hat{\mathcal{C}}_{vp}$ through the formula \eqref{eq:Cvp}. This was a missing piece of information in all the existence results hitherto put forward. They are all now validated.

\subsection{Duality in the sp-IMD variant}

We reproduce the same steps in the sp-IMD setting, where the norm should be  
\begin{equation}
    \label{eq:norm_sp}
	\opnorm{\boldsymbol{\tau}} =  {{{\left| {{\rm{Tr }}\,{\boldsymbol{\tau}}} \right|}} +\beta {{\left\| {{\rm{dev}}\,{\boldsymbol{\tau}}} \right\|}}}.
\end{equation}
The chain of estimates must be handled differently:
\begin{align*}
	\boldsymbol{\tau} \cdot \boldsymbol{\eps} 
	& \leq\abs{\mathrm{Tr}\,\boldsymbol{\tau}} \, \abs{\mathrm{Tr}\,\boldsymbol{\eps}} + \norm{\mathrm{dev} \, \boldsymbol{\tau}}  \norm{\mathrm{dev} \, \boldsymbol{\eps}}\\
	&  = \abs{\mathrm{Tr}\,\boldsymbol{\tau}} \, \abs{\mathrm{Tr}\,\boldsymbol{\eps}} + \beta\norm{\mathrm{dev} \, \boldsymbol{\tau}}  \,\tfrac{1}{\beta}\norm{\mathrm{dev} \, \boldsymbol{\eps}}\\
	& \leq \Big( \abs{\mathrm{Tr}\,\boldsymbol{\tau}} +\beta\norm{\mathrm{dev} \, \boldsymbol{\tau}}   \Big) \max\Big\{ \abs{\mathrm{Tr}\,\boldsymbol{\eps}} , \tfrac{1}{\beta}\norm{\mathrm{dev} \, \boldsymbol{\eps}} \Big\} \\
	& \leq \frac{1}{r} \Big( \abs{\mathrm{Tr}\,\boldsymbol{\tau}} +\beta\norm{\mathrm{dev} \, \boldsymbol{\tau}}   \Big)^{r} + \frac{1}{r'}\Big(  \max\Big\{ \abs{\mathrm{Tr}\,\boldsymbol{\eps}} , \tfrac{1}{\beta}\norm{\mathrm{dev} \, \boldsymbol{\eps}} \Big\}  \Big)^{r'},
\end{align*}
where to pass to the last line we used the Young inequality. The three inequalities are equalities if and only if, respectively,
\begin{itemize}
	\item[(1)] $\alpha_1\mathrm{Tr}\, \boldsymbol{\tau} = \beta_1\mathrm{Tr}\, \boldsymbol{\eps}$  and $\alpha_2 \,\mathrm{dev}\, \boldsymbol{\tau} = \beta_2 \,\mathrm{dev}\, \boldsymbol{\eps}$ for some $\alpha_i,\beta_i \geq 0$;
	\item[(2)] $\mathrm{Tr}\,\boldsymbol{\tau} = 0$ unless  $\abs{\mathrm{Tr}\,\boldsymbol{\eps}} \geq \tfrac{1}{\beta}\norm{\mathrm{dev} \, \boldsymbol{\eps}} 
    $, and $\mathrm{dev}\,\boldsymbol{\tau} = 0$ unless  $\abs{\mathrm{Tr}\,\boldsymbol{\eps}} \leq \tfrac{1}{\beta}\norm{\mathrm{dev} \, \boldsymbol{\eps}} 
    $;
	\item [(3)] the equality  $\big( \abs{\mathrm{Tr}\,\boldsymbol{\tau}} +\beta\norm{\mathrm{dev} \, \boldsymbol{\tau}}   \big)^{r}  = \big(  \max\big\{ \abs{\mathrm{Tr}\,\boldsymbol{\eps}} , \tfrac{1}{\beta}\norm{\mathrm{dev} \, \boldsymbol{\eps}} \big\}  \big)^{r'}$ holds true.
\end{itemize}
The system of conditions (1),\,(2),\,(3) can be rewritten as a pair of inclusions
\begin{equation}
	\label{eq:relations_sp}
	\mathrm{Tr}\,\boldsymbol{\eps} \in \big(\abs{\mathrm{Tr}\,\boldsymbol{\tau}}+\beta\,\norm{\mathrm{dev}\,\boldsymbol{\tau}}\big)^{r-1} \mathcal{N} (\mathrm{Tr}\,\boldsymbol{\tau}), \qquad \mathrm{dev}\,\boldsymbol{\eps} \in  \big(\abs{\mathrm{Tr}\,\boldsymbol{\tau}}+\beta\,\norm{\mathrm{dev}\,\boldsymbol{\tau}}\big)^{r-1} \beta\, \mathcal{N}(\mathrm{dev}\,\boldsymbol{\tau}).
\end{equation}
where, for a scalar $t\in \R$ and a tensor $\boldsymbol{\tau} \in E_s^2$, we introduce the multivalued normalization operators:
\begin{equation}
    \label{eq:N}
	{\mathcal{N}}(t) := \begin{cases}
		\big\{\! \frac{t}{\abs{t}}\! \big\} & \text{if} \quad t \neq 0, \\
		[-1,1] & \text{if} \quad t = 0,
	\end{cases}
	\qquad\qquad
	\mathcal{N}(\boldsymbol{\tau}) := \begin{cases}
		\big\{\! \frac{\boldsymbol{\tau}}{\norm{\boldsymbol{\tau}}} \!\big\} & \text{if} \quad \boldsymbol{\tau} \neq 0, \\
		B_{E_s^2}(0,1) & \text{if} \quad \boldsymbol{\tau} = 0.
	\end{cases}
\end{equation}
Above, $B_{E_s^2}(0,1)$ stands for the closed unit  ball in  $E_s^2$ (with respect to the Euclidean norm $\norm{\argu}$).

Arguing as in the previous subsection, the dual norm is recovered as
\begin{equation}
	\opnorm{\boldsymbol{\eps}}_* = \max\Big\{ \abs{\mathrm{Tr}\,\boldsymbol{\eps}} , \tfrac{1}{\beta}\norm{\mathrm{dev} \, \boldsymbol{\eps}} \Big\},
\end{equation}
see also \cite{czarnecki2015a,bolbotowski2022b} where this formula has also been found.  Moreover, the relations \eqref{eq:relations_sp} characterize the extremality condition \eqref{eq:opt_cond_norm}. The duality framework for the problem (sp-IMD) can be now stated:
\begin{corollary}
	Consider the stress-based formulation of the (sp-IMD) problem:
	\begin{align}
		\label{eq:primal_sp}
		{Y_r} =  \mathop {\min }\limits_{{\boldsymbol{\tau}} \in \Sigma_{{f}} \left( \Omega  \right)} \frac{1}{r} \int_\Omega  {{{\Big( {{{\left| {{\rm{Tr }}\,{{\boldsymbol{\tau}}}} \right|}^{}} + \beta {{\left\| {{\rm{dev}}\,{{\boldsymbol{\tau}} }} \right\|}^{}}} \Big)}^{r}}dx}.
	\end{align}	
	Then, $Y_r < +\infty$, and its the dual displacement-based problem reads
	\begin{equation}
		\label{eq:dual_sp}
		Y_r = \max_{\mbf{v} \in \mbf{V}^{r'}\!(\Omega) } \left\{ {f}(\mbf{v})- \frac{1}{r'}  \int_\Omega \Big(\max\big\{ \abs{\mathrm{Tr}\,\boldsymbol{\eps}(\mbf{v})} , \tfrac{1}{\beta}\norm{\mathrm{dev} \, \boldsymbol{\eps}(\mbf{v})} \big\} \Big)^{r'} dx\right\}.
	\end{equation}
	and attains a (possibly non-unique) solution.
	Moreover, admissible functions $\hat{\boldsymbol{\tau}} \in \Sigma_f(\Omega)$ and $\hat{\mbf{v}} \in \mbf{V}^{r'}\!(\Omega)$ solve the respective problems if and only if the optimality relations hold true a.e. in $\Omega$:
	\begin{equation*}
			\mathrm{Tr}\,\boldsymbol{\eps}(\hat{\mbf{v}}) \in \big(\abs{\mathrm{Tr}\,\hat{\boldsymbol{\tau}}}+\beta\,\norm{\mathrm{dev}\,\hat{\boldsymbol{\tau}}}\big)^{r-1} \mathcal{N} (\mathrm{Tr}\,\hat{\boldsymbol{\tau}}), \qquad \mathrm{dev}\,\boldsymbol{\eps}(\hat{\mbf{v}}) \in  \big(\abs{\mathrm{Tr}\,\hat{\boldsymbol{\tau}}}+\beta\,\norm{\mathrm{dev}\,\hat{\boldsymbol{\tau}}}\big)^{r-1} \beta\, \mathcal{N}(\mathrm{dev}\,\hat{\boldsymbol{\tau}}).
	\end{equation*}
\end{corollary}

\section{Solutions via a system  of non-linear elasticity}

The original IMD problem [16] reduces to a linear constrained problem composed of two mutually dual problems linked by an optimality condition. In this section we show that the  vp-IMD and sp-IMD methods reduce to solving certain problems which can be interpreted as static problems of elastic bodies with specific non-linear constitutive equations.

\subsection{The equations of non-linear elasticity in the vp-IMD method}

Set any exponent $p \in (1,\infty)$ and then define $r = \frac{2p}{p+1} \in (1,2)$ and its H\"{older} conjugate exponent $r' = \frac{r}{r-1} = \frac{2p}{p-1}$. Consider the following problem: find a virtual stress $\hat{\boldsymbol{\tau}} \in L^r(\O;E_s^2)$ and virtual displacement $\hat{\mbf{v}} \in W^{1,r'}\!(\Omega;\Rd)$ which satisfy the system of equations
\begin{equation}
	\label{eq:non-linear_system_vp}
	\begin{cases}
		(1) & \int_\Omega  {\hat{\boldsymbol{\tau}} \cdot \,} {\boldsymbol{\eps}}\left( {\bf{v}} \right)dx  = {f}(\mbf{v}) \qquad \forall \,  \mbf{v} \in \mbf{V}^{r'}\!(\Omega),\\
		(2) & \hat{\mbf{v}} \in \mbf{V}^{r'}\!(\Omega),\\
		(3.a) & \mathrm{Tr}\,\boldsymbol{\eps}(\hat{\mbf{v}}) = \abs{\mathrm{Tr}\,\hat{\boldsymbol{\tau}}}^{r-2} \mathrm{Tr}\,\hat{\boldsymbol{\tau}} \qquad \text{a.e. in $\Omega$},\\
		(3.b) & \mathrm{dev}\,\boldsymbol{\eps}(\hat{\mbf{v}}) = \big(\beta\,\norm{\mathrm{dev}\,\hat{\boldsymbol{\tau}}}\big)^{r-2} \mathrm{dev}\,\hat{\boldsymbol{\tau}} \qquad \text{a.e. in $\Omega$}.
	\end{cases}
\end{equation}
The above formulation can be deemed a system of equations of non-linear elasticity with the power law-type constitutive equations. The forthcoming result shows that this problem is, in fact, equivalent to the (vp-IMD) optimal design problem.

\begin{theorem}
	\label{thm:non-linear_sys_vp}
	The system of equations of non-linear elasticity \eqref{eq:non-linear_system_vp}  has a unique solution in the form of the virtual stress $\hat{\boldsymbol{\tau}} \in L^r(\Omega;E_s^2)$ and the virtual displacement $\hat{\mbf{v}} \in W^{1,r'}\!(\Omega;\Rd)$. For any such a solution, the following statements hold true.
	\begin{enumerate}[label={(\roman*)}]
		\item $\hat{\boldsymbol{\tau}}$ is the unique solution of the stress-based minimization problem \eqref{eq:primal_vp}, and  $\hat{\mbf{v}}$ is the unique solution of the displacement-based maximization problem \eqref{eq:dual_vp}.
		\item For $Z_r = \frac{1}{r} \int_\Omega  {\big( {{{\left| {{\rm{Tr }}\,\hat{\boldsymbol{\tau}}} \right|}^{r}} + \beta^{2-r}{{\left\| {{\rm{dev}}\,\hat{\boldsymbol{\tau}}} \right\|}^{r}}} \big)dx}$ define the functions below:
		\begin{alignat}{2}
			&\hat{\boldsymbol{\sigma}} = \hat{\boldsymbol{\tau}}, \qquad
			&&\hat{\mbf{u}} = \frac{(r Z_r)^{\frac{1}{p}}}{\Lambda} \, \hat{\mbf{v}},\\
			&{\hat{k}} = \frac{\Lambda }{{{{d( {{r Z_r}} )}^{\frac{1}{{p}}}}}}{\left| {{\rm{Tr }}\,{\hat{\boldsymbol{\tau}} }} \right|^{\frac{2}{p+1}}}, \qquad 
			&&{\hat\mu} = \frac{\Lambda }{{{{2( {{r Z_r}} )}^{\frac{1}{{p}}}}}}{\left\| { \tfrac{1}{\beta} \,{\rm{dev}}\,{\hat{\boldsymbol{\tau}} }} \right\|^{\frac{2}{p+1}}}.
		\end{alignat}
		Then, $(\hat{k},\hat{\mu})$ uniquely  solves the (vp-IMD) problem in \eqref{eq:vp-IMD}, while $\hat{\boldsymbol{\sigma}}$ and $\hat{\mbf{u}}$ are, respectively, the stress and displacement fields in the optimal body. Namely, $\hat{\boldsymbol{\sigma}} \in \Sigma_f(\Omega)$ and the constitutive law of linear isotropy is met:
		\begin{equation}
			\label{eq:const_opt}
			{\rm{Tr}} \,\hat{\boldsymbol{\sigma}} = d\hat{k} \,{\rm{Tr }}\,{\boldsymbol{\eps}}(\hat{\bf{u}}), \qquad {\rm{dev}}\,\hat{\boldsymbol{\sigma}} = 2\hat\mu\, {\rm{dev}}\,{\boldsymbol{\eps}}(\hat{\bf{u}}) \qquad \text{a.e. in $\O$}.
		\end{equation}
	\end{enumerate}
\end{theorem}
\begin{proof}
	Existence and uniqueness of minimizer $\hat{\boldsymbol{\tau}}$ and maximizer $\hat{\mbf{v}}$ in problems \eqref{eq:primal_vp} and, respectively, \eqref{eq:dual_vp}, have been already established in Theorem \ref{thm:stress_vp} and Corollary \ref{cor:dual_vp}. Then, the equations (3.a), (3.b) in \eqref{eq:non-linear_system_vp} must be satisfied as optimality conditions, see the "moreover part" in Corollary \ref{cor:dual_vp}. It is thus left to show the assertion (ii), i.e. the equations \eqref{eq:const_opt}. This amounts to simply plugging the formulae for $\hat{\boldsymbol{\sigma}}, \hat{\mbf{v}}, \hat{k}, \hat{\mu}$ into the constitutive law \eqref{eq:const_opt}. After using relations between $p$ and $r$, this leads to equations  (3.a), (3.b) in \eqref{eq:non-linear_system_vp} exactly.
\end{proof}

\subsection{The inclusions of non-linear elasticity in the sp-IMD method}

Similarly, we will show that the (sp-IMD) method is equivalent to the system where the constitutive law are two non-linear inclusions:
\begin{equation}
	\label{eq:non-linear_system_sp}
	\begin{cases}
		(1) & \int_\Omega  {\hat{\boldsymbol{\tau}} \cdot \,} {\boldsymbol{\eps}}\left( {\bf{v}} \right)dx  = {f}(\mbf{v}) \qquad \forall \,  \mbf{v} \in \mbf{V}^{r'}\!(\Omega),\\
		(2) & \hat{\mbf{v}} \in \mbf{V}^{r'}\!(\Omega),\\
		(3.a) & \mathrm{Tr}\,\boldsymbol{\eps}(\hat{\mbf{v}}) \in \big(\abs{\mathrm{Tr}\,\hat{\boldsymbol{\tau}}}+\beta\,\norm{\mathrm{dev}\,\hat{\boldsymbol{\tau}}}\big)^{r-1} \mathcal{N} (\mathrm{Tr}\,\hat{\boldsymbol{\tau}}) \qquad \text{a.e. in $\Omega$},\\
		(3.b) & \mathrm{dev}\,\boldsymbol{\eps}(\hat{\mbf{v}}) \in  \big(\abs{\mathrm{Tr}\,\hat{\boldsymbol{\tau}}}+\beta\,\norm{\mathrm{dev}\,\hat{\boldsymbol{\tau}}}\big)^{r-1} \beta\, \mathcal{N}(\mathrm{dev}\,\hat{\boldsymbol{\tau}}) \qquad \text{a.e. in $\Omega$}.
	\end{cases}
\end{equation}
The proof of the theorem below runs in full analogy to the proof of Theorem \ref{thm:non-linear_sys_vp}.

\begin{theorem}
	The system of inclusions of non-linear elasticity \eqref{eq:non-linear_system_sp}  has a (possibly non-unique) solution in the form of the virtual stress $\hat{\boldsymbol{\tau}} \in L^r(\Omega;E_s^2)$ and the virtual displacement $\hat{\mbf{v}} \in W^{1,r'}\!(\Omega;\Rd)$. For such a solution, the following statements hold true.
	\begin{enumerate}[label={(\roman*)}]
		\item $\hat{\boldsymbol{\tau}}$ is a solution of the stress-based minimization problem \eqref{eq:primal_sp}, and  $\hat{\mbf{v}}$ is a solution of the displacement-based maximization problem \eqref{eq:dual_sp}.
		\item For $Y_r =\frac{1}{r} \int_\Omega  {{{\big( {{{\left| {{\rm{Tr }}\,{\hat{\boldsymbol{\tau}}}} \right|}^{}} + \beta {{\left\| {{\rm{dev}}\,{\hat{\boldsymbol{\tau}} }} \right\|}^{}}} \big)}^{r}}dx}$ define the functions below:
		\begin{alignat}{2}
			&\hat{\boldsymbol{\sigma}} = \hat{\boldsymbol{\tau}}, \qquad
			&&\hat{\mbf{u}} = \frac{(r Y_r)^{\frac{1}{p}}}{\Lambda} \, \hat{\mbf{v}},\\
			&{\hat{k}} = \frac{\Lambda }{{{{d( {{r Y_r}} )}^{\frac{1}{{p}}}}}}\frac{{{{\left| {{\rm{Tr }}\,{\hat{{\boldsymbol{\tau}}} }} \right|}^{}}}}{{{{\left( {\left| {{\rm{Tr }}\,{\hat{{\boldsymbol{\tau}}} }} \right| + \beta \left\| {{\rm{dev}}\,{\hat{{\boldsymbol{\tau}}} }} \right\|} \right)}^{\frac{p-1}{{p+ 1}}}}}}, \qquad 
			&&{\hat\mu} = \frac{\Lambda }{{{{2( {{r Y_r}} )}^{\frac{1}{{p}}}}}}\frac{{\frac{1}{\beta} \left\| {{\rm{dev}}\,{\hat{{\boldsymbol{\tau}}} }} \right\|}}{{{{\left( {\left| {{\rm{Tr }}{\hat{{\boldsymbol{\tau}}}}} \right| + \beta \left\| {{\rm{dev}}\,{\hat{{\boldsymbol{\tau}}} }} \right\|} \right)}^{\frac{p-1}{{p + 1}}}}}}.
		\end{alignat}
		Then, $(\hat{k},\hat{\mu})$ solves the (sp-IMD) problem in \eqref{eq:sp-IMD}, while $\hat{\boldsymbol{\sigma}}$ and $\hat{\mbf{u}}$ are, respectively, the stress and displacement function in the optimal body.
	\end{enumerate}
\end{theorem}

\bigskip

\section{On the choice of the exponent $p$}

Thus far, the two methods vp-IMD and sp-IMD where analysed for a fixed exponent $p \in (1,\infty)$. This section is to study how the value of $p$ impacts the optimal design problem and its solutions. For instance, we have seen that the optimal moduli, stresses, and displacements lie in the following functional spaces:
\begin{equation}
    (\hat{k},\hat{\mu}) \in L^p(\O;\R^2), \qquad \hat{\boldsymbol{\tau}} \in L^r(\O;E_s^2), \qquad \hat{\mbf{v}} \in W^{1,r'}\,(\O;\R^d).
\end{equation}
It is thus clear that integrability of both the moduli and the stresses increases with $p$, recall that $r =\frac{2p}{p+1}$, which increasingly varies in the range $(1,2)$. On the other hand, $r' = \frac{2p}{p-1}$ which decreases with $p$: it tends to infinity when $p \to 1$ and to the value 2 when $p \to \infty$. Therefore, the regularity of the displacement $\hat{\mbf{v}}$ drops when $p$ increases.

Below we will show that the minimal compliances $\hat{\mathcal{C}}_{vp}$ and $\hat{\mathcal{C}}_{sp}$ also vary monotonically with the exponent $p$. To achieve that, the dependence of the upper bound $\Lambda$ in the cost condition on $p$ must be introduced: 
\begin{equation}
    \Lambda = \Lambda_p := \abs{\Omega}^{\frac{1}{p}} E_0.
\end{equation}
Accordingly, the cost conditions takes the form of the bound on the $p$-mean: in the case of the vp-IMD method,
\begin{equation}
    \label{eq:vpcost_E0}
    {\bigg( {\frac{1}{{\left| \Omega  \right|}}\int_\Omega  {\left[ {{{\left( {dk} \right)}^{p}} + {{\left( {{\beta^{2/p}}2\mu } \right)}^{p}}} \right]dx} } \bigg)^{\frac{1}{{p}}}} \le {E_0}	 
\end{equation}
and in the case of the sp-IMD method,
\begin{equation}
	\label{eq:spcost_E0}
	{\bigg( {\frac{1}{{\left| \Omega  \right|}}\int_\Omega  {{{\left( {dk + {\beta ^2}2\mu } \right)}^{p}}dx} } \bigg)^{\frac{1}{{p}}}} \le {E_0}.
\end{equation}

Furthermore, we will investigate the limit case when $p$ approaches one or infinity. For $p=1$ the original Isotropic Material Design problem will be recast.

\subsection{Monotonicity of the optimal compliance with respect to $p$}

In this subsection we shall show that the optimal compliance in the (sp-IMD) problem is non-decreasing with respect to the exponent $p$. For the (vp-IMD) problem, similar results hold true upon a proper rescaling. 

Beforhand, we recall a simple inequality holding in any probability space.
Let $\lambda$ be a probability on a space $X$, and let
$\mbf{f} : X \mapsto \R^n$ be a $\lambda$-measurable vector function. Then, for any $1 \leq p \leq q < \infty$, H\"{o}lder inequality yields monotonicity of the $L^p$ norm, i.e.
\begin{equation}
	\label{eq:Holder_mu}
	{\bigg( {\int_X  {{{\norm{\bf{f}}}^p} d\lambda} } \bigg)^{1/p}} \le {\bigg( {\int_X  {{{\norm{\bf{f}}}^q} d\lambda} } \bigg)^{1/q}}.
\end{equation}

\subsubsection{Monotonicity for the sp-IMD method}
Upon the choice  $X = \Omega$ and $d\lambda = \frac{1}{\abs{\Omega}} dx$,  \eqref{eq:Holder_mu} becomes
\begin{equation}
	\label{eq:Holder}
	{\bigg( {\frac{1}{{\left| \Omega  \right|}}\int_\Omega  {{{\norm{\bf{f}}}^p} dx} } \bigg)^{1/p}} \le {\bigg( {\frac{1}{{\left| \Omega  \right|}}\int_\Omega  {{{\norm{\bf{f}}}^q} dx} } \bigg)^{1/q}}.
\end{equation}  
If for $\mbf{f}$ we chose the real-valued function $d{k} + 2{\beta ^2}{\mu}$, what we get is 
\begin{equation}
	\label{eq:cost_comp_II}
	{\left( {\frac{1}{{\left| \Omega  \right|}}\int_\Omega  {\left[ {{{dk}} +  \beta^2{{ {2\mu } }}} \right]^p dx} } \right)^{\frac{1}{{p}}}}  \leq {\left( {\frac{1}{{\left| \Omega  \right|}}\int_\Omega  {\left[ {{{dk}} +  \beta^2{{ {2\mu } }}} \right]^q dx} } \right)^{\frac{1}{{q}}}}.
\end{equation}
With the above inequality at hand one can easily show the monotonicity result:
\begin{proposition}
    \label{prop:mono_sp}
    For any two exponents $1< p_1 \leq  p_2 < \infty $ the minimal compliances in the (sp-IMD) problem satisfy the inequality
    \begin{equation}
        \hat{\mathcal{C}}_{sp_1} \leq \hat{\mathcal{C}}_{sp_2}.
    \end{equation}
\end{proposition}
\begin{proof}
    For $i=1,2$ take the minimizers $(\hat{k}_i,\hat{\mu}_i)$ for the (sp-IMD) problem for the exponents $p=p_i$, that is $\hat{\mathcal{C}}_{sp_i} = \mathcal{C}(\hat{k}_i,\hat{\mu}_i)$, and $(\hat{k}_i,\hat{\mu}_i) \in{M_{sp_i}}(\Omega)$. After plugging $k = \hat{k}_2$, $\mu=\hat\mu_2$ into  \eqref{eq:cost_comp_II} with $p=p_1$, $q=p_2$, it turns out that also $(\hat{k}_2,\hat{\mu}_2) \in{M_{sp_1}}(\Omega)$. Indeed, the right hand side of the inequality is not larger than $E_0$ due to  $(\hat{k}_2,\hat{\mu}_2) \in{M_{sp_2}}(\Omega)$. By the admissibility of $(\hat{k}_2,\hat{\mu}_2)$ for $p=p_1$, there must hold $\hat{\mathcal{C}}_{sp_1} \leq \mathcal{C}(\hat{k}_2,\hat{\mu}_2) = \hat{\mathcal{C}}_{sp_2}$.
\end{proof}

\subsubsection{Monotonicity for the vp-IMD method}
This variant necessitates using \eqref{eq:Holder_mu} twice. If we take the finite space $X=\{1,2\}$ and then set $\lambda(\{1\}) = \frac{1}{1+ \beta^2 }$, $\lambda(\{2\}) =\frac{\beta^2}{1+\beta^2}$ and $f(1) = d k$, $f(2) = 2\lambda$, we get
\begin{equation}
	\label{eq:holder}
	\rho_p(dk,2\mu) \leq \rho_q(dk,2\mu),
\end{equation}
where the norm  $\rho_p$ on $\R^2$ is defined as
\begin{equation}
	\rho_p(dk,2\mu) :=	\left[ \frac{1}{1+ \beta^2 } {{{\left( {dk} \right)}^{p}} + \frac{\beta^2}{1+\beta^2}{{\left( {2\mu } \right)}^{p}}} \right]^{1/p}.
\end{equation}
Now let $k,\mu$ be functions on $\Omega$ rather than numbers. The chain of inequality follows:
\begin{align*}
	&{\left( {\frac{1}{(1+\beta^2){\left| \Omega  \right|}}\int_\Omega  {\left[ {{{\left( {dk} \right)}^{p}} +  \beta^2{{\left( {2\mu } \right)}^{p}}} \right]dx} } \right)^{\frac{1}{{p}}}} \\ 
	& \qquad = 		{\left( {\frac{1}{{\left| \Omega  \right|}}\int_\Omega  {\big(\rho_p(dk,2\mu)\big)^p dx} } \right)^{\frac{1}{{p}}}}
	\leq 	{\left( {\frac{1}{{\left| \Omega  \right|}}\int_\Omega  {\big(\rho_q(dk,2\mu)\big)^p  dx} } \right)^{\frac{1}{{p}}}} \leq 	{\left( {\frac{1}{{\left| \Omega  \right|}}\int_\Omega  {\big(\rho_q(dk,2\mu)\big)^q  dx} } \right)^{\frac{1}{{q}}}} \\
	& \qquad ={\left( {\frac{1}{(1+\beta^2){\left| \Omega  \right|}}\int_\Omega  {\left[ {{{\left( {dk} \right)}^{q}} +  \beta^2{{\left( {2\mu } \right)}^{q}}} \right]dx} } \right)^{\frac{1}{{q}}}}.
\end{align*}
Above, the first inequality is simply \eqref{eq:holder}.
The second inequality is \eqref{eq:Holder} for $\mbf{f} = (dk,2\mu)$ and for $\norm{\argu} = \norm{\argu}_q$. What follows is monotonicity not of the minimal compliance itself but of the minimal compliance scaled by the factor $1/(1+\beta^2)^{1/p}$. Indeed, with the help of the fact that
$\mathcal{C}(tk,t\mu) = \frac{1}{t}\,\mathcal{C}(k,\mu)$ for $t>0$, one can now reproduce the arguments of the proof of Proposition \ref{prop:mono_sp} to deduce the following:
\begin{proposition}
    \label{prop:mono_vp}
    For any two exponents $1< p_1 \leq  p_2 < \infty $ the minimal compliances in the (vp-IMD) problem satisfy the inequality
    \begin{equation}
        \frac{\hat{\mathcal{C}}_{vp_1}}{(1+\beta^2)^{1/p_1}}   \leq \frac{\hat{\mathcal{C}}_{vp_2}}{(1+\beta^2)^{1/p_2}}.
    \end{equation}
\end{proposition}

\subsection{Recovering the original IMD method as $p \to 1$}
\label{ssec:IMD_recovery}

As mentioned in the Introduction, the forgoing theory cannot be directly extended to $p=1$, which would correspond to the original Isotropic Material Design problem (IMD) initiated in the work \cite{czarnecki2015a}, see also \cite{czarnecki2015b,bolbotowski2021,bolbotowski2022b}. For $p=1$ the set $M_{vp}(\Omega)$ is no longer weakly compact as a subset of the space $L^1(\Omega;\R^2)$. The same goes for  $M_{sp}(\Omega)$, note that the two sets are equal for $p=1$. According to the work \cite{bolbotowski2022b} the IMD problem calls for a measure-theoretic approach. We recapitulate the formulation and main results below, the reader is referred to \cite{czarnecki2015a} and \cite{bolbotowski2022a} for details.

The set of admissible moduli for $p = 1$ consists of pairs of positive measure whose total mass is controlled:
\begin{equation}
	M_1(\Omega ) := \left\{ (k,\mu) \in (\Mes_+(\Ob))^2 \ \left\vert \ \int_\Ob d\,k(dx) + \int_\Ob \beta^2 2\mu(dx) \leq \Lambda_1 \right. \right\},
\end{equation}
where $\Lambda_1 =  E_0 \abs{\Omega}$.
Note that $M_1(\Omega)$ contains elements of $M_{vp}(\Omega)$ for $p=1$ in the following sense. To every function $(k_1,\mu_1) \in L^1(\Omega;\R_+)^2$ one can assign the measures $k=k_1 dx$, $\mu = \mu_1 dx$ where $dx$ stands for the Lebesgue measure on $\Rd$. We say that such measures $k,\mu$ are absolutely continuous. The set $M_1(\Omega)$, however, is much bigger. In particular, it includes lower dimensional geometrical measures that give non-zero measure to lower-dimensional sets such as curves ($d=2,3$) or surfaces ($d=3$). Such measures correspond to elastic stiffeners in the form of bars and shells.

The Isotropic Material Design problem is similar to (vp-IMD) or (sp-IMD):
\begin{equation}
    \label{eq:IMD}
    (\mathrm{\text{IMD}}) \qquad \qquad\qquad  \qquad\qquad  \hat{\mathcal{C}}_1 =  \mathop {\min }\limits_{(k,\mu ) \in M_1\left( \Omega  \right)} \mathcal{C}(k,\mu ). \qquad \qquad  \qquad\qquad  \qquad\qquad 
\end{equation}
The definition \eqref{eq:comp} of the compliance functional $\mathcal{C}(k,\mu)$ requires adjustment to measures, see Eq. (3.9) in \cite{bolbotowski2022b}. In fact, one can use the dual definition \eqref{eq:comp_disp} above directly, provided that the supremum is taken with respect to smooth functions. 

The (IMD) problem also admits its stress-based formulation that is a natural extension to $p=1$ of the problems \eqref{eq:primal_vp} or \eqref{eq:primal_sp} introduced in this paper:
\begin{equation}
\label{eq:Z1}
	{Z}_1 :=  \mathop {\min }\limits_{{\boldsymbol{\tau}} \in \Sigma\Mes_{{f}} \left( \Omega  \right)}  \int_\Ob  {{{\Big( {{{\left| {{\rm{Tr }}\,{{\boldsymbol{\tau}(dx)}}} \right|}^{}} + \beta {{\left\| {{\rm{dev}}\,{{\boldsymbol{\tau}(dx)} }} \right\|}^{}}} \Big)}}}.
\end{equation}
Here,
$\Sigma\Mes_f(\Omega)$ is the set of tensor-valued measures $\boldsymbol\tau \in \Mes(\Ob;E_s^2)$ such that $\int_\Ob  {\hat{\boldsymbol{\tau}}(dx) \cdot \,} {\boldsymbol{\eps}}\left( {\bf{v}} \right) = f(\mbf{v})$ for any smooth function $\mbf{v}$ that is zero on $\Gamma_0$. The integral above is understood in the sense of convex positively one-homogeneous functionals on measures \cite{goffman1964}. Optimal solutions of (IMD) and of \eqref{eq:Z1} are linked by the formulae
\begin{equation}
	\label{eq:opt_kmu_tau_IMD}
	d{\hat{k} } = \frac{\Lambda_1 }{{{{ {{Z_1}}}}}} \left| {{\rm{Tr }}\,{\hat{{\boldsymbol{\tau}}} }} \right|
    ,\qquad 2 \hat{\mu  } = \frac{\Lambda_1 }{{{{ {{Z_1}}}}}} \frac{1}{\beta} \left\| {{\rm{dev}}\,{\hat{{\boldsymbol{\tau}}} }} \right\|.
\end{equation} 
Note that these equations ought to be understood in the sense of measures. Up to this point, we can see that the above formulae can be deduced from either of Theorems \ref{thm:stress_vp} or \ref{thm:stress_sp} simply by putting $p=1$.

Guessing the displacement-based formulation of (IMD) is less straightforward. First, it requires introduction of
$\mbf{V}^{\infty}(\Omega)$, the space of vectorial functions $\mbf{v}$ that belong to $W^{1,q}(\Omega;\Rd)$ for any $q \in [1,\infty)$ and which are zero on $\Gamma_0$. It is important to emphasize that we cannot imply that $\mbf{v} \in W^{1,\infty}(\Omega;\Rd)$, i.e. that $\mbf{v}$ is Lipschitz continuous. This is related to the lack of Korn's inequality for $L^\infty$-norms, see \cite{demengel1985}. The dual displacement-based formulation involves a locking condition:
\begin{equation}
    \label{eq:dual_Z1}
	 Z_1 =  \max_{\mbf{v} \in \mbf{V}^{\infty}(\Omega) } \left\{ {f}(\mbf{v}) \ \bigg\vert \ \max\big\{ \abs{\mathrm{Tr}\,\boldsymbol{\eps}(\mbf{v})} , \tfrac{1}{\beta}\norm{\mathrm{dev} \, \boldsymbol{\eps}(\mbf{v})} \big\} \leq 1 \ \ \text{a.e. in $\Omega$} \right\}.
\end{equation}

All three variational problems: \eqref{eq:IMD}, \eqref{eq:Z1}, \eqref{eq:dual_Z1} have solutions, and in general they are non-unique. One comes across a difficulty when writing down the counterpart of the systems \eqref{eq:non-linear_system_vp} or \eqref{eq:non-linear_system_sp}. Formally, it reads:
\begin{equation}
	\label{eq:non-linear_system_IMD}
	\begin{cases}
		(1) & \int_\Ob  {\hat{\boldsymbol{\tau}}(dx) \cdot \,} {\boldsymbol{\eps}}\left( {\bf{v}} \right)  = {f}(\mbf{v}) \qquad \forall \,  \mbf{v} \in C^1(\Ob;\Rd), \ \ \mbf{v} = 0 \text{ on } \Gamma_0,\\
		(2) & \hat{\mbf{v}} \in \mbf{V}^{\infty}(\Omega),\\
		(3.a) & " \ \mathrm{Tr}\,\boldsymbol{\eps}(\hat{\mbf{v}}) \in  \mathcal{N} (\mathrm{Tr}\,\hat{\boldsymbol{\tau}}) \quad \text{in $\Ob$} \ ",\\
		(3.b) &  " \ \mathrm{dev}\,\boldsymbol{\eps}(\hat{\mbf{v}}) \in  \beta\, \mathcal{N}(\mathrm{dev}\,\hat{\boldsymbol{\tau}}) \quad \text{in $\Ob$} \ ".
	\end{cases}
\end{equation}
The stress-strain relations (3.a), (3.b) are  meaningful in a rigorous manner only if $\hat{\mbf{v}}$ is of class $C^1$, hence the quotation marks. First, since $\hat{\boldsymbol{\tau}}$ is a measure, let us explain that the quotients $\frac{\mathrm{Tr}\,\hat{\boldsymbol{\tau}}}{\abs{\mathrm{Tr}\,\hat{\boldsymbol{\tau}}}}$, $\frac{\mathrm{dev}\,\hat{\boldsymbol{\tau}}}{\norm{\mathrm{dev}\,\hat{\boldsymbol{\tau}}}}$ entering the definition $\mathcal{N} (\mathrm{Tr}\,\hat{\boldsymbol{\tau}})$, $\mathcal{N}(\mathrm{dev}\,\hat{\boldsymbol{\tau}})$ (see \eqref{eq:N}) should be understood in the sense of Radon-Nikodym derivative, which are well defined almost everywhere with respect to $\abs{\mathrm{Tr}\,\hat{\boldsymbol{\tau}}}$ and $\norm{\mathrm{dev}\,\hat{\boldsymbol{\tau}}}$, respectively. Accordingly, conditions (3.a), (3.b) should be enforced a.e. with respect to those measures. Meanwhile, since $\mbf{v}$ is  merely in the Sobolev space, the tensor $\boldsymbol{\eps}(\hat{\mbf{v}})$ is \textit{a priori} defined a.e. with respect to the Lebesgue measure. This is not enough when $\abs{\mathrm{Tr}\,\hat{\boldsymbol{\tau}}}$, $\norm{\mathrm{dev}\,\hat{\boldsymbol{\tau}}}$ charge lower-dimensional sets, for instance.

The foregoing issues can be addressed by employing the \textit{measure-tangential calculus} first proposed in \cite{bouchitte1997}. For the positive measure $\lambda = \abs{\mathrm{Tr}\,\hat{\boldsymbol{\tau}}}+\norm{\mathrm{dev}\,\hat{\boldsymbol{\tau}}}$ and any ${\mbf{v}}$ feasible in \eqref{eq:dual_Z1} one can define a $\lambda$-tangential strain $\boldsymbol{\eps}_\lambda(\mbf{v})$ that is an element of $L^\infty_\lambda(\Ob;E_s^2)$. In particular, $\boldsymbol{\eps}_\lambda(\mbf{v})$ is meaningful $\lambda$-a.e., which paves the way to the rigorous form of the inclusions (3.a), (3.b).

\smallskip

The short summary of the original IMD method demonstrates that it touches a number of subtle mathematical issues due to the inherent measure-theoretic setting. One of the premises for the current paper is to replace the problem (IMD) with the more accessible problems (vp-IMD) or (sp-IMD) with $p =1+\eps$ for small $\eps>0$. This way, $\eps$ serves as a regularizing parameter for the problem (IMD). It facilitates not only easier analysis but also the numerical methods: discretizing the space of measures is otherwise very difficult. Similar regularizations have been already studied in the literature, see \cite{golay2001,barrett2007}, for example. Therein, however, the regularization was introduced directly at the level of mutually dual problem of the type \eqref{eq:Z1}, \eqref{eq:dual_Z1}. One of the novelties of this work lies in recovering a similar type of regularization through modification of the cost in the original optimal design problem.

In order to justify treating solutions of (vp-IMD) or (sp-IMD) problem as an approximation of solutions to the problem (IMD), one would have to show convergence with $p \to 1$ (for instance, in the sense of $\Gamma$-convergence, see \cite{braides2002}). An analysis of this type has been successful in \cite{barrett2007}. Being rather technical, this issue falls out of the scope of this work. Thus, below we write down the statement as a conjecture. Note that the right notion of convergence in the space of measure is the weak-* convergence: we write that $k_n \,\weakstar\, k$ in $\Mes(\Ob;\R)$ if $\int_\Ob \varphi \,k_n(dx) \to \int_\Ob \varphi \,k(dx)$ for every continuous function $\varphi :\Ob \to \R$.

\begin{conjecture}
    \label{conj:p=1}
    Take a decreasing sequence $\{p_n\}_{n=1}^\infty$ such that $1 < p_n <\infty$ and $p_n \to 1$. Let $(\hat{k}_n,\hat{\mu}_n)$, $\hat{\boldsymbol{\tau}}_n$, and $\hat{\mbf{v}}_n$ be the unique solutions of, respectively, the (vp-IMD) problem, the stress-based problem \eqref{eq:primal_vp}, and the dual displacement-based problem \eqref{eq:dual_vp}.
    
    Then, there exist measures $(\hat{k},\hat{\mu}) \in (\Mes_+(\Ob))^2$, $\hat{\boldsymbol{\tau}} \in \Mes(\Ob;E_s^2)$ and a function $\hat{\mbf{v}} \in \mbf{V}^\infty(\Omega)$ such that the following assertions hold true:
    \begin{enumerate} [label={(\roman*)}]
        \item $(\hat{k},\hat{\mu})$, $\hat{\boldsymbol{\tau}}$, and $\hat{\mbf{v}}$ solve the problems \eqref{eq:IMD}, \eqref{eq:Z1}, and \eqref{eq:dual_Z1}, respectively;
        \item up to selecting a subsequence, we have the convergences:
        \begin{align}
            (\hat{k}_n dx,\hat{\mu}_n dx) \ & \, \weakstar \, \ (\hat{k},\hat{\mu}) \ \ \text{ in $(\Mes_+(\Ob))^2$},\\
            \hat{\boldsymbol{\tau}}_n  dx\ & \, \weakstar \, \ \hat{\boldsymbol{\tau}} \ \ \text{ in $\Mes(\Ob;E_s^2)$}, \\
            \hat{\mbf{v}}_n \ &\to \ \hat{\mbf{v}}  \ \ \text{uniformly on $\Ob$}, \\
            \hat{\mathcal{C}}_{vp_n} \ & \to \ \hat{\mathcal{C}}_{1},
        \end{align}
        where by uniform convergence we understand that $\sup_{x\in \Ob} \norm{\hat{\mbf{v}}_n(x) - \hat{\mbf{v}}(x)} \to 0$.
    \end{enumerate}
    Analogous conjecture can be postulated when $(\hat{k}_n,\hat{\mu}_n)$, $\hat{\boldsymbol{\tau}}_n$, and $\hat{\mbf{v}}_n$ solve (sp-IMD) and the related stress-based and displacement-based problems \eqref{eq:primal_sp} and \eqref{eq:dual_sp}. 

\end{conjecture}

\subsection{The limit case of $ p = \infty$}
\label{ssec:p=infty}
After studying the case of the exponent $p=1$, it is now natural to investigate the other limit scenario: when $p=\infty$. Since the unit ball in $L^\infty$ is weakly-* compact, there is no need to go beyond the mathematical framework that served its purpose herein for $p \in (1,\infty)$. In fact, for both methods vp-IMD and sp-IMD it suffices to adapt the above results by taking the limit $p \to \infty$ in the formulations and formulae. In particular, observe that we should operate with $r = \lim_{p\to\infty} \frac{2p}{p+1}=2$, and $r'=2$ also.

Below, we shall present the most crucial conclusions for the two methods in this limit case, starting from the sp-IMD method since, as it will unravel, the (vp-IMD) method becomes trivial for $p = \infty$.

\subsubsection{The sp-IMD method for $p=\infty$}
\label{sssec:p_infty_spIMD}
The set of admissible moduli for $p = \infty$ is as follows:
\begin{equation}
	M_{s\infty}(\Omega ) := \Big\{ (k,\mu) \in (L^\infty(\Omega;\R_+))^2 \ \Big\vert \ \norm{dk+\beta^2 
 2\mu}_\infty\leq \Lambda_\infty  \Big\},
\end{equation}
where $\Lambda_\infty = \lim_{p \to \infty} \abs{\Omega}^{1/p} E_0 = E_0$, and $\norm{\argu}_\infty$ is the $L^\infty$-norm being the essential supremum on $\O$.
This is a closed subset of a closed ball in the space $L^\infty(\Omega;\R^2)$ and, thus, it is a weakly-* compact set. Therefore, almost identical arguments to those put forth in this paper for $p\in (1,\infty)$ allow to conclude that the (sp-IMD) problem for $p=\infty$ is also well posed. More precisely, there exists a solution $(\hat{k},\hat{\mu})$, yet it is possible that it is non-unique.

Let us remark that the compliance functional is decreasing, in the sense that $\mathcal{C}(k_1,\mu_1) \geq \mathcal{C}(k_2,\mu_2)$ whenever $k_1 \leq k_2$, $\mu_1 \leq \mu_2$ pointwisely. This fact follows directly from the definition \eqref{eq:comp}. It reduces the (sp-IMD) problem to looking merely for the partitioning between the bulk and shear moduli at each point, which together give $dk+\beta^2 2\mu = E_0$ a.e. in $\O$.

To find this proportion one can again turn to the stressed-based auxiliary problem. It suffices to choose $r=2$ in the problem \eqref{eq:primal_sp}:
	\begin{align}
		{Y_2} =  \mathop {\min }\limits_{{\boldsymbol{\tau}} \in \Sigma_{{f}} \left( \Omega  \right)} \frac{1}{2} \int_\Omega  {{{\Big( {{{\left| {{\rm{Tr }}\,{{\boldsymbol{\tau}}}} \right|}^{}} + \beta {{\left\| {{\rm{dev}}\,{{\boldsymbol{\tau}} }} \right\|}^{}}} \Big)}^{2}}dx}.
	\end{align}	
This problem is again well-posed. Its solution yields the optimal partition at every point at which $\hat{\boldsymbol{\tau}}$ is non-zero:
\begin{equation}
    \label{eq:opt_kmu_tau_infty}
	d{\hat{k} } = E_0\, \frac{{{{\left| {{\rm{Tr }}\,{\hat{{\boldsymbol{\tau}}} }} \right|}^{}}}}{{{\left| {{\rm{Tr }}\,{\hat{{\boldsymbol{\tau}}} }} \right| + \beta \left\| {{\rm{dev}}\,{\hat{{\boldsymbol{\tau}}} }} \right\|} }},\qquad 2 \hat{\mu  } = E_0\,\frac{{\frac{1}{\beta} \left\| {{\rm{dev}}\,{\hat{{\boldsymbol{\tau}}} }} \right\|}}{{{{ {\left| {{\rm{Tr }}{\hat{{\boldsymbol{\tau}}}}} \right| + \beta \left\| {{\rm{dev}}\,{\hat{{\boldsymbol{\tau}}} }} \right\|} }}}}.
\end{equation}
At points $x$ where $\hat{\boldsymbol{\tau}}(x)=0$ any pair $\big(\hat{k}(x),\hat{\mu}(x)\big)$ satisfying $d\hat{k}(x)+\beta^2 
 2\hat\mu(x) \leq E $ will do; in particular, one can put $\big(\hat{k}(x),\hat{\mu}(x)\big) = (0,0)$.
 
Similarly, the dual displacement-based formulation is non other than \eqref{eq:dual_sp} with $r'=2$. For the system of non-linear inclusions one puts $r = 2$ in \eqref{eq:non-linear_system_sp}. This affects the stress-strain relations which now read
    \begin{equation}
	\begin{cases}
		(3.a) & \mathrm{Tr}\,\boldsymbol{\eps}(\hat{\mbf{v}}) \in \big(\abs{\mathrm{Tr}\,\hat{\boldsymbol{\tau}}}+\beta\,\norm{\mathrm{dev}\,\hat{\boldsymbol{\tau}}}\big) \mathcal{N} (\mathrm{Tr}\,\hat{\boldsymbol{\tau}}) \qquad \text{a.e. in $\Omega$},\\
		(3.b) & \mathrm{dev}\,\boldsymbol{\eps}(\hat{\mbf{v}}) \in  \big(\abs{\mathrm{Tr}\,\hat{\boldsymbol{\tau}}}+\beta\,\norm{\mathrm{dev}\,\hat{\boldsymbol{\tau}}}\big) \beta\, \mathcal{N}(\mathrm{dev}\,\hat{\boldsymbol{\tau}}) \qquad \text{a.e. in $\Omega$}.
	\end{cases}
\end{equation}

Let us remark that, similarly as for $p \to 1$, one can expect that the (sp-IMD) problem converges to the above formulation when $p \to \infty$. More precisely, we can conjecture convergences analogous to those in Conjecture \ref{conj:p=1}, with $(\hat{k}_n,\hat{\mu}_n)$ now converging weakly-* in $L^\infty$, $\hat{\boldsymbol{\tau}}_n$ weakly in $L^2$, and $\hat{\mbf{v}}_n$ weakly in $H^1 =W^{1,2}$.
\bigskip

\subsubsection{Trivialisation of the vp-IMD method for $p=\infty$}
\label{sssec:trivial}
Pushing $p$ to infinity in the vp-IMD method leads to the following admissible set of moduli
\begin{equation}
	M_{v\infty}(\Omega ) := \Big\{ (k,\mu) \in (L^\infty(\Omega;\R_+))^2 \ \Big\vert \ \norm{\max\{dk,2\mu\}}_\infty\leq \Lambda_\infty  \Big\}.
\end{equation}
where again $\Lambda_\infty = E_0$. Thanks to the aforementioned monotonicity of the compliance functional $(k,\mu) \mapsto \mathcal{C}(k,\mu)$, immediately we infer that
\begin{equation}
    \label{eq:trivial_sol}
    d\hat{k} = E_0, \qquad 2 \hat{\mu} =  E_0
\end{equation}
is the universal solution for all admissible loads, in this case $f$ that are continuous functionals on $W^{1,2}(\Omega;\R^d)$. Therefore, the (vp-IMD) problem for the limit case $p=\infty$ turns trivial, and it is not a meaningful optimal design problem.

One can also observe that for $r=2$ the stress-based formulation of (vp-IMD) reduces to minimizing the quadratic functional:
\begin{align}
	{Z_2} =  \mathop {\min }\limits_{{\boldsymbol{\tau}} \in \Sigma_{{f}} \left( \Omega  \right)} \frac{1}{2} \int_\O \boldsymbol{\tau} \cdot \boldsymbol{\tau}  \, dx.
	\end{align}	
It is exactly the stress-based formulation of the linear elasticity problem of the homogeneous body with the stiffness tensor being the identity. This is directly related to the fact that the "optimal" body according to  \eqref{eq:trivial_sol} is also homogeneous, and the stiffness tensor is identity multiplied by $E_0$.

The homogeneous moduli distribution \eqref{eq:trivial_sol} is a solution of the problem (vp-IMD) problem for $p=\infty$, but it is not the only one. The unique solution $\hat{\boldsymbol{\tau}}$ furnishes the full characterization of the set of optimal moduli. $(\hat{k},\hat{\mu})$ is optimal if and only if:
\begin{itemize}
    \item[(i)] $\hat{k}(x) =\frac{E_0}{d}$ for a.e. $x$ where $\mathrm{Tr}\,\hat{\boldsymbol{\tau}}(x) \neq 0$;
    \item[(ii)] $\hat{\mu}(x) = \frac{E_0}{2}$ for a.e. $x$ where $\mathrm{dev}\,\hat{\boldsymbol{\tau}}(x) \neq 0$.
\end{itemize}
At points where $\mathrm{Tr}\,\hat{\boldsymbol{\tau}}(x)$ vanishes one can choose any $\hat{k}(x) \in [0,\frac{E_0}{d}]$, and, similarly, whenever $\mathrm{dev}\,\hat{\boldsymbol{\tau}}(x) =0$, any $\hat{\mu}(x) \in [0,\frac{E_0}{2}]$ is optimal. These comments are not without value. Assuming again that the (vp-IMD) problem for $p <\infty$ converges to the trivial problem described here, one can expect that the solutions for large $p$ should approximately follow the rules (i), (ii), instead of being close to the trivial homogeneous solution \eqref{eq:trivial_sol}. This will be confirmed by the numerical solutions.

\bigskip

\section{Numerical simulations}

\subsection{Implementation}
\label{ssec:implementation}
The strategy for  the implementation of the vp-IMD and sp-IMD method is to discretize the stress-based formulations \eqref{eq:primal_vp} and \eqref{eq:primal_sp}. We will be thus tackling the convex minimization problem
\begin{equation}
	\label{eq:argmin}
 \mathop { \min }\limits_{{\boldsymbol{\tau }} \in \Sigma_f \left( \Omega  \right)} \Phi \left( {\boldsymbol{\tau }} \right)
\end{equation}
where, for the vp-IMD method, 
\begin{equation}
	\label{eq:Phi_1}
	\Phi ( {\boldsymbol{\tau }} ) = 	\widetilde{F}_{vp}( \boldsymbol{\tau}) =\frac{1}{r} \int_\Omega  {\Big( {{{\left| {{\rm{Tr }}\,{\boldsymbol{\tau}}} \right|}^{r}} + \beta^{2-r}{{\left\| {{\rm{dev}}\,{\boldsymbol{\tau}}} \right\|}^{r}}} \Big)dx},
\end{equation}
and, for the sp-IMD method,
\begin{equation}
	\label{eq:Phi_2}
	\Phi ( {\boldsymbol{\tau }} ) = 	\widetilde{F}_{sp}( \boldsymbol{\tau}) = \frac{1}{r} \int_\Omega  {{{\Big( {{{\left| {{\rm{Tr }}\,{{\boldsymbol{\tau}}}} \right|}^{}} + \beta {{\left\| {{\rm{dev}}\,{{\boldsymbol{\tau}} }} \right\|}^{}}} \Big)}^{r}}dx},
\end{equation}
where $r=\frac{2p}{p+1}$ and $ \beta  = {\left( {\frac{1}{2}d(d + 1) - 1} \right)^{\frac{1}{2}}}$. On having solved the problems above one can construct the
optimal distribution of elastic moduli $\hat{k},\hat{\mu} \in L^p(\Omega;\R_+)^2$ via the formulae in Theorems \ref{thm:stress_vp} and \ref{thm:stress_sp}, respectively.

The numerical method adopted follows the previous works by the second and third author, see e.g. \cite{czarnecki2015a,czarnecki2021isotropic,czarnecki2024}. The reader is encouraged to visit these papers for the details of the implementation, whilst below we shall convey the general ideas. We shall restrict ourselves to 2D problems, yet the method can be universally extended to three dimensions.

The core of the computational technique is the construction of the finite dimensional set ${\Sigma^h_f}( \Omega  )$ approximating the set of statically admissible stresses ${\Sigma_f}( \Omega )$.  The parameter $h$ indexes the mesh density. The design domain $\Omega  \cong \bigcup_{k} {{\Omega _k}} $ is meshed into triangular first-order Lagrange finite elements (note that there is no obstruction to use different elements, e.g., quadrilaterals, cf. \cite{czarnecki2021isotropic}). That is to say, the stress functions $\boldsymbol{\tau}$ are assumed continuous and element-wise affine. More formally, for an $n$-dimensional vector of nodal stresses $\mbf{T}$ we have:
\begin{equation}
    \label{eq:tau_interpol}
    \boldsymbol{\tau}(x) = \boldsymbol{\tau}_k\big( \boldsymbol{\psi}_k^{-1}(x)\big) \quad \forall\,x \in \O_k, \qquad \text{where} \qquad   \boldsymbol{\tau}_k(\xi) := \sum_{l=1}^{{\mathrm{dof}}} \boldsymbol{\eta}_l(\xi) \big(\mbf{A}(k) \mbf{T}\big)_l \quad \forall\,\xi \in \omega.
\end{equation}
Above, the 2D simplex $\omega \subset \R^2$ is the master element, and $\boldsymbol{\psi}_k: \xi \in \omega \mapsto  x \in\Omega_k$ is the affine isomorphism for the $k$-th finite element. For every $k$, $\mbf{A}(k)$ is the $\mathrm{dof} \times n$ allocation matrix consisting of ones and zeros, where $\mathrm{dof}$ is the number of degrees of freedom corresponding to a single element. Here, $\mathrm{dof} = 3 \times 3 = 9$, i.e., there are 3 stress components per each of the  nodes (for quadrilateral elements $ \mathrm{dof} = 3 \times 4$).
Finally, $\boldsymbol\eta_j:\omega \to E_s^2$ are the affine tensor-valued shape functions. 

To construct the set ${\Sigma^h_f}( \Omega)$, for stress functions of the form \eqref{eq:tau_interpol} the equilibrium equations \eqref{eq:stat_adm} must be enforced in an approximate sense. To that aim we introduce the $m$-dimensional set  $\mbf{V}_h(\Omega)$ of those virtual displacements $\mbf{v}:\Ob \to\R^2$ which are element-wise affine, continuous on $\Ob$ ($\mbf{v}$ are of the form similar to $\boldsymbol{\tau}$ in \eqref{eq:tau_interpol}), and are zero on $\Gamma_0$ (or on its approximation by polygonal chains). Testing \eqref{eq:stat_adm} against such displacements $\mbf{v}\in \mbf{V}_h(\O)$ leads to the linear system
\begin{equation}
	\label{eq:BTQ}
	\mbf{B} \mbf{T} = \mbf{Q}.
\end{equation}						
Above $\mbf{B}$ is a rectangular $m \times n$ matrix ($m < n$), while $\mbf{Q} \in\R^m$ is the vector whose entries are the virtual works $Q_j = f(\mbf{v}_j^h)$ on the displacements $\mbf{v}_j^h\in \mbf{V}_h(\O)$ that correspond to unit nodal displacements at the subsequent nodes of the mesh.
Readily, the set ${\Sigma^h_f}( \Omega)$ consists of $\boldsymbol{\tau}$ of the form \eqref{eq:tau_interpol} with $\mbf{T}$ satisfying \eqref{eq:BTQ}.

Our goal is to reach a finite dimensional unconstrained convex optimization problem. This is possible through the following representation of $\mbf{T}$ solving the linear system \eqref{eq:BTQ}:
\begin{equation}
    \label{eq:T_rep}
    {\mbf{T}} = \mathring{\mbf{T}}  +  \mbf{N}\boldsymbol{\alpha} ,\qquad {\boldsymbol{\alpha }} \in {\R^s},
\end{equation}
where $\mathring{\mbf{T}}$ is a fixed solution of the system \eqref{eq:BTQ}, while
the columns of the $n \times s$ matrix $\mbf{N}$ span the $s$-dimensional null space of the matrix $\mbf{B}$, see e.g. \cite{czarnecki2024}. It is thus justified to introduce the notations $\boldsymbol{\tau}(x,\boldsymbol{\alpha})$ and $ \boldsymbol{\tau}_k(\xi,\boldsymbol{\alpha})$ to be understood via the interpolation \eqref{eq:tau_interpol} with $\mbf{T}$ of the form \eqref{eq:T_rep}.

We are now in a position to pose the discretized variant of the formulation \eqref{eq:argmin} being an unconstrained convex optimization problem:
\begin{equation}
\label{eq:discrete}
\mathop {\min }\limits_{{\boldsymbol{\alpha }} \in {\R^s}} \Phi_h ( {\boldsymbol{\alpha }} ), 
\end{equation}
where the functional  $\Phi_h :{\R^s} \to \R_+$ is an approximation of $\Phi$ via numerical integration: 
\begin{equation}
	\label{eq:discrete_Phi}
	\Phi_h ( {\boldsymbol{\alpha }} ) = \sum\limits_k {\sum\limits_{e=1}^4 {\frac{w_e}{r} \, \opnorm{{{\boldsymbol{\tau }_k}( {\zeta_e,\boldsymbol{\alpha} } )} }^r \left| {\det {{\mbf{J}}_k}} \right|} }
\end{equation}
where $\opnorm{\argu}$ is the norm on $E_s^2$ introduced in \eqref{eq:norm_vp} or \eqref{eq:norm_sp}, for the vp-IMD or sp-IMD method, respectively.
The quantities $w_e$ are weights at the four Gauss points $\zeta_e  \in \omega $, whereas ${{\mbf{J}}_k}$ is the constant Jacobian matrix of the affine geometrical mapping $\boldsymbol{\psi}_k:\omega  \mapsto {\Omega _k}$.

Since the mappings $\boldsymbol{\alpha}\mapsto \boldsymbol{\tau}_k(\xi,\boldsymbol{\alpha})$ are affine, convexity of the functional $\Phi$ follows, and the well-posedness of the discrete problem \eqref{eq:discrete} with it.
Any solution $\hat{\boldsymbol{\alpha}}$ furnishes the interpolated stress function $\hat{\boldsymbol{\tau}}_h := \boldsymbol{\tau}(\argu,\hat{\boldsymbol{\alpha}}) \in \Sigma_f^h(\Omega)$. In turn, through the formulae in Theorems \ref{thm:stress_vp} and \ref{thm:stress_sp}, this transfers to an approximation $(\hat{k}_h,\hat{\mu}_h)$ of moduli solving the (vp-IMD) or (sp-IMD) problem. Let us stress, that convergence of the numerical method lies outside of the scope of this paper.

In addition, let us observe that in the case of the vp-IMD method the function $\opnorm{\argu}^r:E_s^2 \to \R_+$ is strictly convex when $p>1$ (then $r=\frac{2p}{p+1} >1$ as well), and so is the function $\Phi_h$ in this case. Accordingly, the discretized problem \eqref{eq:discrete} has a unique solution in the vp-IMD variant for $p>1$. This cannot be asserted for the sp-IMD method, or when $p=1$, for any of the methods.

Any algorithm of non-linear mathematical programming can be implemented to find a solution $\hat{\boldsymbol{\alpha}}$, e.g. a gradient-oriented BFGS method, see \cite{press1996}. To that aim, the gradient of $\Phi_h$ must be calculated:
\begin{equation}
\frac{\partial \Phi_h}{\partial \alpha_i} ( {\boldsymbol{\alpha }} ) = \sum\limits_k {\sum\limits_{e=1}^4 {w_e \, g_{k,i} ( {\zeta_{e} ,{\boldsymbol{\alpha }}}  )\left| {\det {{\mbf{J}}_k}} \right|} }, \qquad i = 1,2,...,s,
\end{equation}
where, for the vp-IMD method,
$$ g_{k,i}= {\left| {{\rm{Tr }}\,{{\boldsymbol{\tau }}_k}} \right|^{r - 2}}\,{\rm{Tr }}\,{{\boldsymbol{\tau }}_k}\;{\rm{Tr }}\,\tfrac{{\partial {{\boldsymbol{\tau }}_k}}}{{\partial {\alpha _i}}} + {\left\| {{\beta\,\rm{dev }}\,{{\boldsymbol{\tau }}_k}} \right\|^{r - 2}} \, \mathrm{dev}\, {\boldsymbol{\tau }}_k \cdot \mathrm{dev} \,\tfrac{\partial {\boldsymbol{\tau }}_k}{\partial \alpha_i},$$											
or, for the sp-IMD method,
$$ g_{k,i} ={{\big( {\left| \rm{Tr }\,{\boldsymbol{\tau }}_k \right|} + \beta \left\| \rm{dev }\,{\boldsymbol{\tau }}_k \right\| \big)}^{r-1}}\left( \frac{\rm{Tr }\,{\boldsymbol{\tau }}_k\, \rm{Tr }\frac{\partial {\boldsymbol{\tau }}_k}{\partial {{\alpha }_{i}}}}{\left| \rm{Tr }{\boldsymbol{\tau }}_k \right|}+\frac{\beta \,\mathrm{dev}\, {\boldsymbol{\tau }}_k \cdot \mathrm{dev} \,\tfrac{\partial {\boldsymbol{\tau }}_k}{\partial \alpha_i} }{\left\| \rm{dev }{\boldsymbol{\tau }}_k \right\|} \right).$$       
Note that, above, $g_{k,i} = g_{k,i}(\xi,\boldsymbol{\alpha})$ and $\boldsymbol{\tau}_k = \boldsymbol{\tau}_k(\xi,\boldsymbol{\alpha})$ for $\xi \in \omega$, $\boldsymbol{\alpha} \in \R^s$. Moreover,
\begin{equation}
    \frac{{\partial {{\boldsymbol{\tau }}_k}}}{{\partial {\alpha _i}}}(\xi,\boldsymbol{\alpha}) = \sum_{l=1}^9 \boldsymbol{\eta}_l(\xi) \big(\mbf{A}(k) \mbf{n}^{(i)}\big)_l \ \in \ E_s^2 \qquad \forall\, i ,
\end{equation}
where $\mbf{n}^{(i)}$ is the $i$-th column of the matrix $\mbf{N}$. 

To solve the problem \eqref{eq:discrete} the gradient-oriented \textit{frprmn}(\ldots) procedure implementing the Fletcher–Reeves–Polak–Ribiere algorithm of the unconstrained minimization (for parameter ftol = 1.0e-7) will be applied (see \cite{press1996}) within the own program written in C++ language. It should be noted that, in the case of the sp-IMD method (or for both methods when $p=1$), the functional $\Phi_h$ is non-smooth; this can be seen from the formulae for $g_{k,i}$. Notwithstanding this, the multiple numerical experiments carried out by the authors have not indicated the need for employing  non-smooth optimization algorithms, e.g., the sub-gradient methods. Further testing is warranted, and, ultimately, implementation of suitable regularization techniques may prove essential.

\subsection{Case studies}

The numerical scheme put forth above will be now demonstrated on two problems involving two different design domains $\Omega \subset \R^2$: the L-shaped domain with the rounded reentrant corner and the rectangular domain, see Fig. \ref{fig:dom}. The (vp-IMD) and (sp-IMD) problems will be solved for four values of the exponent: $p \in \{1,2,3,100\}$.
The loads will be assumed as a piece-wise constant traction $\mbf{t}$ on the boundary. Namely, the virtual work functional will be of the form ${f}(\mbf{v}) = \int_{\bO \backslash \Gamma_0} \mbf{t} \cdot \mbf{v}\, ds$, being a continuous functional on $W^{1,r'}\!(\O;\R^2)$ for any $p \in[1,\infty)$. Since we are engaging 2D elasticity, the moduli $k, \mu$ are, in fact, stiffnesses being the integrals of the moduli across the thickness of a thin plate, and therefore are of units N/m.
The value of the referential stiffness will be fixed as $E_0 = 216554$ [N/m].
Note that its value acts merely as the scaling factor for the final values of the optimal elastic moduli $\hat{k},\hat{\mu}$, see the formulae in Theorems \ref{thm:stress_vp}, \ref{thm:stress_sp} with $\Lambda= \abs{\Omega}^{\frac{1}{p}} E_0$.

\begin{figure}[h]
\captionsetup[subfloat]{labelformat=simple,farskip=3pt,captionskip=1pt}
	\centering
    \subfloat[(a)]{\includegraphics*[trim={4cm 3cm 20cm 4cm},clip,width=0.4\textwidth]{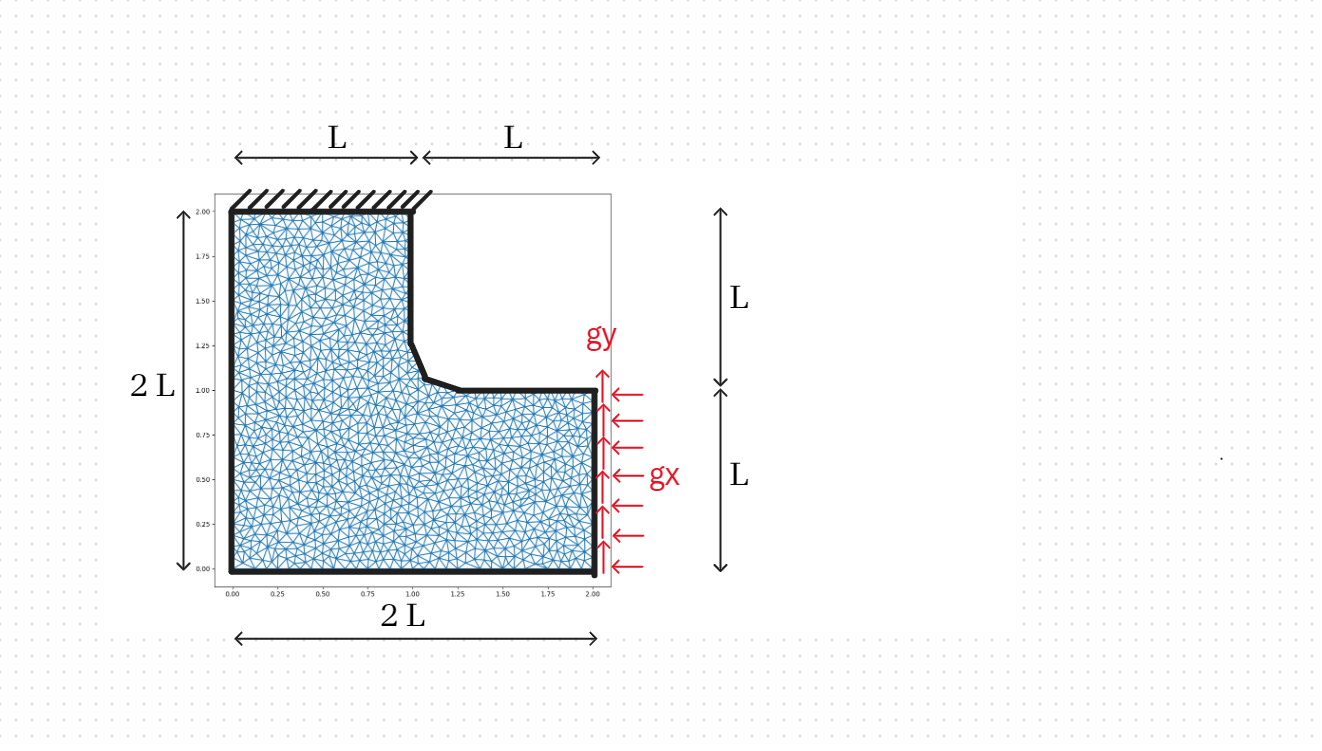}}\hspace{2cm}
	\subfloat[(b)]{\includegraphics*[trim={17cm 3cm 4cm 4cm},clip,width=0.31\textwidth]{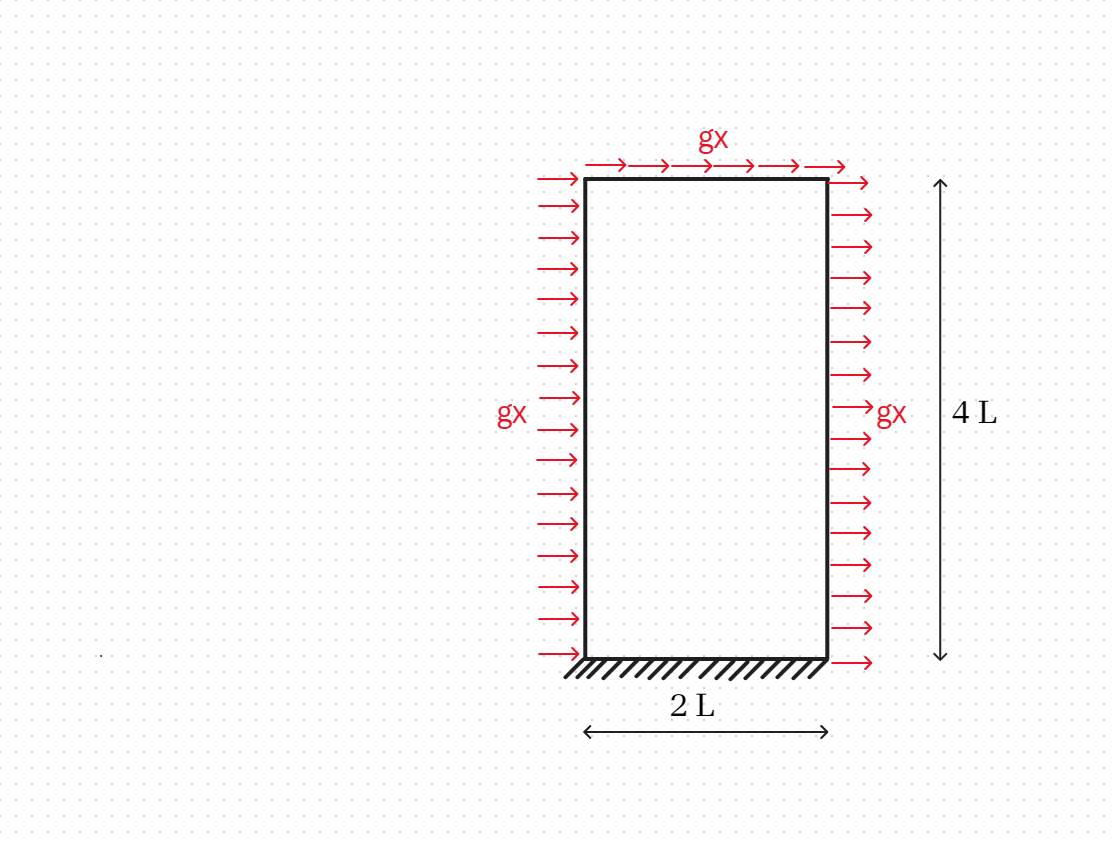}}
	\caption{The problem data: (a) the L-shaped domain problem; (b) the cantilever problem; $L=1.0\,\mathrm{m}$, $g_x = g_y= 5\, \mathrm{N/m} $.}
	\label{fig:dom}       
\end{figure}

We shall present the results in the form of the scatter plots of the numerical approximation of the optimal moduli $\hat{k},\hat{\mu}$: for both the methods (vp-IMD), (sp-IMD) and for the range of exponents $p$. The values of the minimal compliances will be given. In addition, for the L-shape domain problem, we will show the plots of the optimal Young modulus $\hat{E}$ and Poisson ratio $\hat{\nu}$.

Let us note that, although the case $p=1$ is not a part of the main stream of this paper, adding the results for this limit scenario facilitates a comparison to the solutions for the original (IMD) problem \cite{czarnecki2015a}. Although the theory of (IMD) is not a straightforward extension of the one expounded in this work, the numerical strategy put forth above works fine for $p=1$. Recall that the two methods (vp-IMD) and (sp-IMD) collapse to one for this choice of the exponent.

\begin{example}[\textbf{L-shaped domain problem}] The geometrical parameters of the L-shaped domain $\Omega$ are illustrated in Fig. \ref{fig:dom}(a) along with the boundary conditions and the triangulation into the finite elements. The resultant load applied to the lower right vertical contour line is directed left and up at an angle of $45\degree$. The minimal compliances are summarized in the Table \ref{tab:Ldomain}. The scatter plots of optimal bulk modulus $\hat{k}$ and shear modulus $\hat{\mu}$ are to  be found in Figs. \ref{fig:k} and \ref{fig:mu}, respectively. Using the formulae \eqref{eq:Enu2D}, the optimal Young modulus $\hat{E}$ and Poisson ratio $\hat{\nu}$ can be computed. These are presented in Figs. \ref{fig:E}, \ref{fig:nu}.

\begin{table}[h]
	\scriptsize
	\centering
	\caption{Optimal compliance in the L-shaped domain problem for various $p\in[1,\infty)$.}
	\begin{tabular}{l|ccccc}
		 & $ p=1$ (IMD)  & $ p=2$   & $ p=3$       & $ p=10^2$  & $p=10^6$\\
		\midrule[0.36mm]
		$\hat{\mathcal{C}}_{vp}$ [Nm]
		& 0.00293849 & 0.00250848 & 0.00238303 & 0.00217998 & 0.00217534 \\
		 $\hat{\mathcal{C}}_{sp}$ [Nm] & 0.00293849 & 0.00415270 & 0.00463301 & 0.00566942 & 0.00570388 \\
         $\hat{\mathcal{C}}_{vp}/(1+\beta^2)^{1/p}$ [Nm] &
         0.00097950 & 0.00144827 & 0.00165230 & 0.00215639 & 0.00217534
	\end{tabular}
	\label{tab:Ldomain}
\end{table}
\begin{figure}[h]
\captionsetup[subfloat]{labelformat=simple,farskip=3pt,captionskip=1pt}
	\centering
    \subfloat[IMD, $p=1$]{\includegraphics*[trim={16.3cm 6.2cm 5.4cm 4.7cm},clip,width=0.30\textwidth]{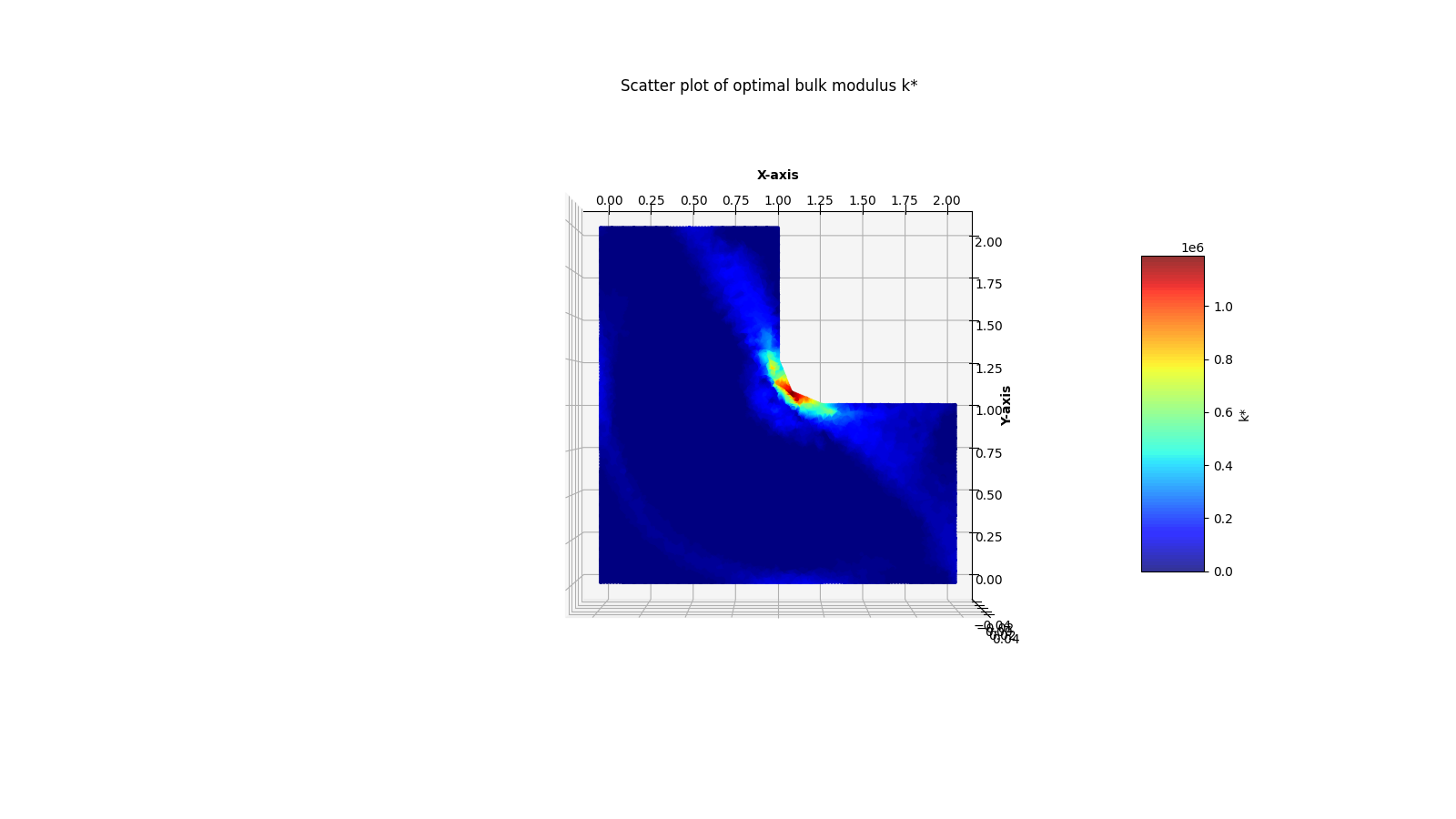}}\\
	\subfloat[vp-IMD, $p=2$]{\includegraphics*[trim={16.3cm 6.2cm 5.4cm 4.7cm},clip,width=0.30\textwidth]{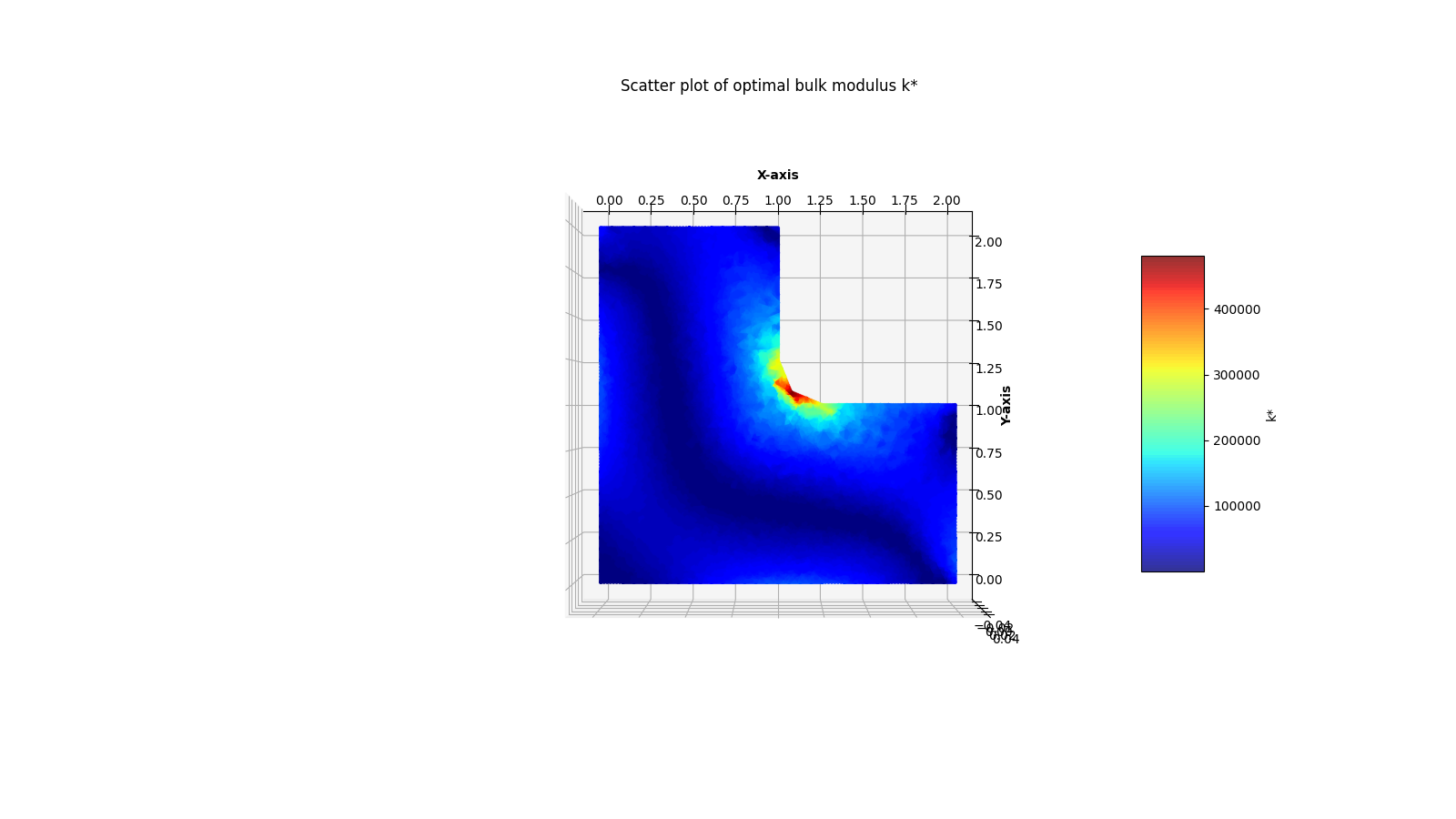}}\hspace{2cm}
	\subfloat[sp-IMD, $p=2$]{\includegraphics*[trim={16.3cm 6.2cm 5.4cm 4.7cm},clip,width=0.30\textwidth]{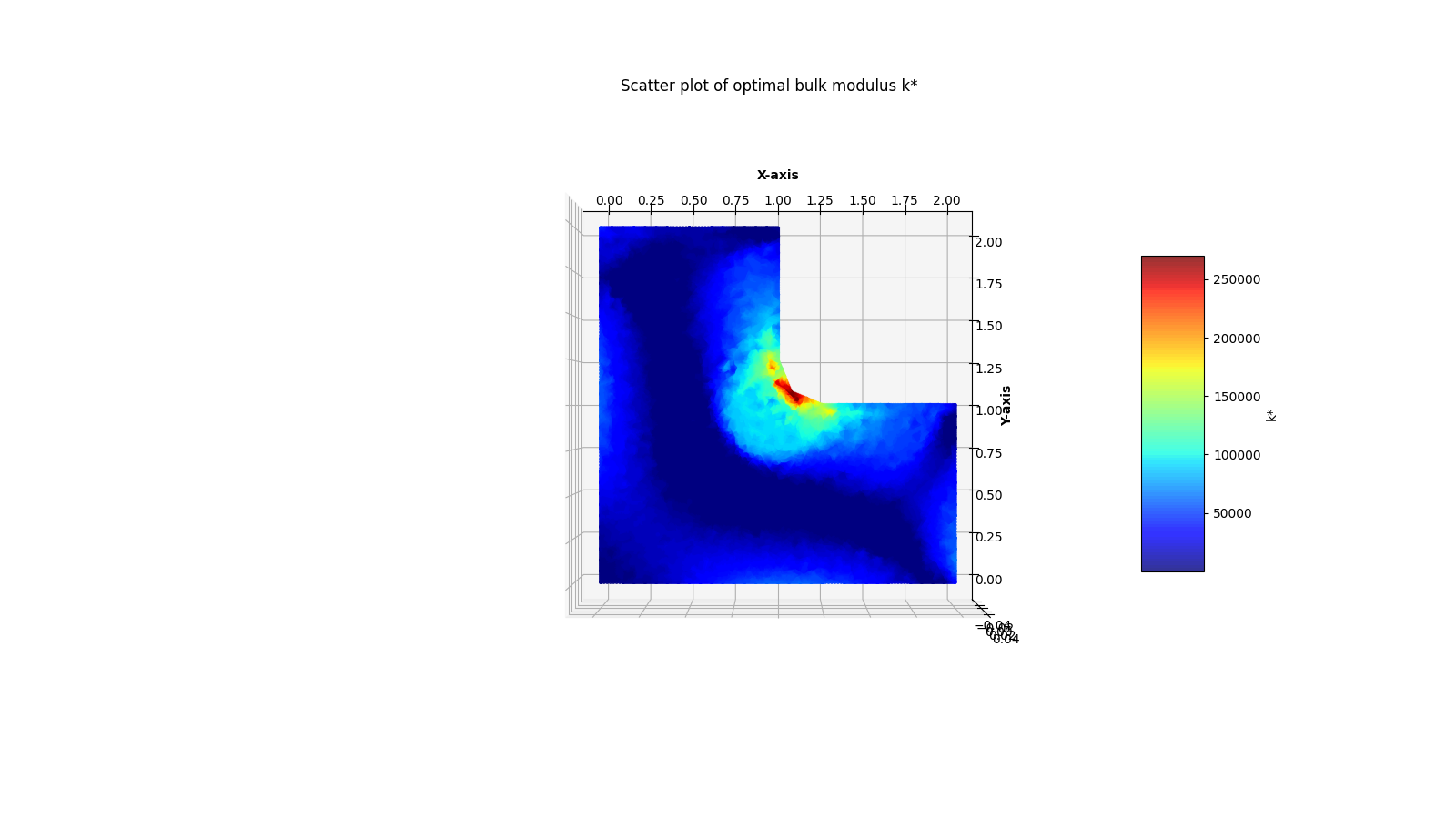}}\\
	\subfloat[vp-IMD, $p=3$]{\includegraphics*[trim={16.3cm 6.2cm 5.4cm 4.7cm},clip,width=0.30\textwidth]{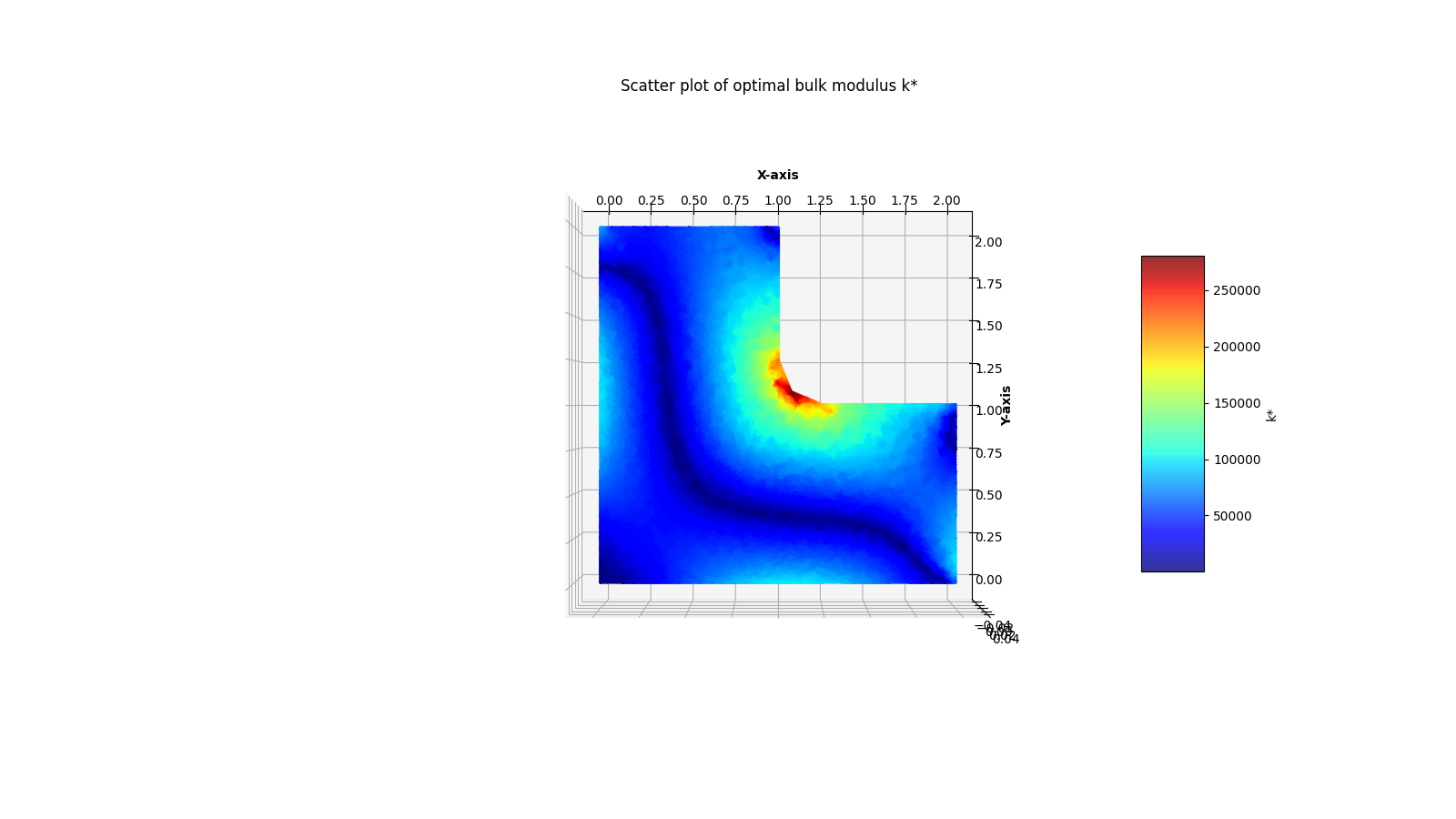}}\hspace{2cm}
	\subfloat[sp-IMD, $p=3$]{\includegraphics*[trim={16.3cm 6.2cm 5.4cm 4.7cm},clip,width=0.30\textwidth]{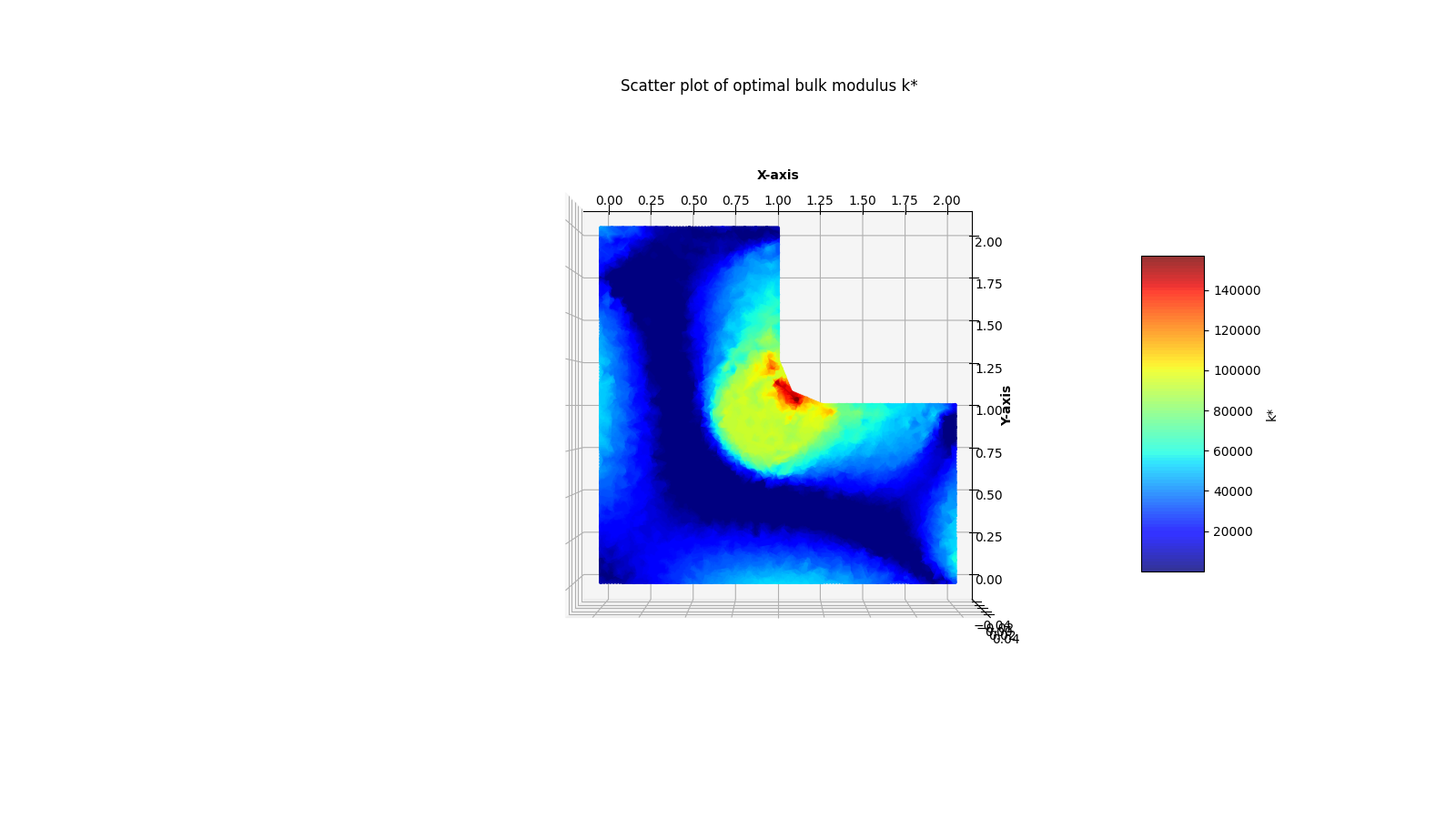}}\\
	\subfloat[vp-IMD, $p=100$]{\includegraphics*[trim={16.3cm 6.2cm 5.4cm 4.7cm},clip,width=0.30\textwidth]{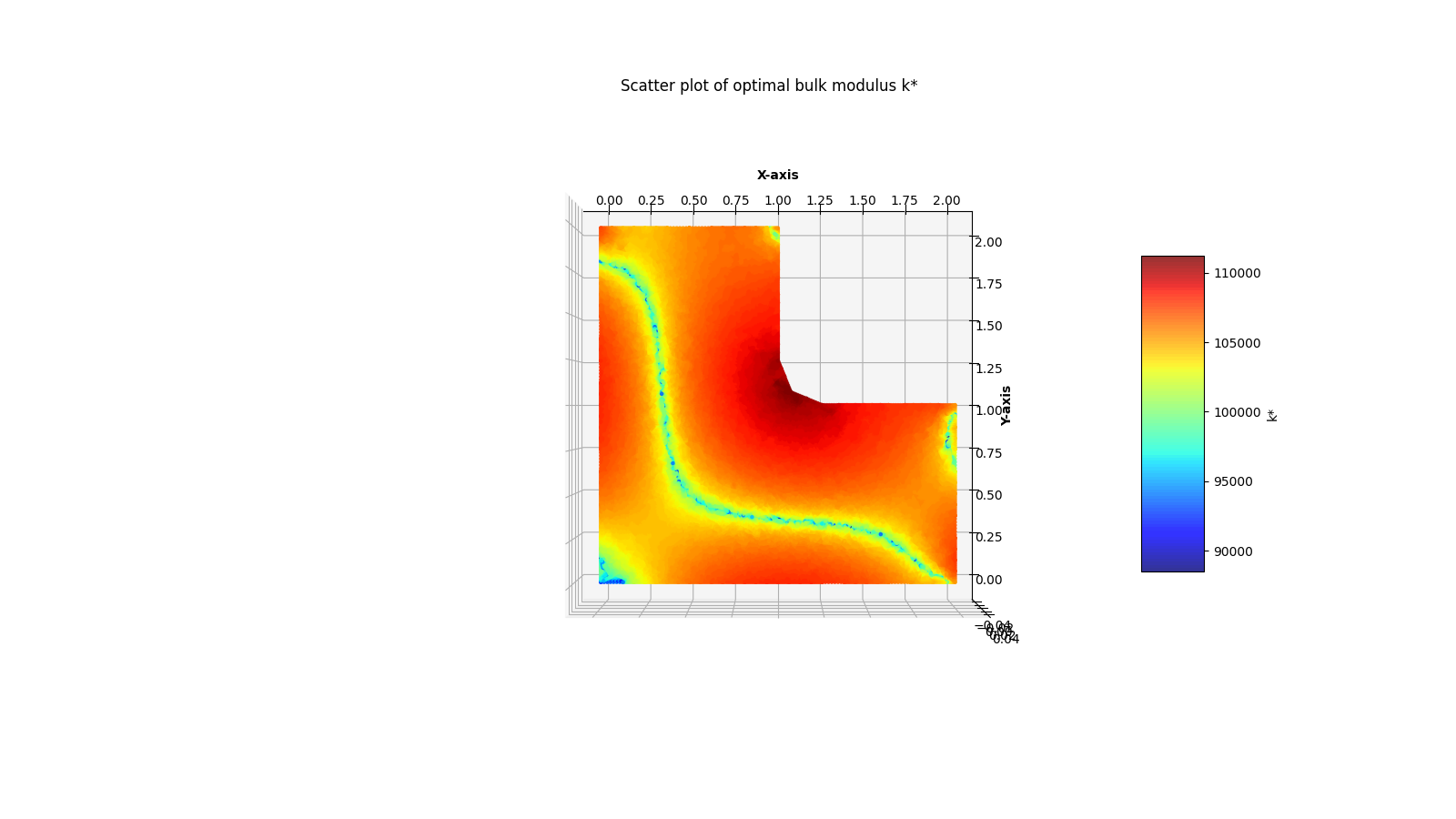}}\hspace{2cm}
	\subfloat[sp-IMD, $p=100$]{\includegraphics*[trim={16.3cm 6.2cm 5.4cm 4.7cm},clip,width=0.30\textwidth]{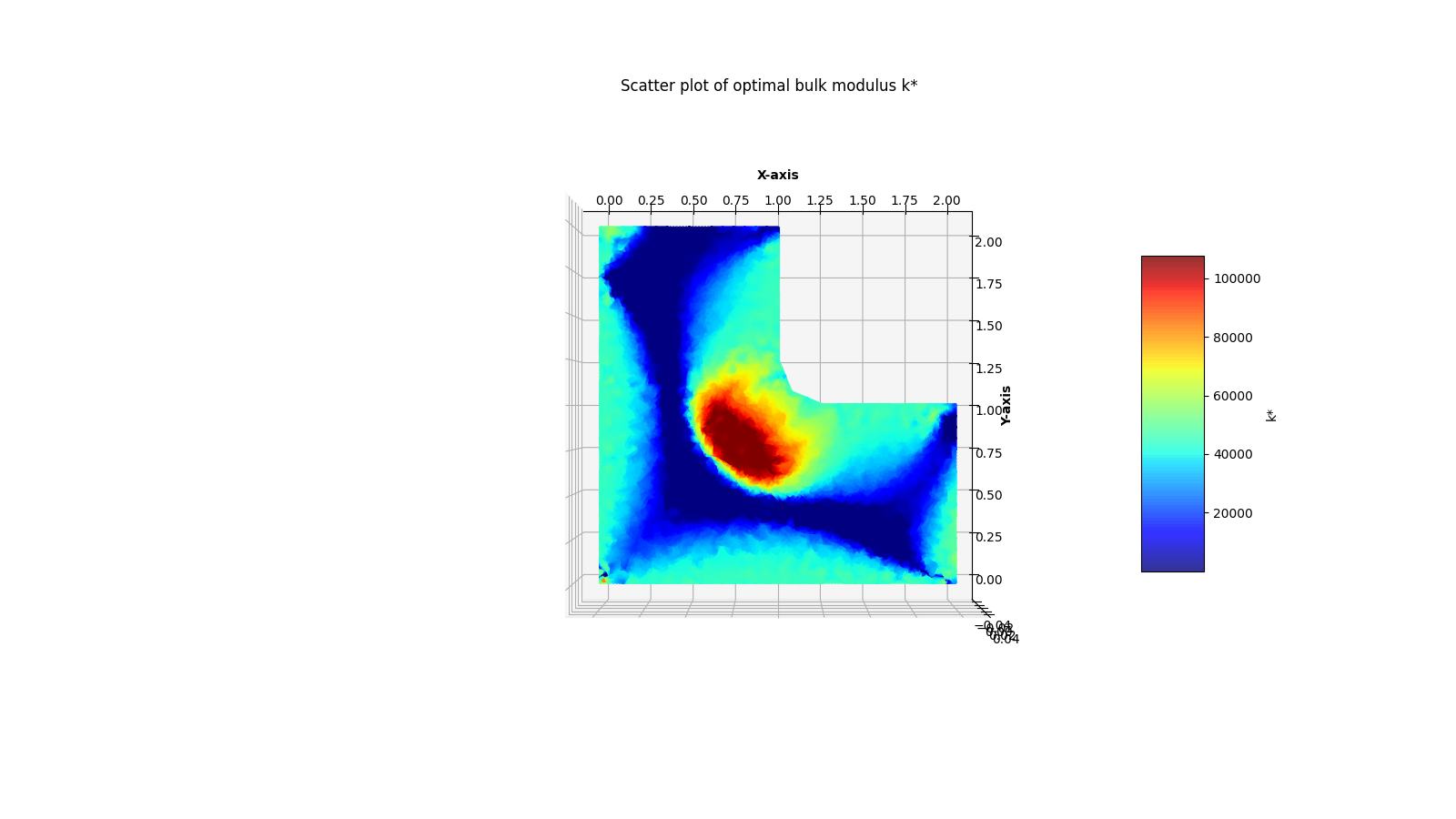}}
	\caption{L-shaped domain problem, optimal bulk modulus $\hat{k}$ for various $p\in[1,\infty)$.}
	\label{fig:k}       
\end{figure}

\begin{figure}[h]
\captionsetup[subfloat]{labelformat=simple,farskip=3pt,captionskip=1pt}
	\centering
    \subfloat[IMD, $p=1$]{\includegraphics*[trim={16.3cm 6.2cm 5.4cm 4.7cm},clip,width=0.30\textwidth]{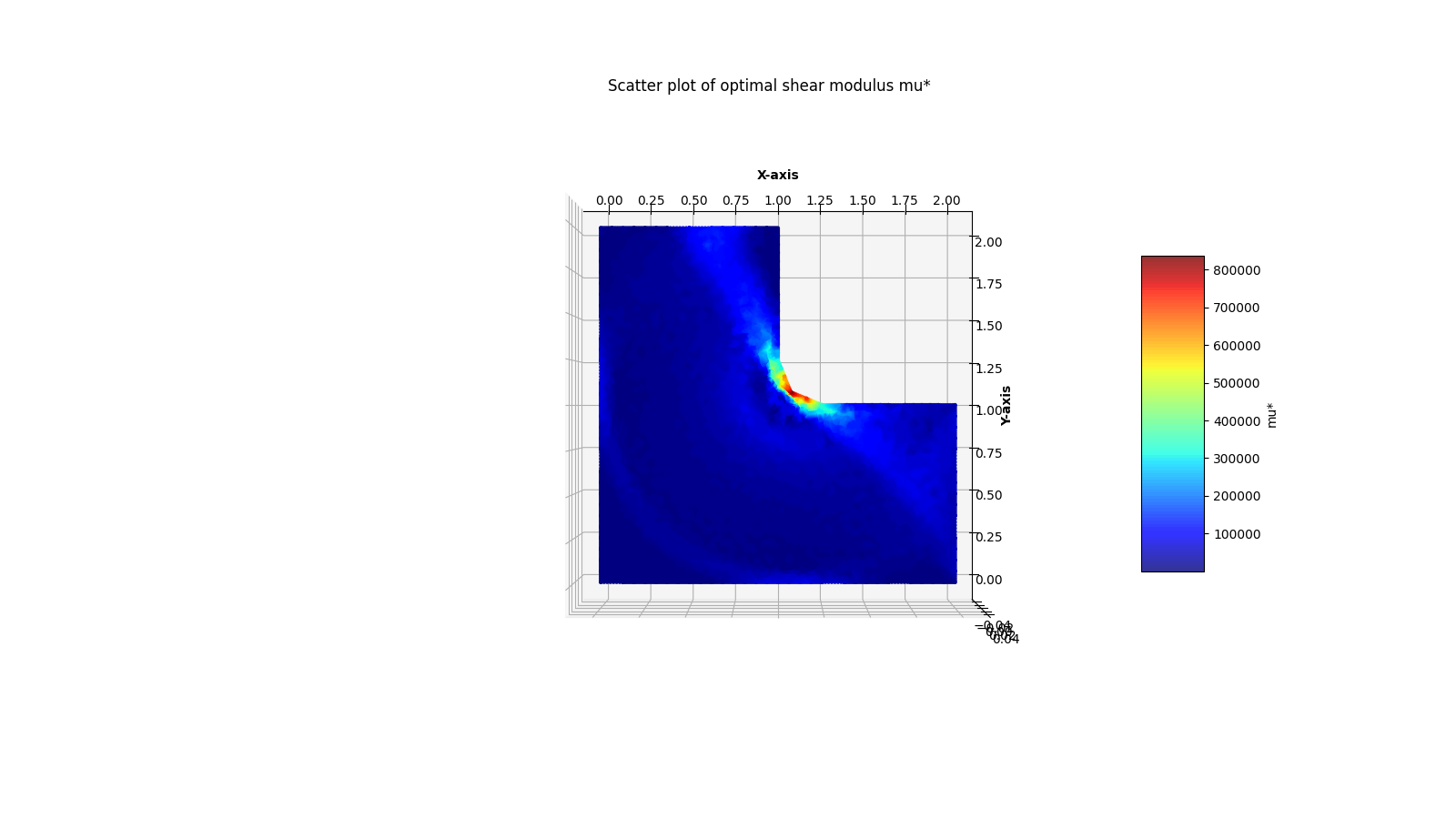}}\\
	\subfloat[vp-IMD, $p=2$]{\includegraphics*[trim={16.3cm 6.2cm 5.4cm 4.7cm},clip,width=0.30\textwidth]{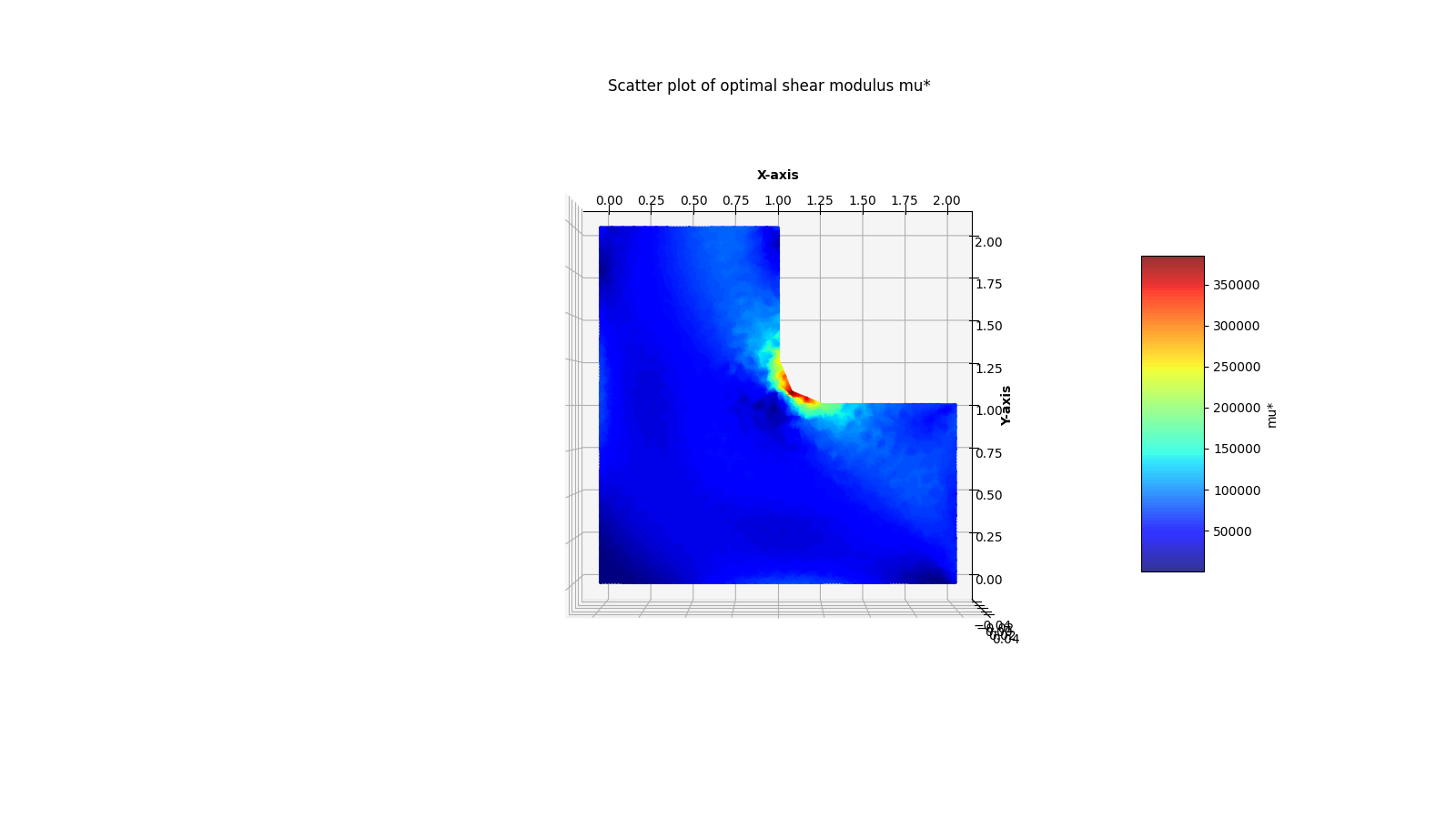}}\hspace{2cm}
	\subfloat[sp-IMD, $p=2$]{\includegraphics*[trim={16.3cm 6.2cm 5.4cm 4.7cm},clip,width=0.30\textwidth]{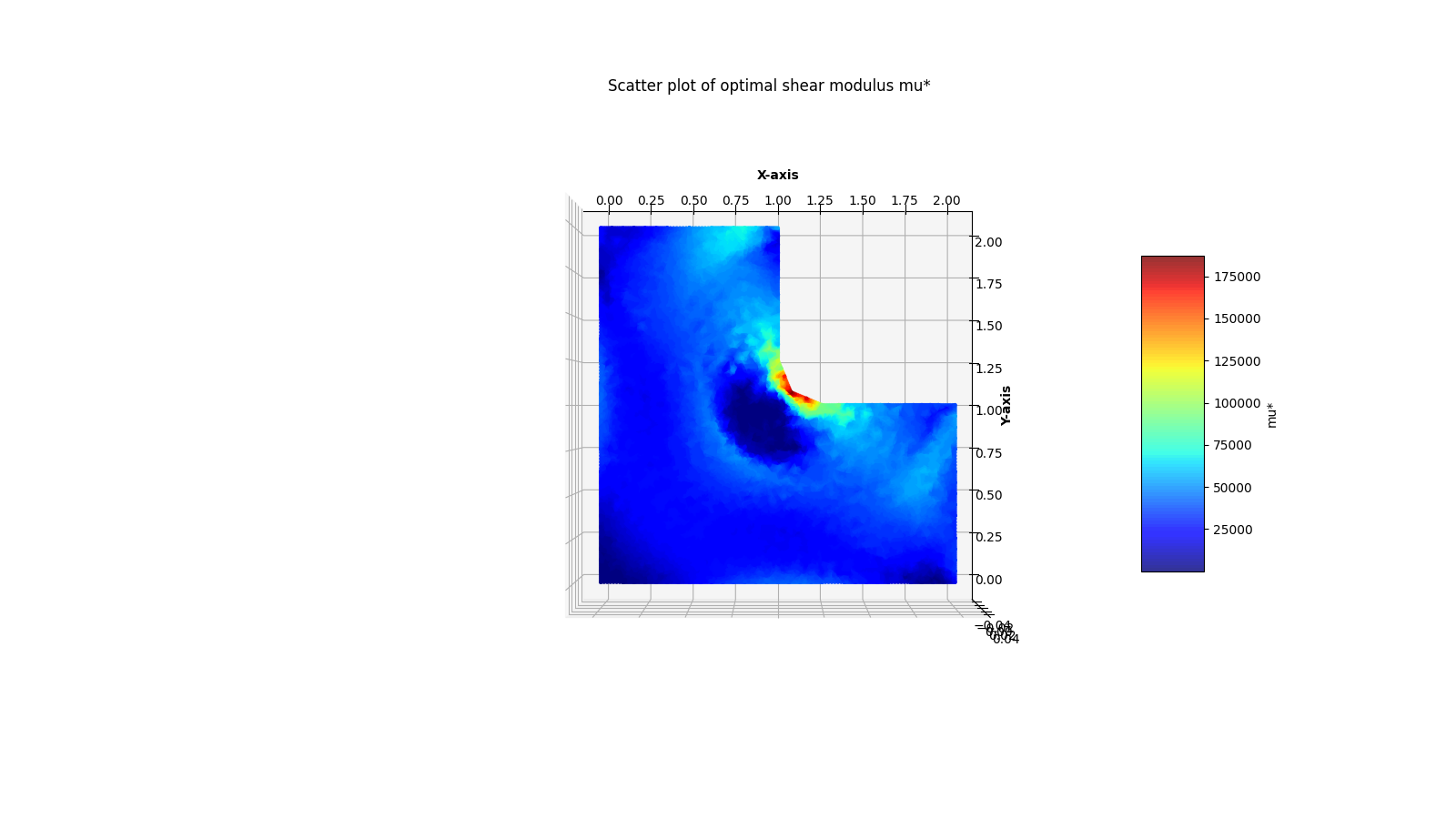}}\\
	\subfloat[vp-IMD, $p=3$]{\includegraphics*[trim={16.3cm 6.2cm 5.4cm 4.7cm},clip,width=0.30\textwidth]{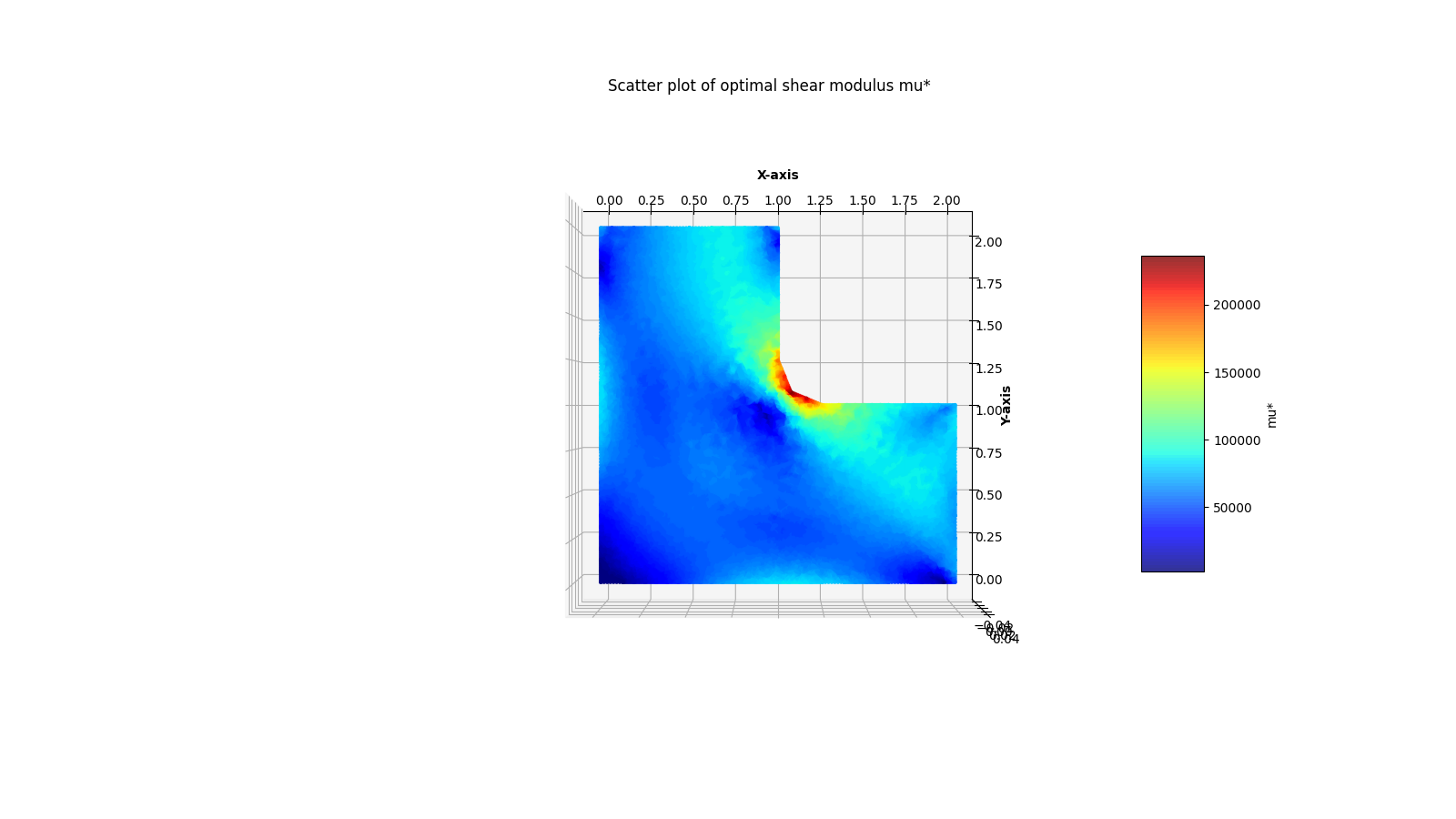}}\hspace{2cm}
	\subfloat[sp-IMD, $p=3$]{\includegraphics*[trim={16.3cm 6.2cm 5.4cm 4.7cm},clip,width=0.30\textwidth]{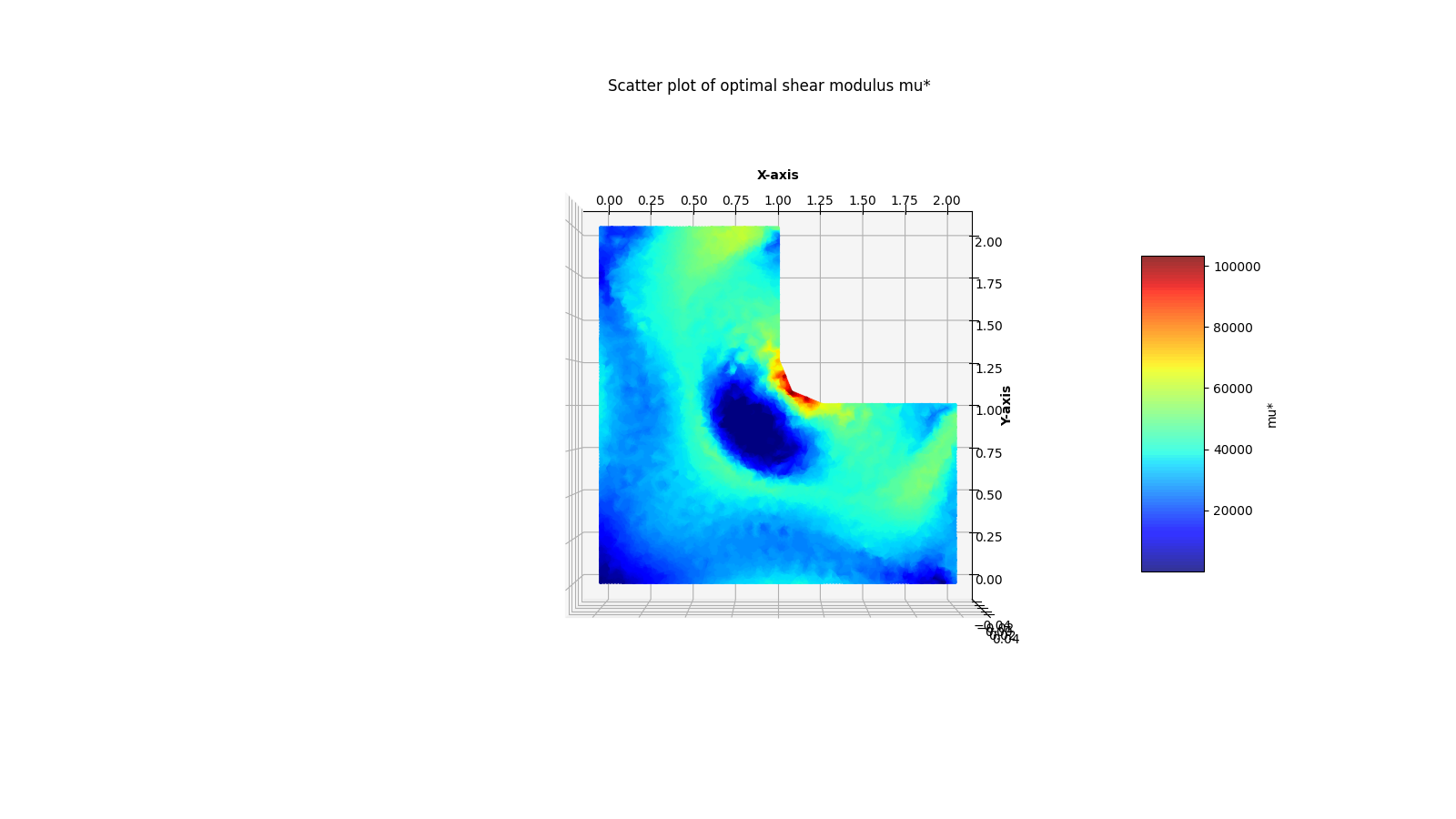}}\\
	\subfloat[vp-IMD, $p=100$]{\includegraphics*[trim={16.3cm 6.2cm 5.4cm 4.7cm},clip,width=0.30\textwidth]{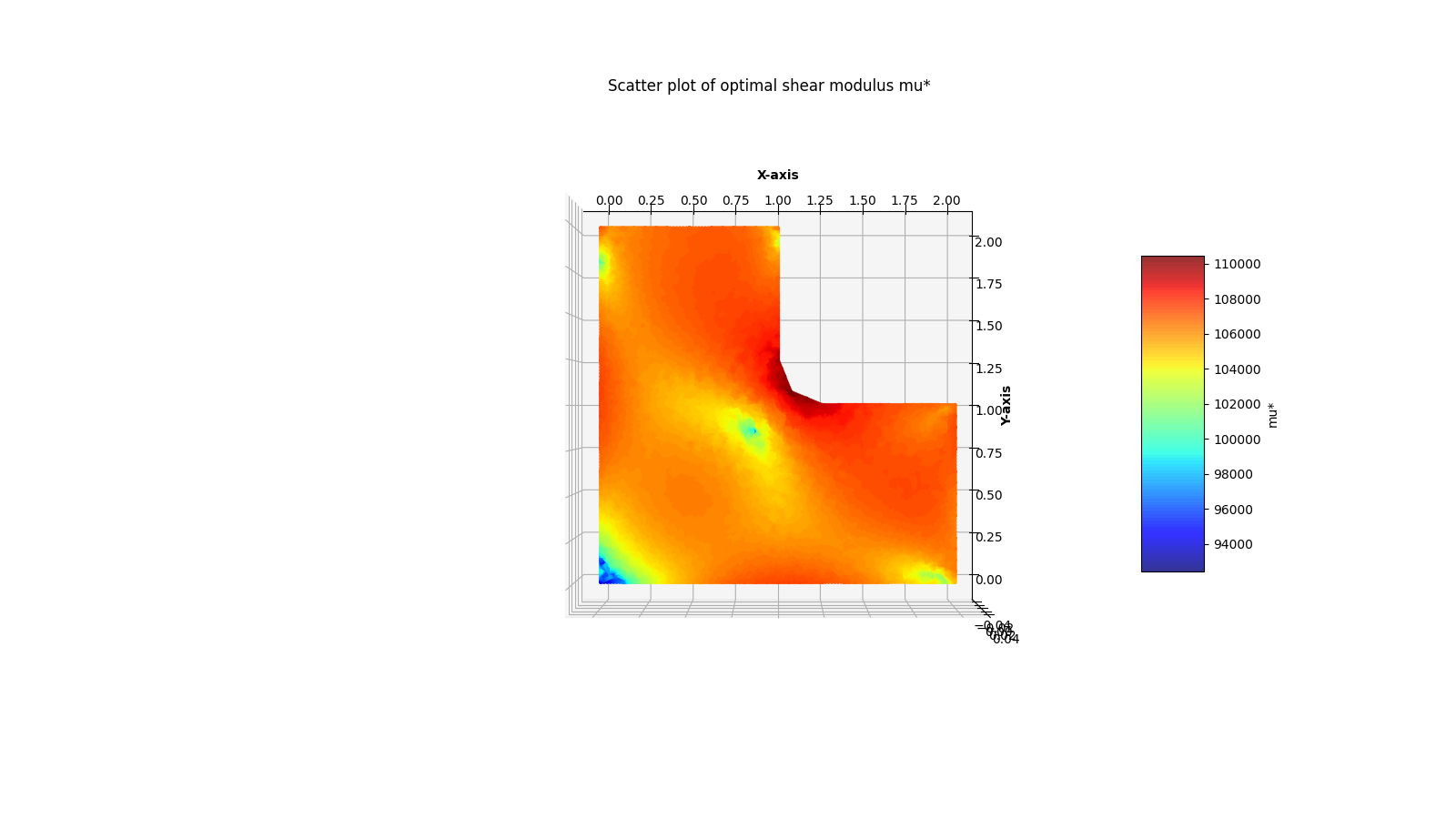}}\hspace{2cm}
	\subfloat[sp-IMD, $p=100$]{\includegraphics*[trim={16.3cm 6.2cm 5.4cm 4.7cm},clip,width=0.30\textwidth]{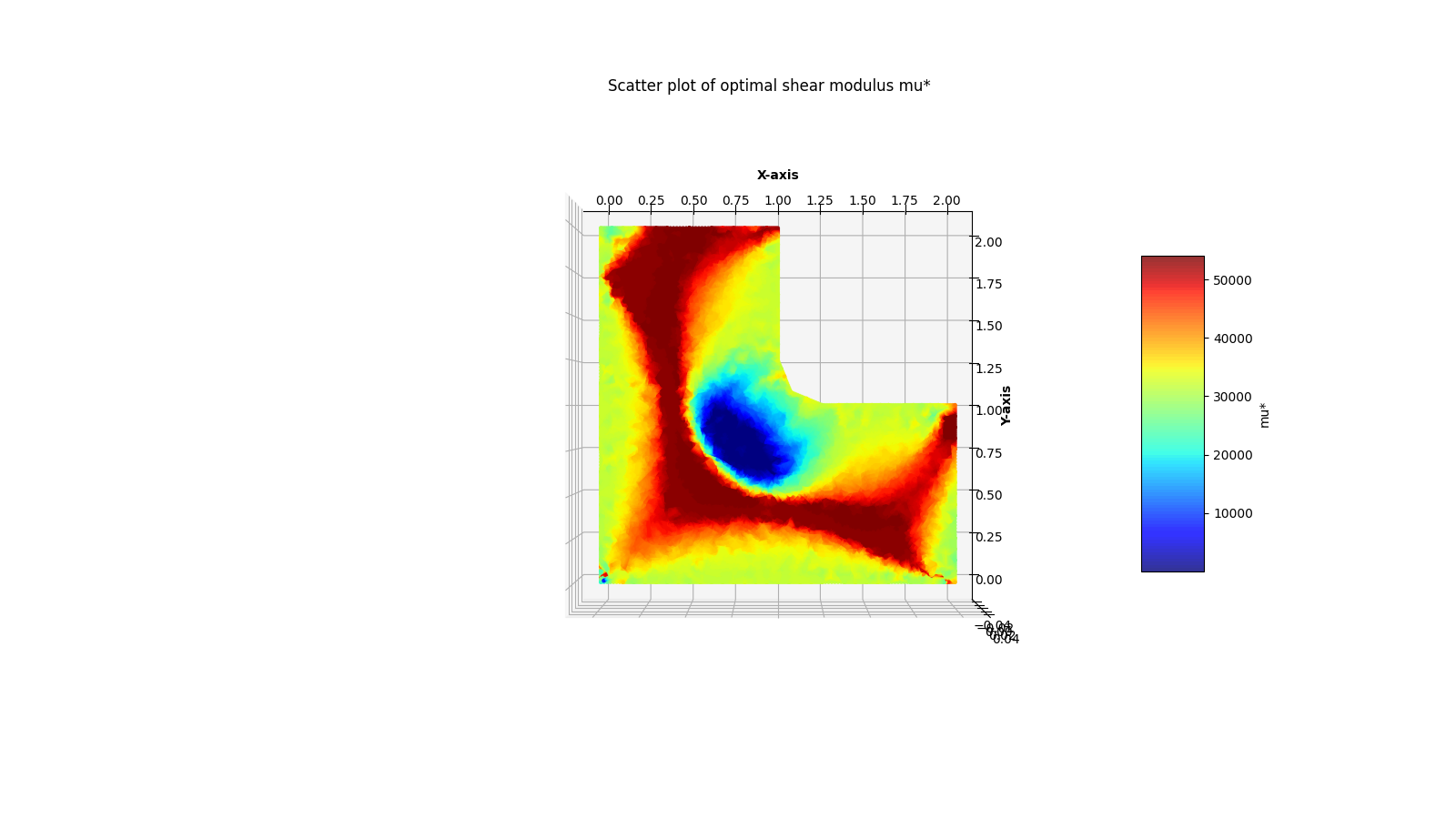}}
	\caption{L-shaped domain problem, optimal shear modulus $\hat{\mu}$ for various $p\in[1,\infty)$.}
	\label{fig:mu}       
\end{figure}

\begin{figure}[h]
\captionsetup[subfloat]{labelformat=simple,farskip=3pt,captionskip=1pt}
	\centering
    \subfloat[IMD, $p=1$]{\includegraphics*[trim={16.3cm 6.2cm 5.4cm 4.7cm},clip,width=0.30\textwidth]{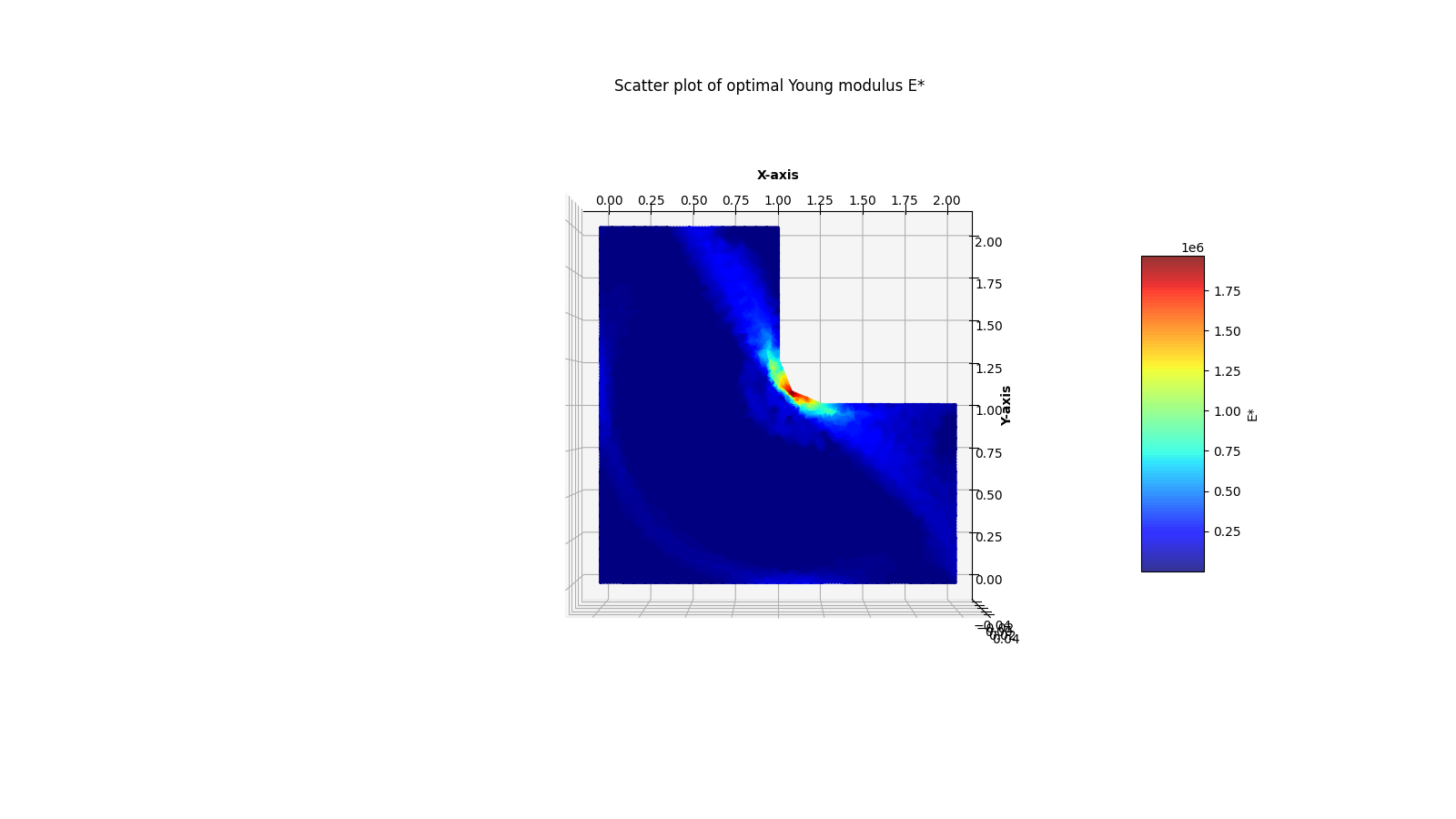}}\\
	\subfloat[vp-IMD, $p=2$]{\includegraphics*[trim={16.3cm 6.2cm 5.4cm 4.7cm},clip,width=0.30\textwidth]{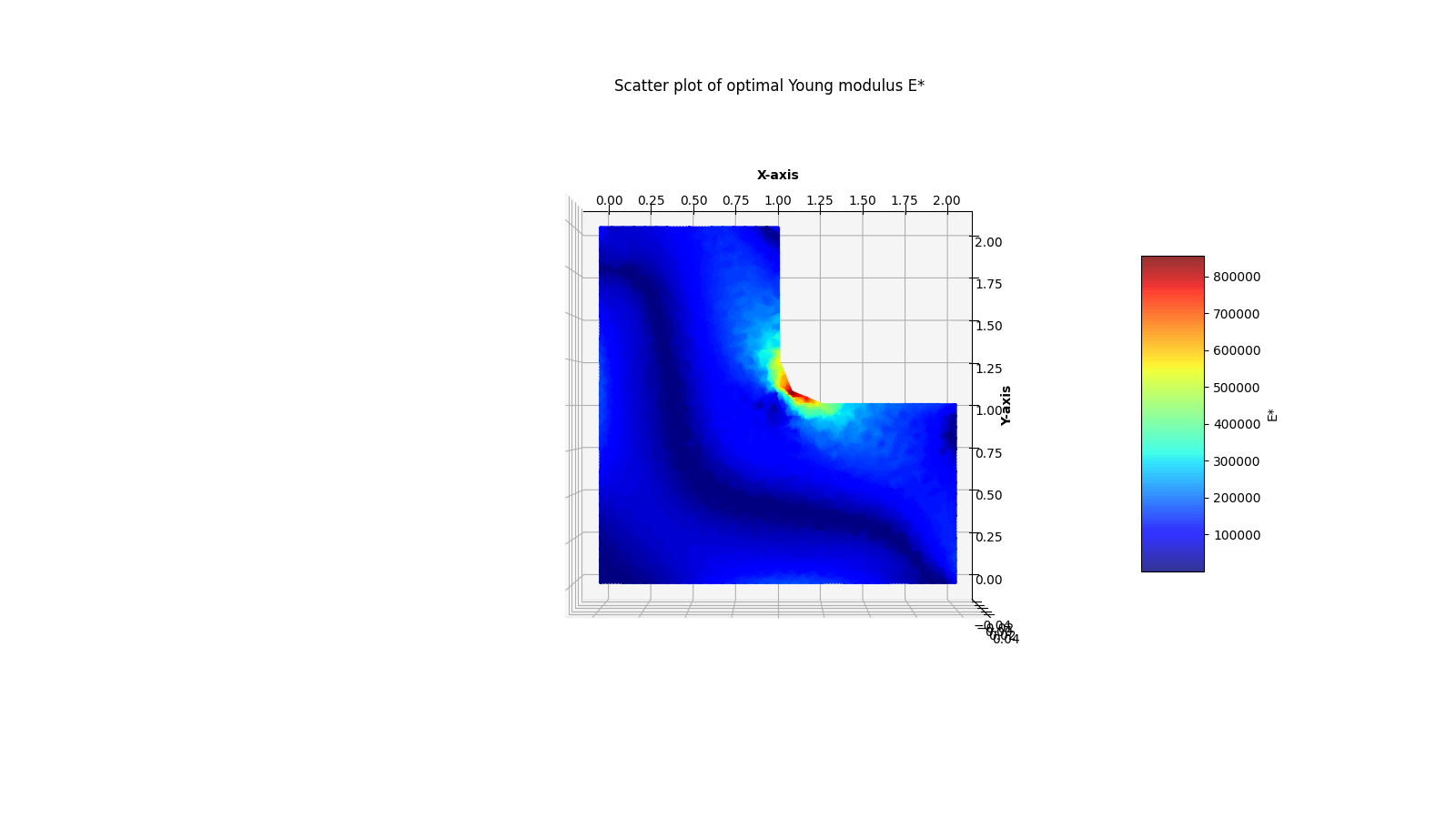}}\hspace{2cm}
	\subfloat[sp-IMD, $p=2$]{\includegraphics*[trim={16.3cm 6.2cm 5.4cm 4.7cm},clip,width=0.30\textwidth]{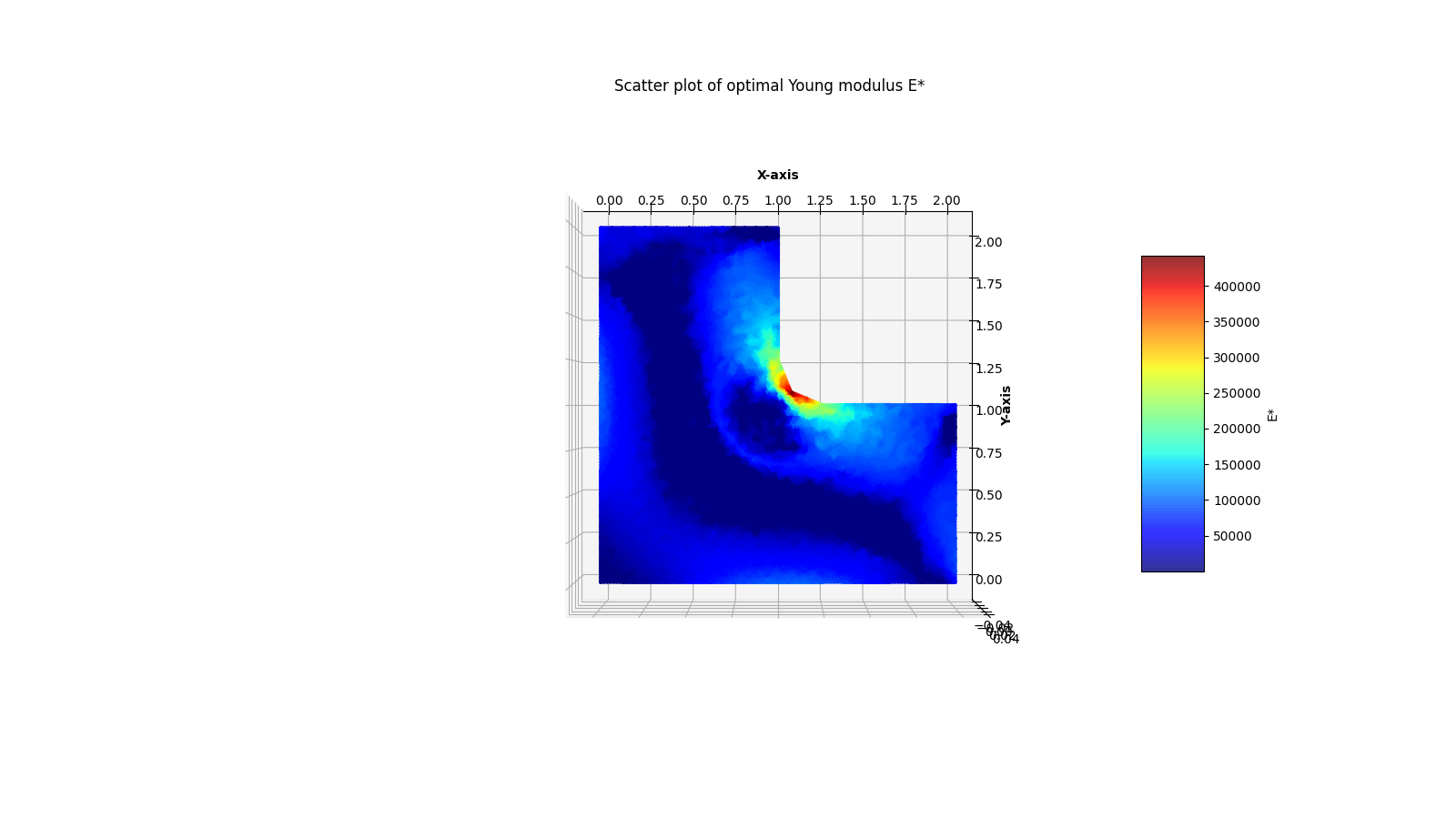}}\\
	\subfloat[vp-IMD, $p=3$]{\includegraphics*[trim={16.3cm 6.2cm 5.4cm 4.7cm},clip,width=0.30\textwidth]{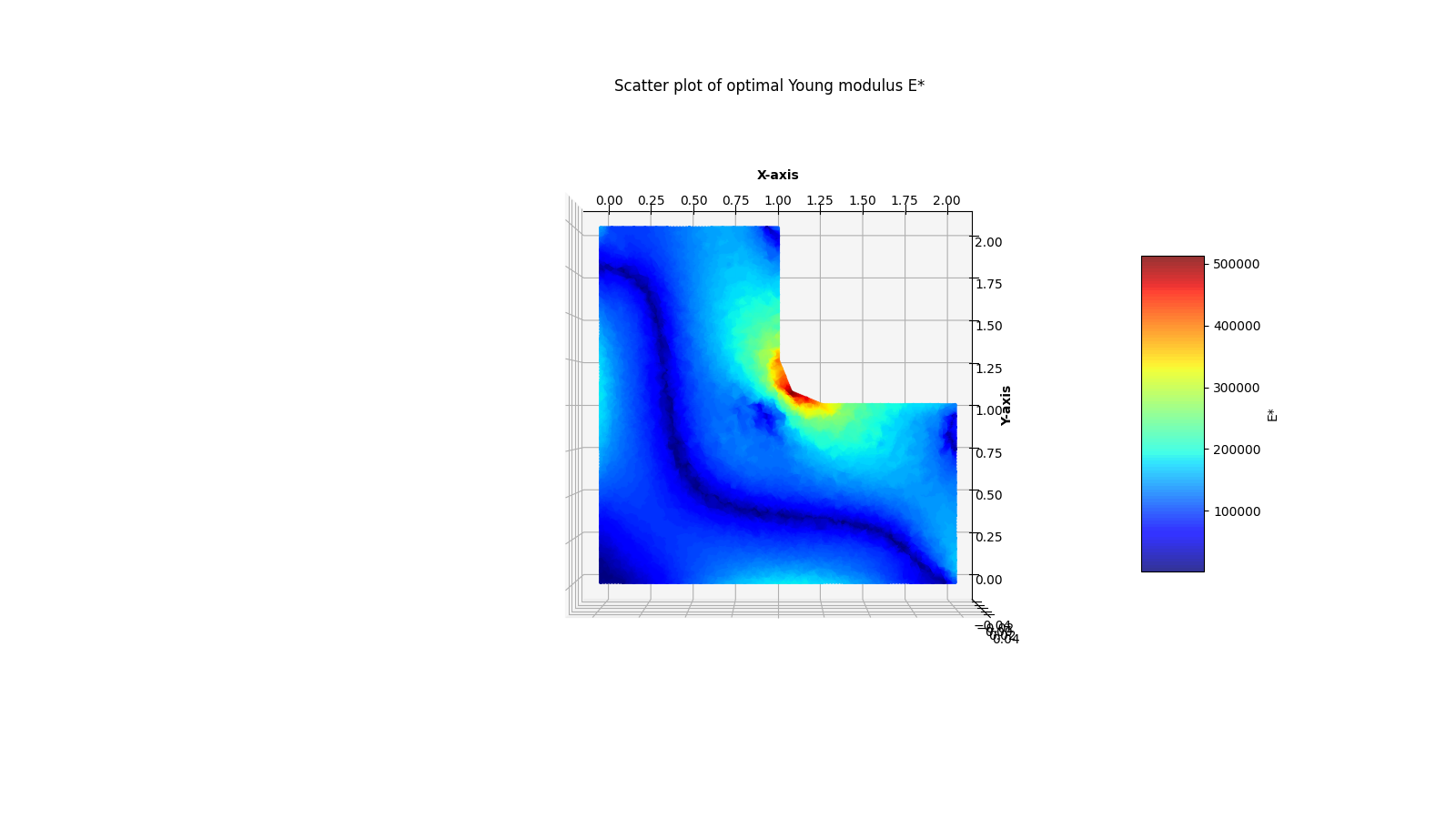}}\hspace{2cm}
	\subfloat[sp-IMD, $p=3$]{\includegraphics*[trim={16.3cm 6.2cm 5.4cm 4.7cm},clip,width=0.30\textwidth]{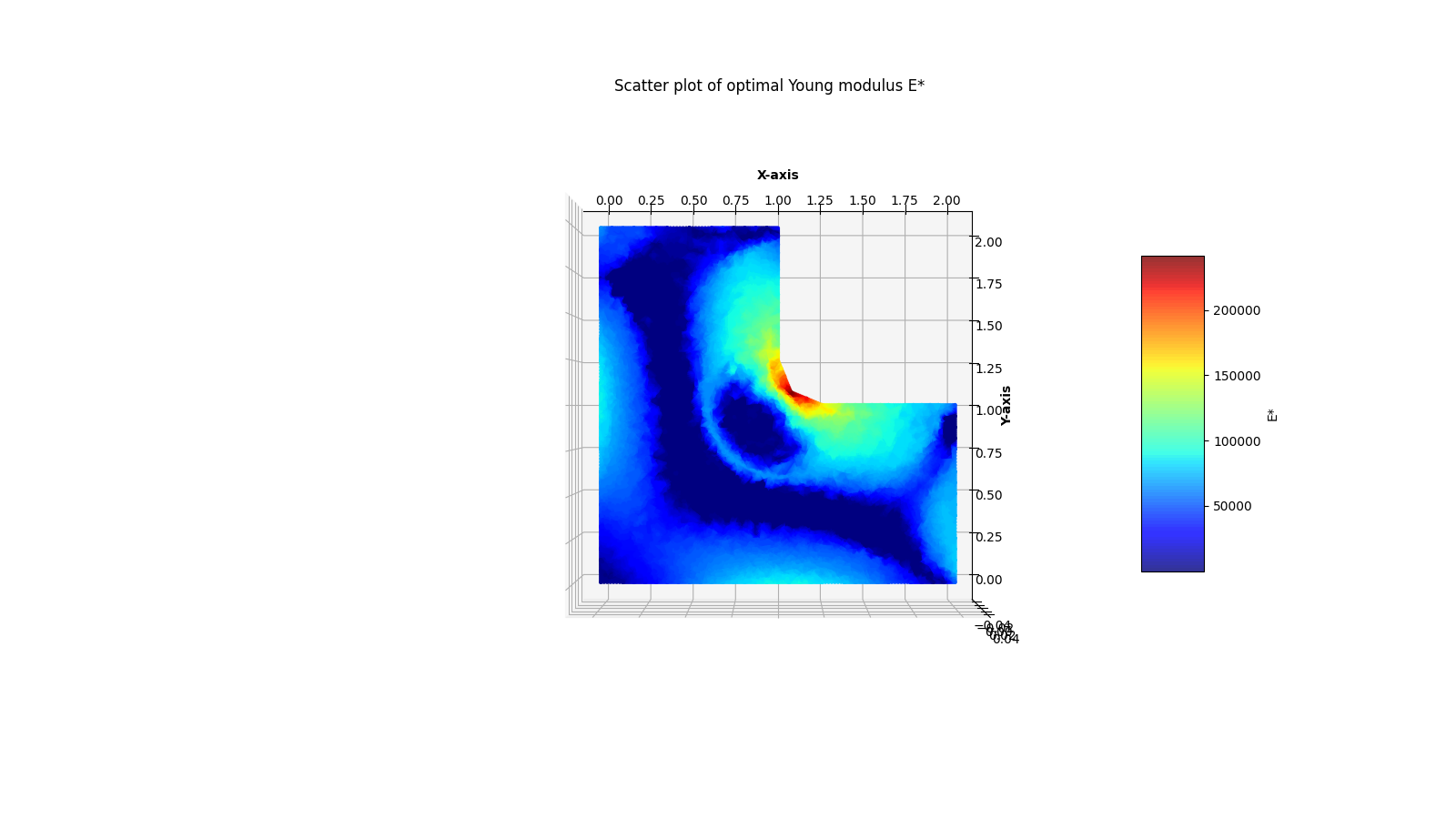}}\\
	\subfloat[vp-IMD, $p=100$]{\includegraphics*[trim={16.3cm 6.2cm 5.4cm 4.7cm},clip,width=0.30\textwidth]{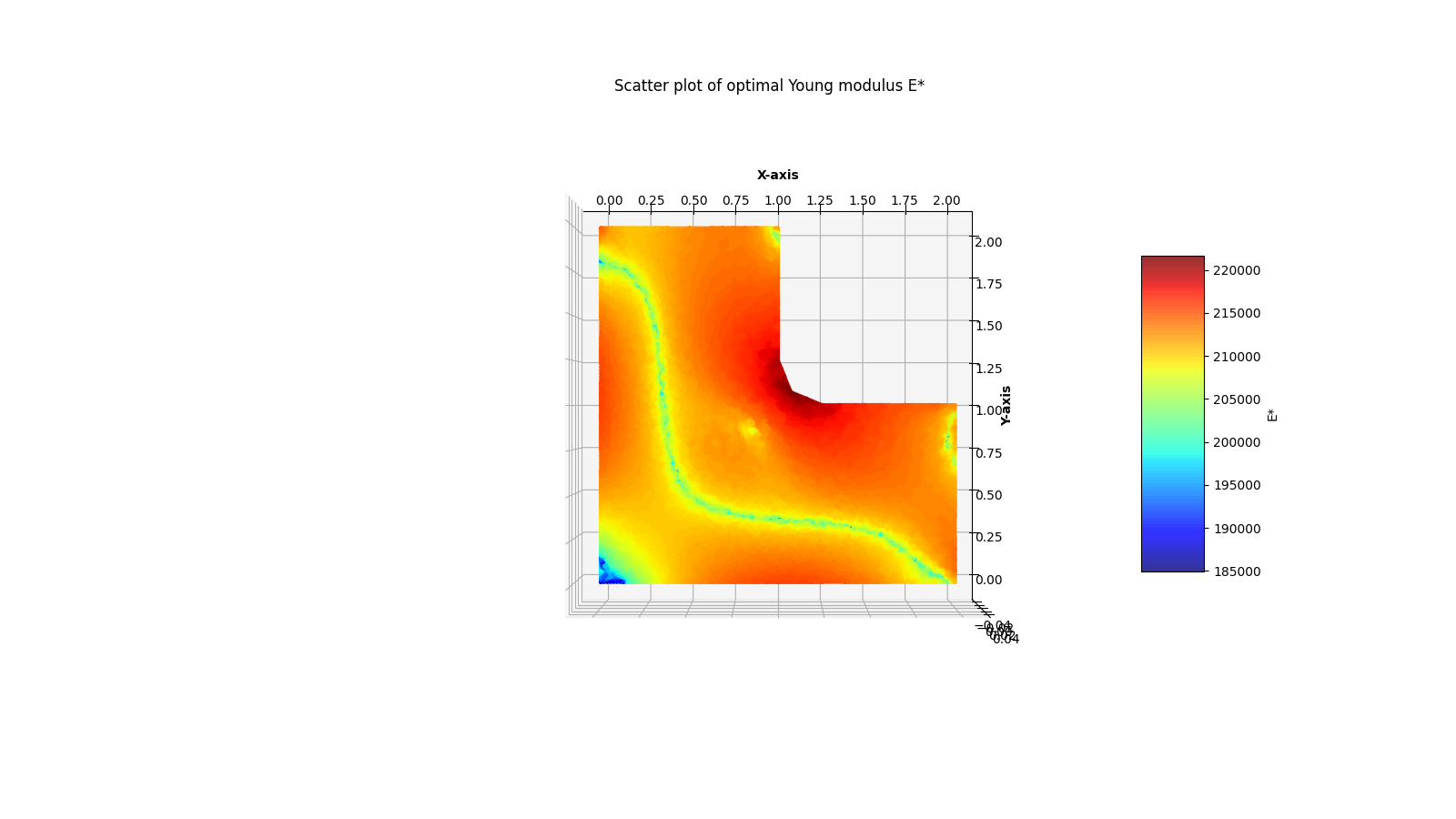}}\hspace{2cm}
	\subfloat[sp-IMD, $p=100$]{\includegraphics*[trim={16.3cm 6.2cm 5.4cm 4.7cm},clip,width=0.30\textwidth]{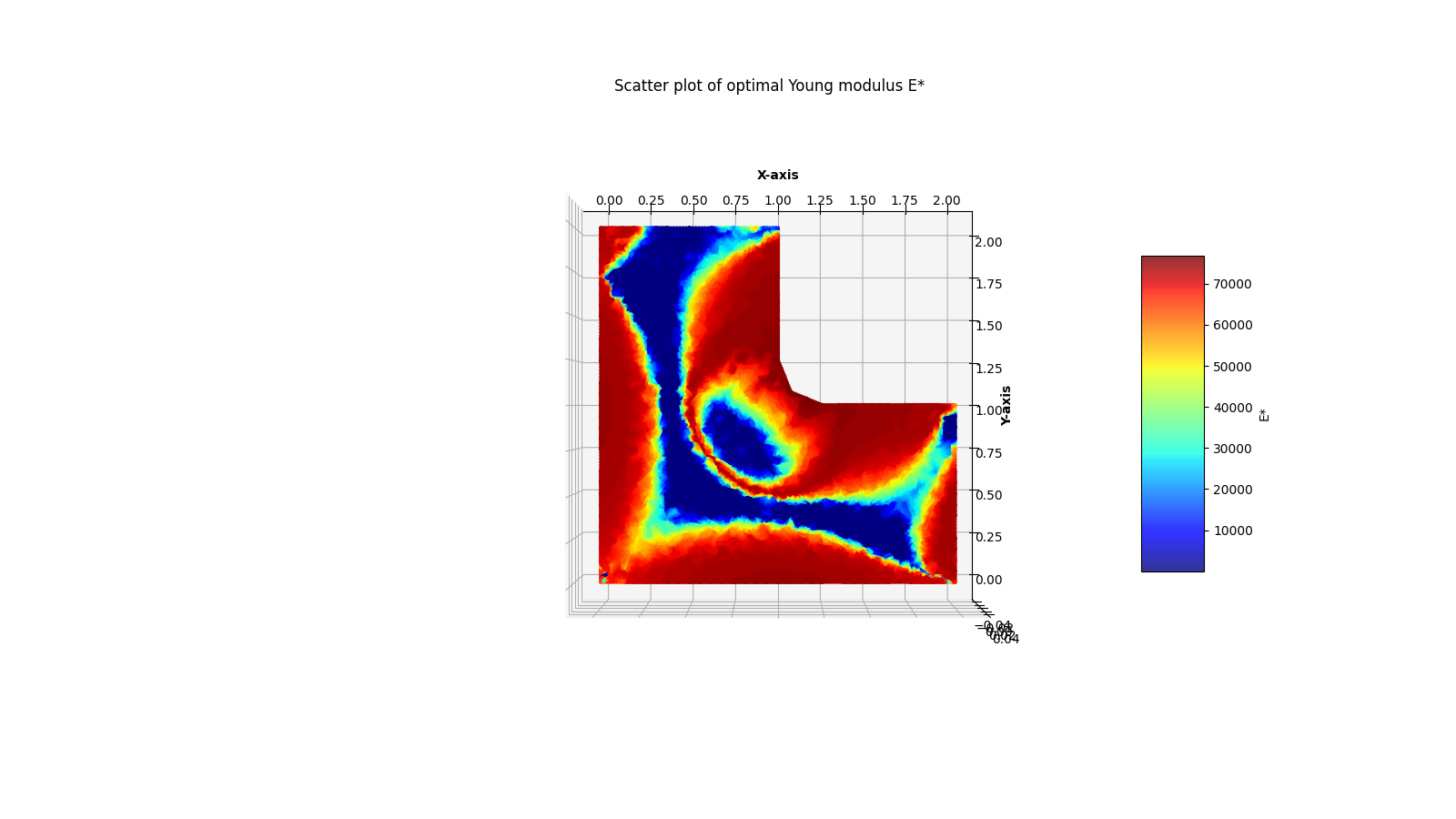}}
	\caption{L-shaped domain problem, optimal Young modulus $\hat{E}$ for various $p\in[1,\infty)$.}
	\label{fig:E}       
\end{figure}

\begin{figure}[h]
\captionsetup[subfloat]{labelformat=simple,farskip=3pt,captionskip=1pt}
	\centering
    \subfloat[IMD, $p=1$]{\includegraphics*[trim={16.3cm 6.2cm 5.4cm 4.7cm},clip,width=0.30\textwidth]{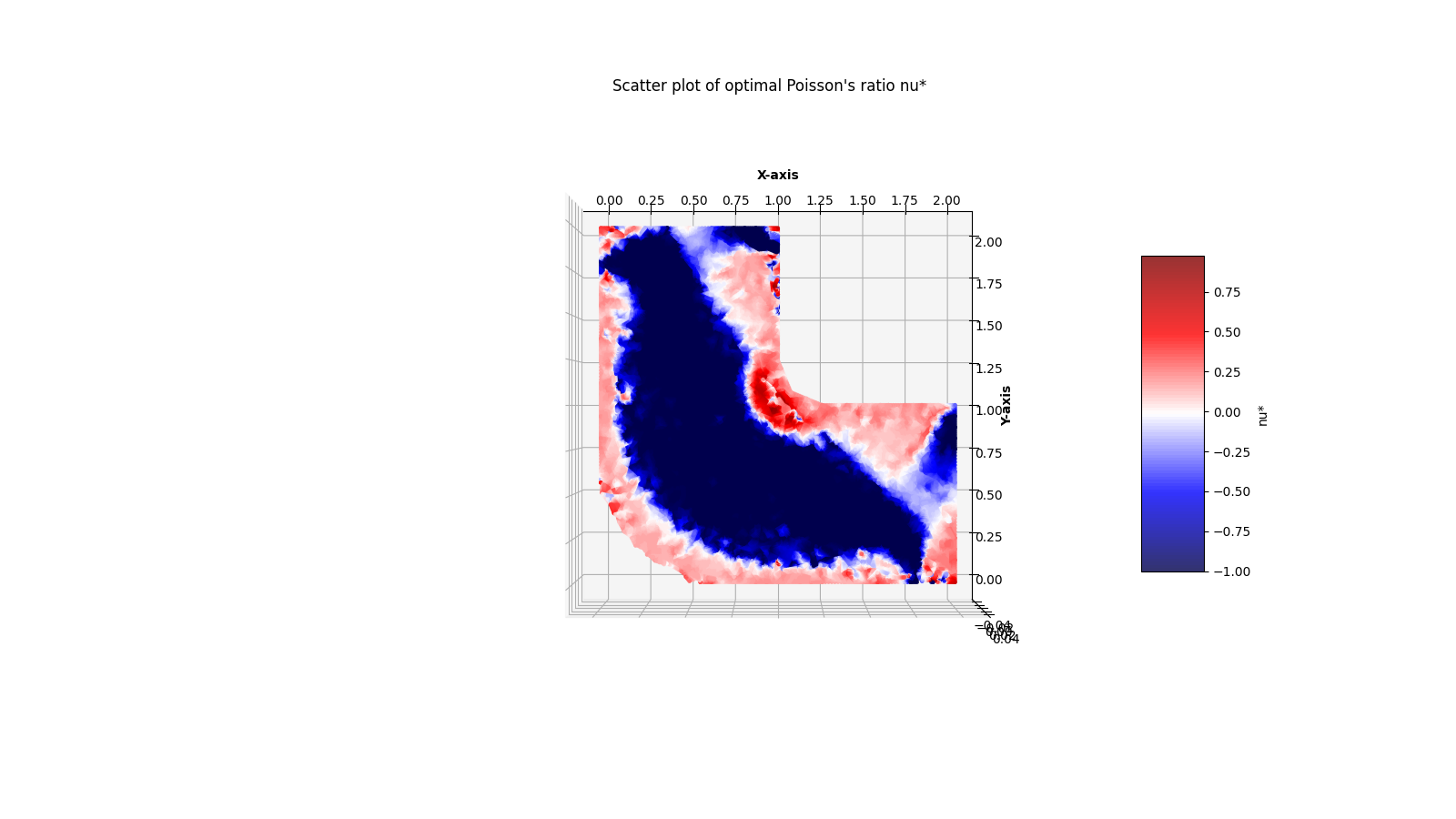}}\\
	\subfloat[vp-IMD, $p=2$]{\includegraphics*[trim={16.3cm 6.2cm 5.4cm 4.7cm},clip,width=0.30\textwidth]{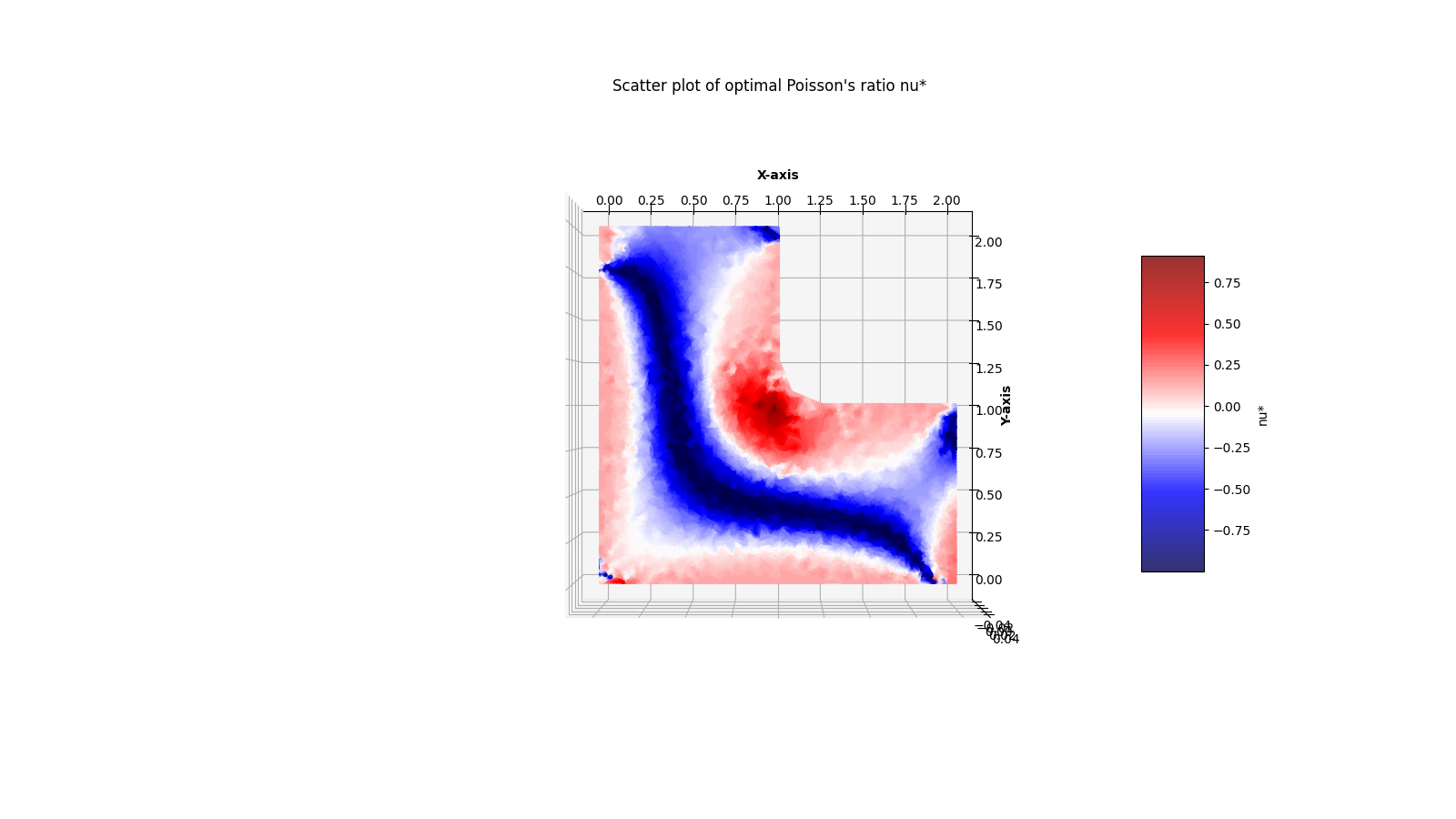}}\hspace{2cm}
	\subfloat[sp-IMD, $p=2$]{\includegraphics*[trim={16.3cm 6.2cm 5.4cm 4.7cm},clip,width=0.30\textwidth]{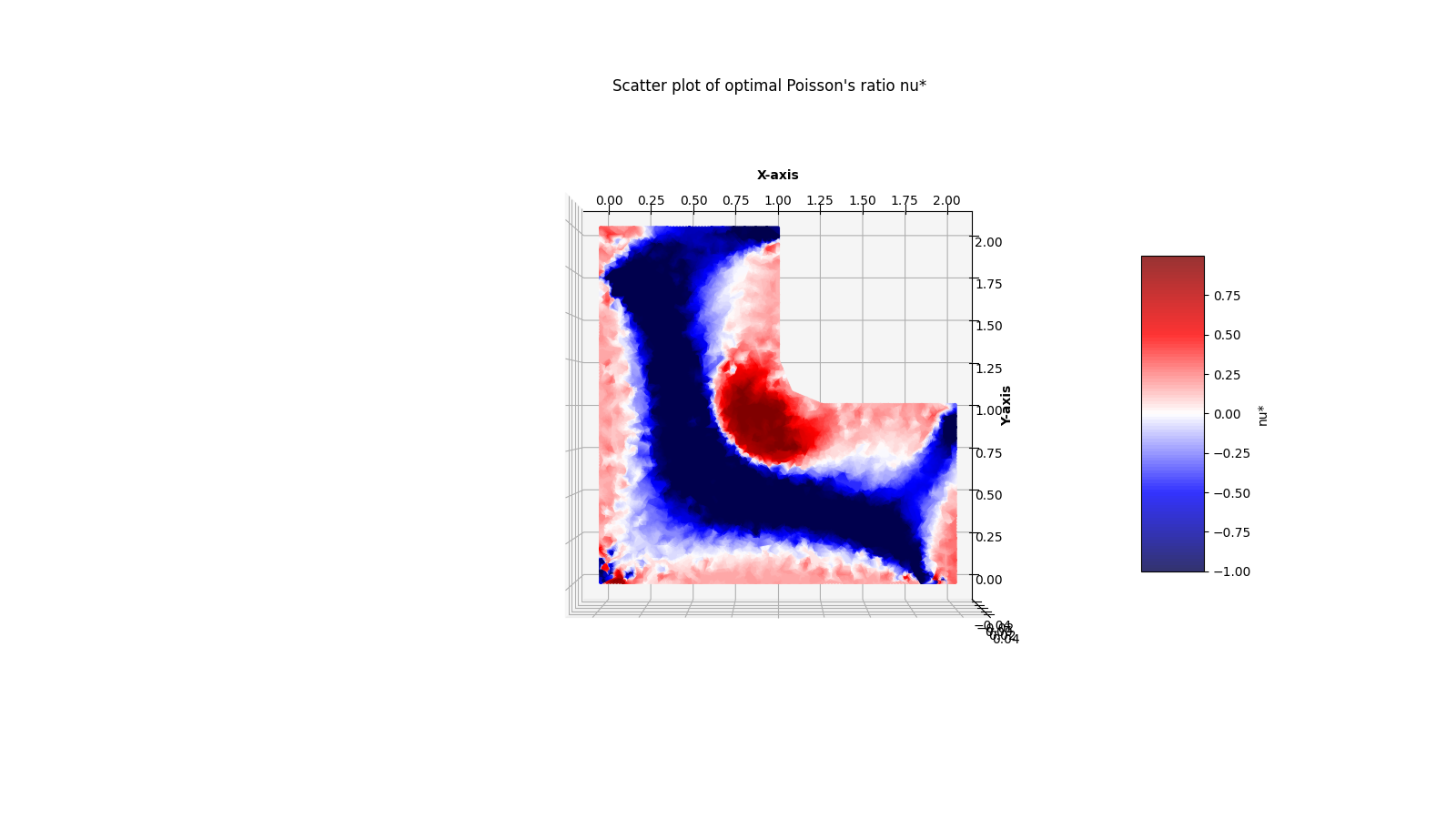}}\\
	\subfloat[vp-IMD, $p=3$]{\includegraphics*[trim={16.3cm 6.2cm 5.4cm 4.7cm},clip,width=0.30\textwidth]{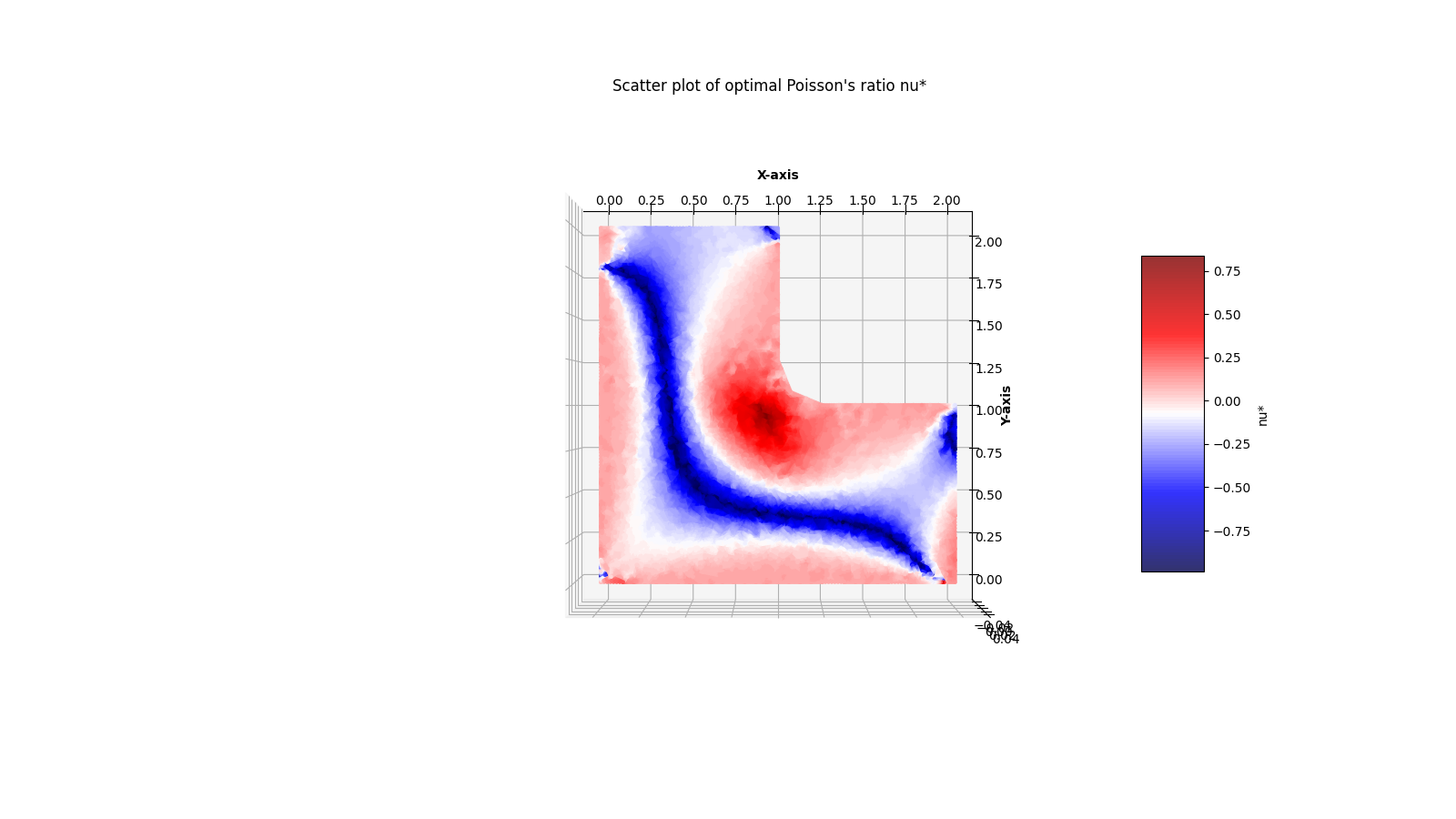}}\hspace{2cm}
	\subfloat[sp-IMD, $p=3$]{\includegraphics*[trim={16.3cm 6.2cm 5.4cm 4.7cm},clip,width=0.30\textwidth]{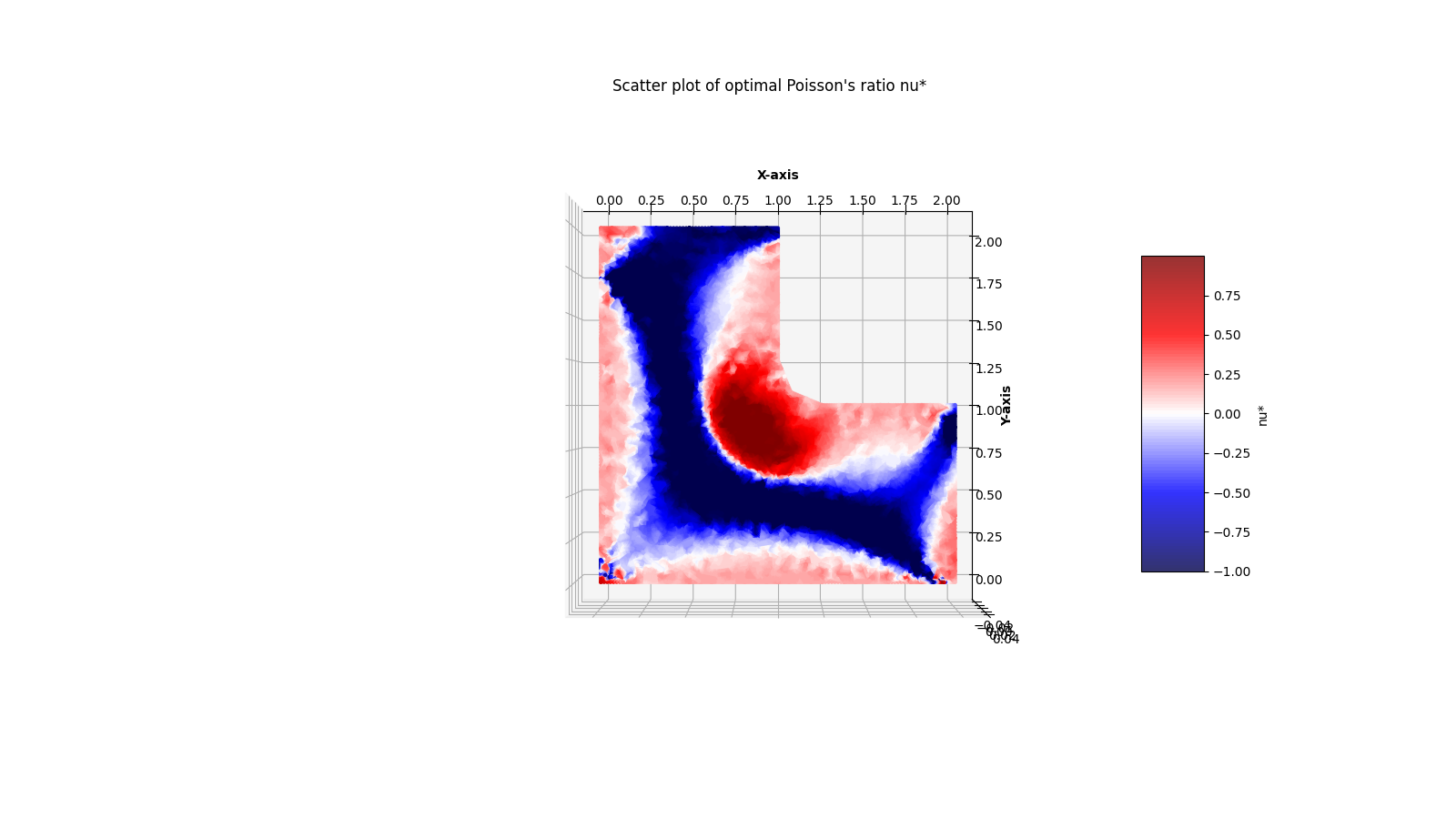}}\\
	\subfloat[vp-IMD, $p=100$]{\includegraphics*[trim={16.3cm 6.2cm 5.4cm 4.7cm},clip,width=0.30\textwidth]{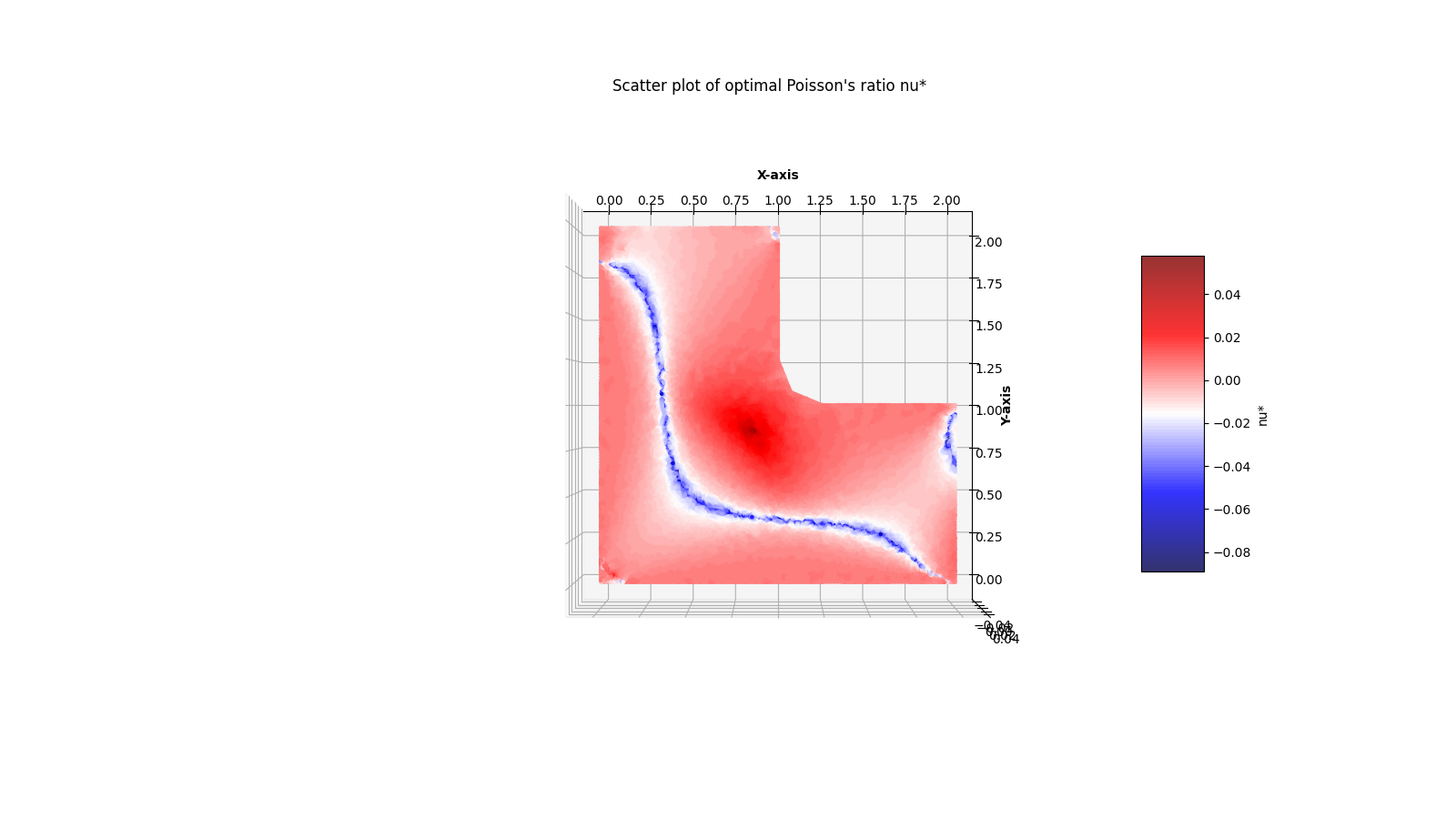}}\hspace{2cm}
	\subfloat[sp-IMD, $p=100$]{\includegraphics*[trim={16.3cm 6.2cm 5.4cm 4.7cm},clip,width=0.30\textwidth]{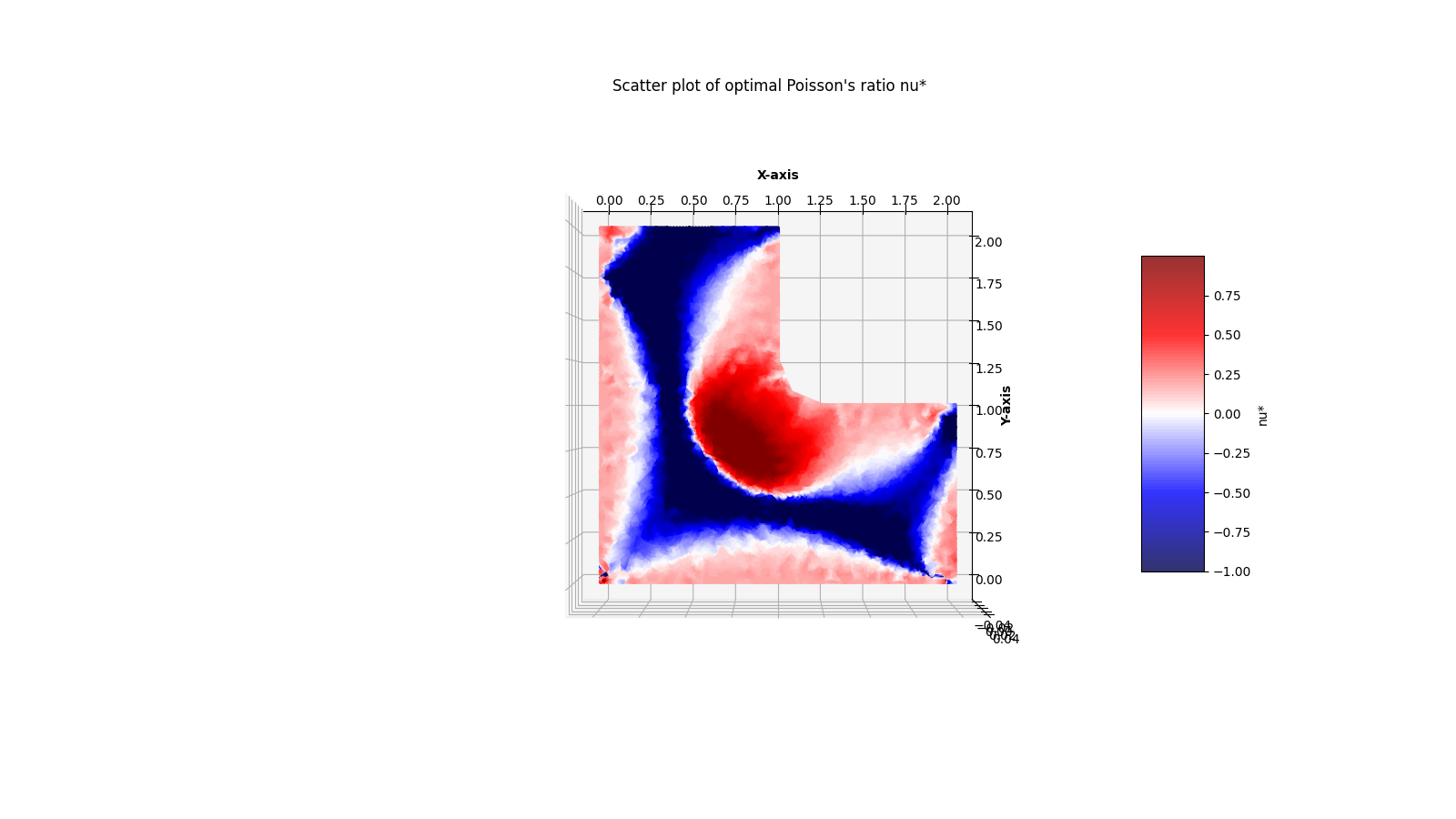}}
	\caption{L-shaped domain problem, optimal Poisson ratio $\hat{\nu}$ for various $p\in[1,\infty)$.}
	\label{fig:nu}       
\end{figure}

First, one can spot the high concentration of the moduli in the vicinity of the reentrant corner for small $p=1,2,3$. This is a manifestation of the fact that exact solutions $(\hat{k},\hat{\mu})$ are elements of $L^p(\O;\R_+)^2$, and no uniform local bounds on the moduli is present in the formulation. Thus, the concentrations occur naturally, since the density of the moduli follow the intensity of the optimal stresses (solving the problems \eqref{eq:primal_vp}, \eqref{eq:primal_sp} rather than the elasticity problem in the homogeneous body), see the formulae in Theorems \ref{thm:stress_vp}, \ref{thm:stress_sp}. We will elaborate on this topic in Section \ref{ssec:vanishing_mass} below.

Increasing the exponent $p$ phases out the foregoing concentrations. This, however, is at the expense of  diffusing the moduli within the whole design domain. Only for $p=1$ the optimization method cuts out a subregion of $\Omega$ around the lower left corner, where $(\hat{k},\hat{\mu})$ is flat zero. Already for small $p>1$ the support of the two functions becomes the whole $\O$. See Section \ref{ssec:cutting-out} for the further discussion.

On another note, for every $p$ one can discern subregions of $\O$ where either $\hat{k}$ or $\hat{\mu}$ vanish. In these regions the Young modulus $\hat{E}$ is zero, whilst the Poisson ratio $\hat{\nu}$ attains one of the extreme values $1$ or $-1$. In particular, it is typical that the optimal design involves subregions where the isotropic material is auxetic, see e.g. \cite{czarnecki2015b}.

Let us now confront the behaviour of the optimization method when $p$ is pushed to infinity with the theoretical predictions from Section \ref{ssec:p=infty}. First, for the (vp-IMD) method, we see that for $p = 100$ both $\hat{k}$ and $\hat{\mu}$ exhibit a relatively small oscillation between, roughly, $90\,000$ N/m and $11\,000$ N/m. Numerical tests have been conducted also for $p=10^6$ for which the optimal moduli $\hat{k}, \hat{\mu}$ were, numerically speaking, constant and equal to $E_0/2$. This confirms the predictions for the vp-IMD method in Section \ref{sssec:trivial}. For the sp-IMD method and $p=100$ one can observe that the plots of $\hat{k}$ and $\hat{\mu}$ are mutually complementary, in the sense that $d\hat{k} +2 \beta^2 \hat\mu = 2 \hat{k}+4 \hat\mu \approx E_0$, which once again tracks with the assertions made in Section \ref{sssec:p_infty_spIMD}.

Another remarkable feature of the sp-IMD variant of the method is worth noting. The point-wise variation of the contribution of $\hat{k}(x)$ and $\hat{\mu}(x)$ to the trace of the stiffness tensor $d\hat{k}(x) +2 \beta^2 \hat\mu(x)$ appears to be very stable with respect to $p$ unless we are not  close to the limit case $p=1$. This can be discerned from the plots of Poisson ratio $\hat{\nu}$, see the right column in Fig. \ref{fig:nu}. The plot barely differs for $p=2,3,100$.

Finally, let us study the variation of the minimal compliances for the two methods, cf. Table \ref{tab:Ldomain}. We can see them monotonincally changing with $p$: $\hat{\mathcal{C}}_{vp}$ decreases, and $\hat{\mathcal{C}}_{sp}$ increases. The second tendency is in line with Proposition \ref{prop:mono_sp}. The bottom line of Table \ref{tab:Ldomain}, which displays $\hat{\mathcal{C}}_{vp}/(1+\beta^2)^{1/p}$, allows to verify Proposition \ref{prop:mono_vp}.  The fact that $\hat{\mathcal{C}}_{vp}$ itself is decreasing reoccurs in the rest of the numerical examples that the authors investigated in their study, yet a theoretical explanation of this phenomenon is, at the moment, out of reach.

\end{example}

\begin{example}[\textbf{Cantilever problem}]
    For the second example we chose a rectangular cantilever subject to a bending load, see Fig. \ref{fig:dom}(b). This time around, a regular mesh of quadrilateral elements is employed, see \cite{czarnecki2021isotropic} for the details of the implementations in this variant of mesh.
    
    The scatter plots of $\hat{k}$ and $\hat{\mu}$ are showed in Figs \ref{fig:k_canti}, \ref{fig:mu_canti}. Table \ref{tab:cantilever} presents how the minimal compliances change with $p$.
    Vast majorities of the comments made in the previous examples smoothly translates here. Therefore, we conclude discussing this examples right away to avoid repetitions.

\begin{table}[h]
	\scriptsize
	\centering
	\caption{Optimal compliance in the cantilever problem for various $p\in[1,\infty)$.}
	\begin{tabular}{l|ccccccc}
		 & $ p=1$ (IMD)  & $ p=2$   & $ p=3$       & $ p=10^2$  & $p=10^6$\\
		\midrule[0.36mm]
		$\hat{\mathcal{C}}_{vp}$ [Nm]
		& 0.167899 & 0.135427 & 0.126563 & 0.11285 & 0.112529 \\
		 $\hat{\mathcal{C}}_{sp}$ [Nm] & 0.167899 & 0.229177 & 0.255042 & 0.321830 & 0.324407
         \\
		 $\hat{\mathcal{C}}_{vp}/(1+\beta^2)^{1/p}$ [Nm] & 0.055966 & 0.078189 & 0.087754 & 0.111629 & 0.112529
	\end{tabular}
	\label{tab:cantilever}
\end{table}

\begin{figure}[h]
\captionsetup[subfloat]{labelformat=simple,farskip=3pt,captionskip=1pt}
	\centering
    \subfloat[IMD, $p=1$]{\includegraphics*[trim={16.3cm 6.2cm 5.4cm 4.7cm},clip,width=0.30\textwidth]{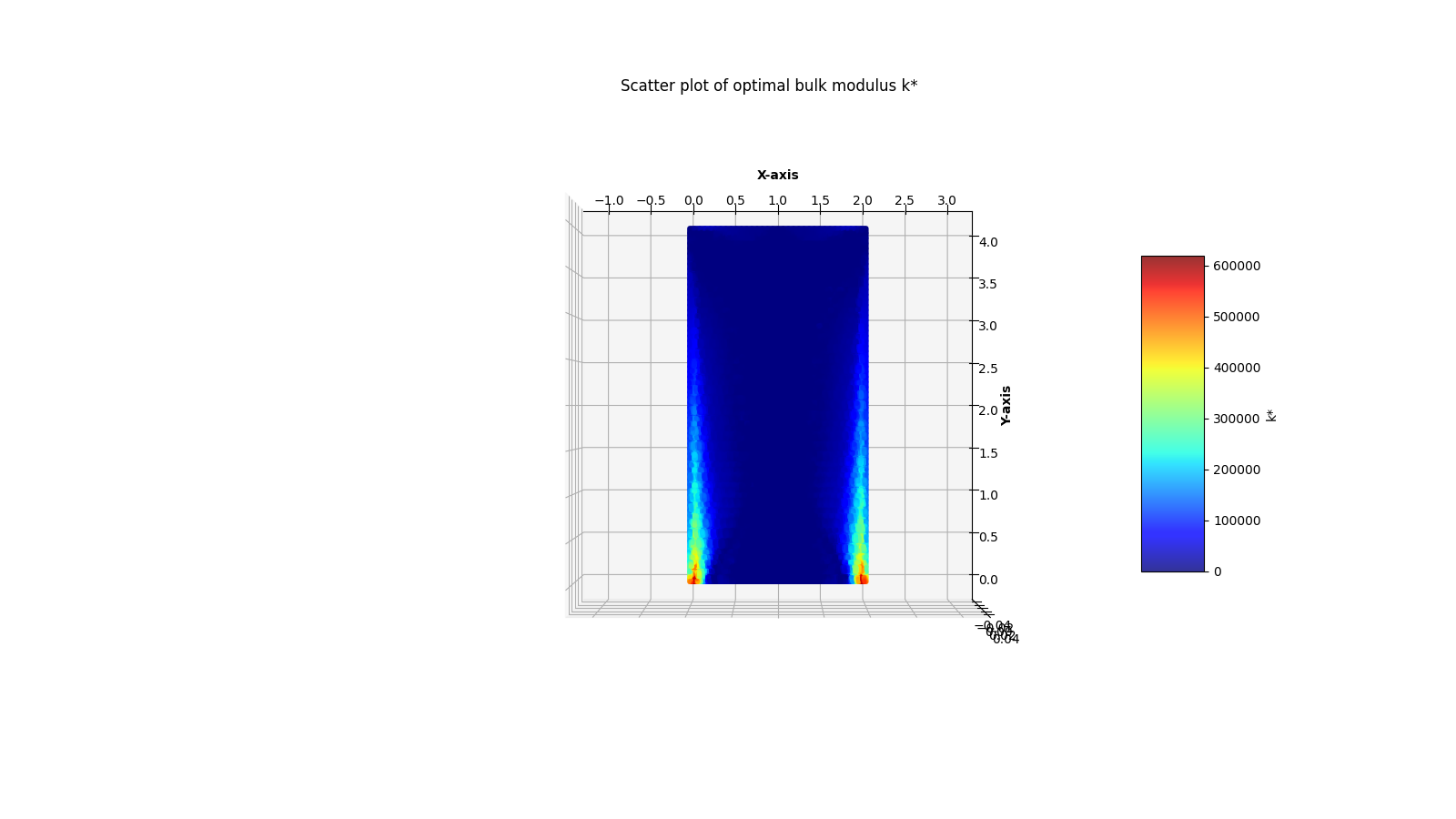}}\\
	\subfloat[vp-IMD, $p=2$]{\includegraphics*[trim={16.3cm 6.2cm 5.4cm 4.7cm},clip,width=0.30\textwidth]{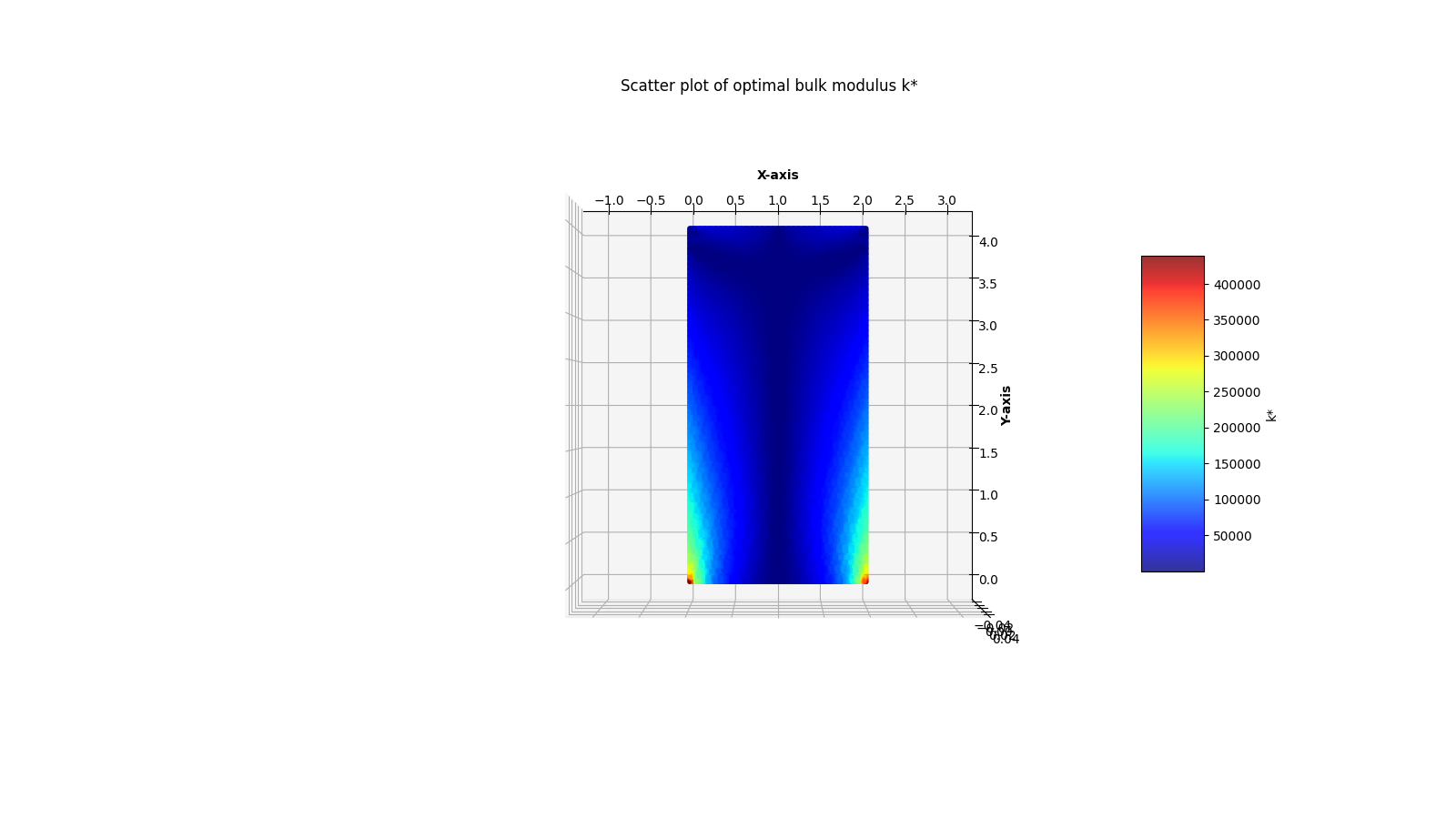}}\hspace{2cm}
	\subfloat[sp-IMD, $p=2$]{\includegraphics*[trim={16.3cm 6.2cm 5.4cm 4.7cm},clip,width=0.30\textwidth]{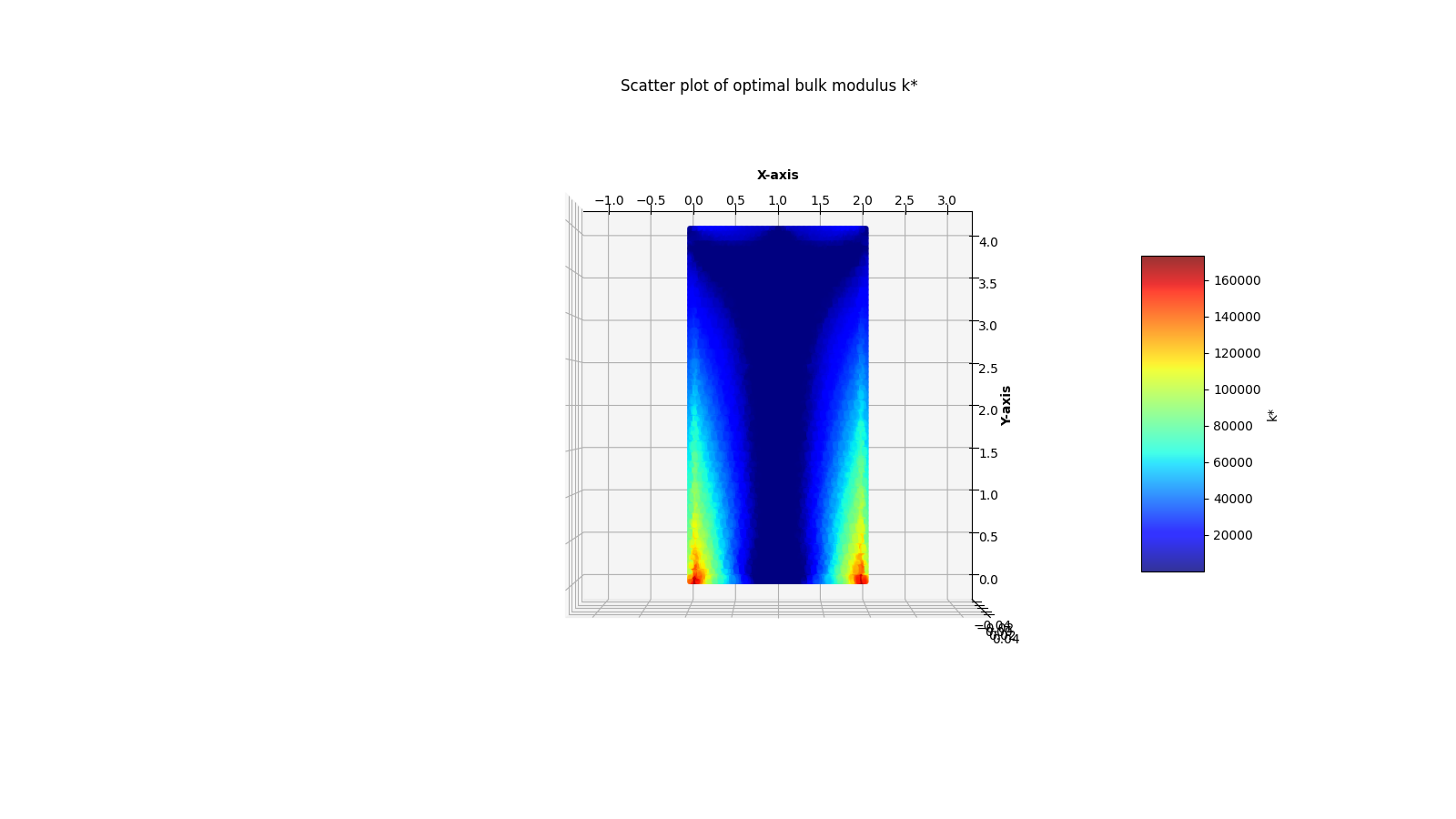}}\\
	\subfloat[vp-IMD, $p=3$]{\includegraphics*[trim={16.3cm 6.2cm 5.4cm 4.7cm},clip,width=0.30\textwidth]{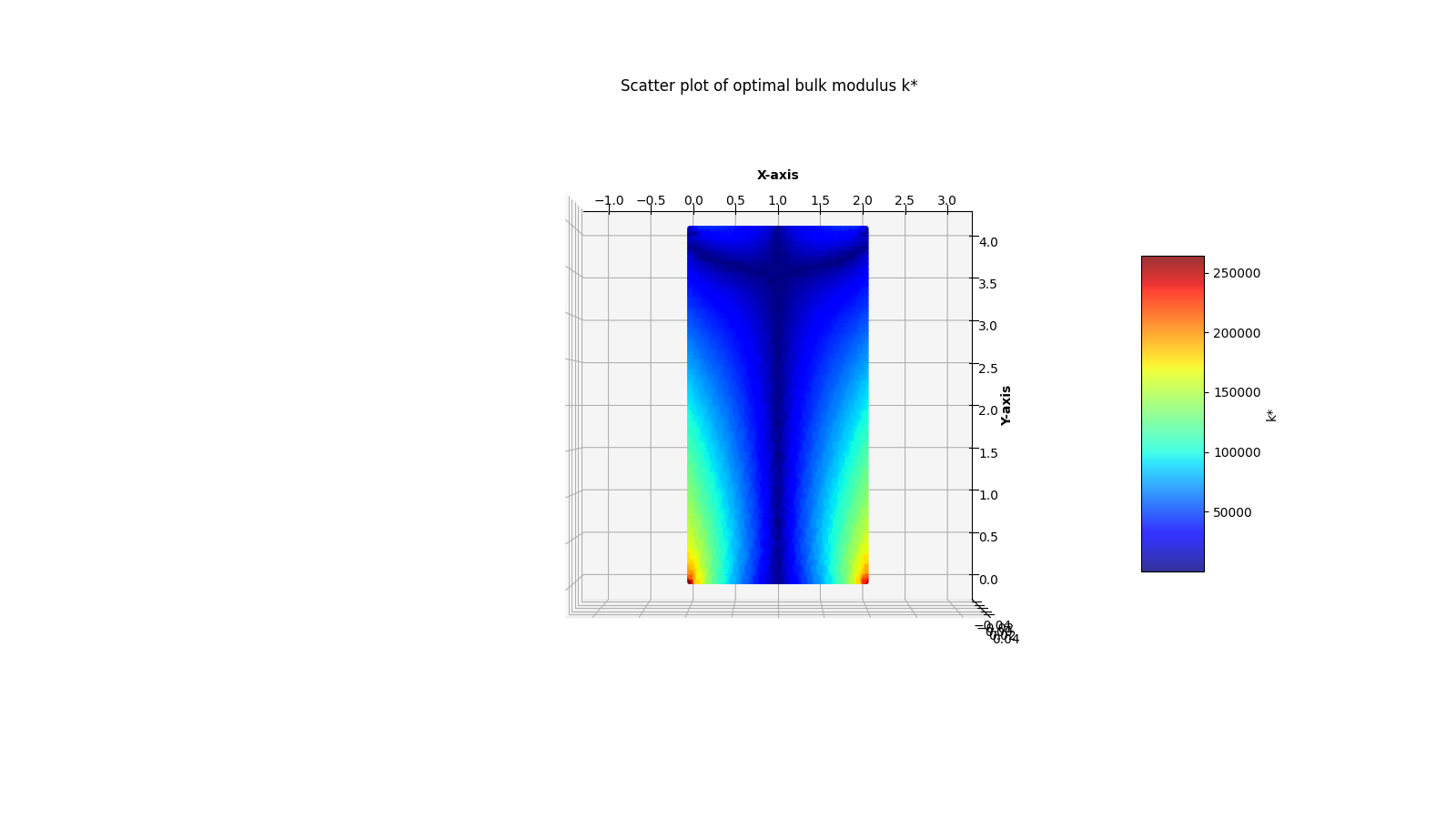}}\hspace{2cm}
	\subfloat[sp-IMD, $p=3$]{\includegraphics*[trim={16.3cm 6.2cm 5.4cm 4.7cm},clip,width=0.30\textwidth]{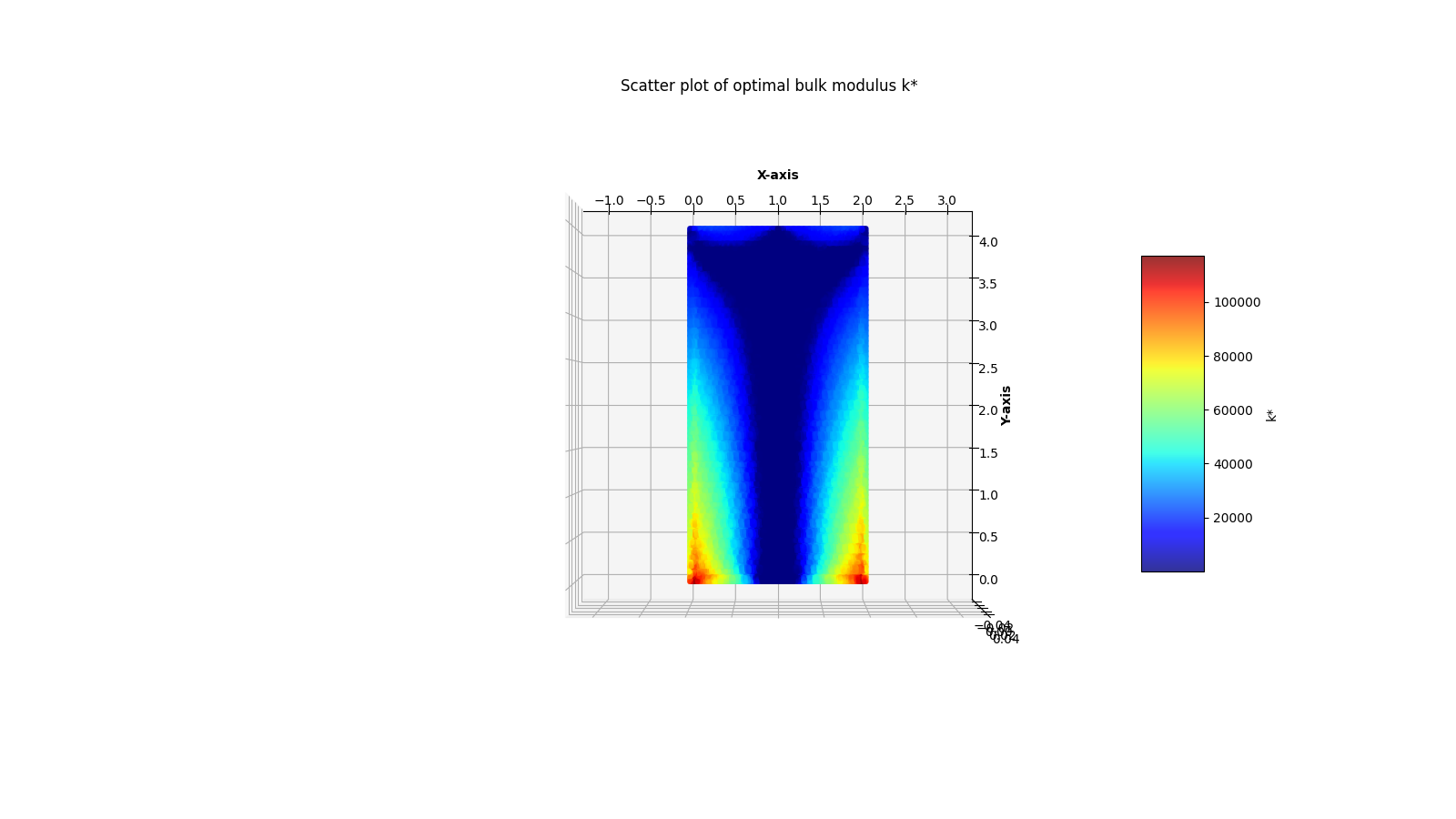}}\\
	\subfloat[vp-IMD, $p=100$]{\includegraphics*[trim={16.3cm 6.2cm 5.4cm 4.7cm},clip,width=0.30\textwidth]{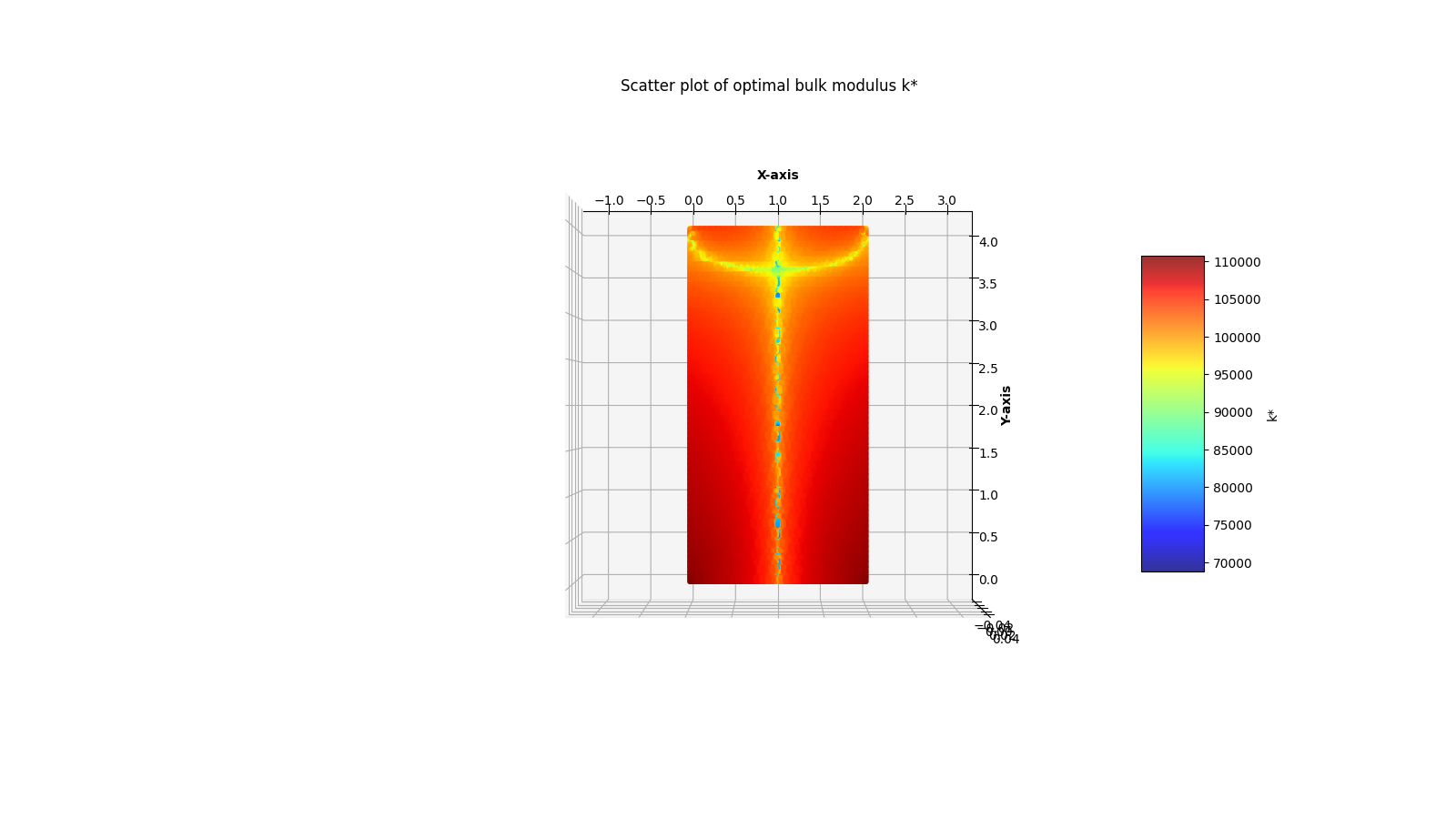}}\hspace{2cm}
	\subfloat[sp-IMD, $p=100$]{\includegraphics*[trim={16.3cm 6.2cm 5.4cm 4.7cm},clip,width=0.30\textwidth]{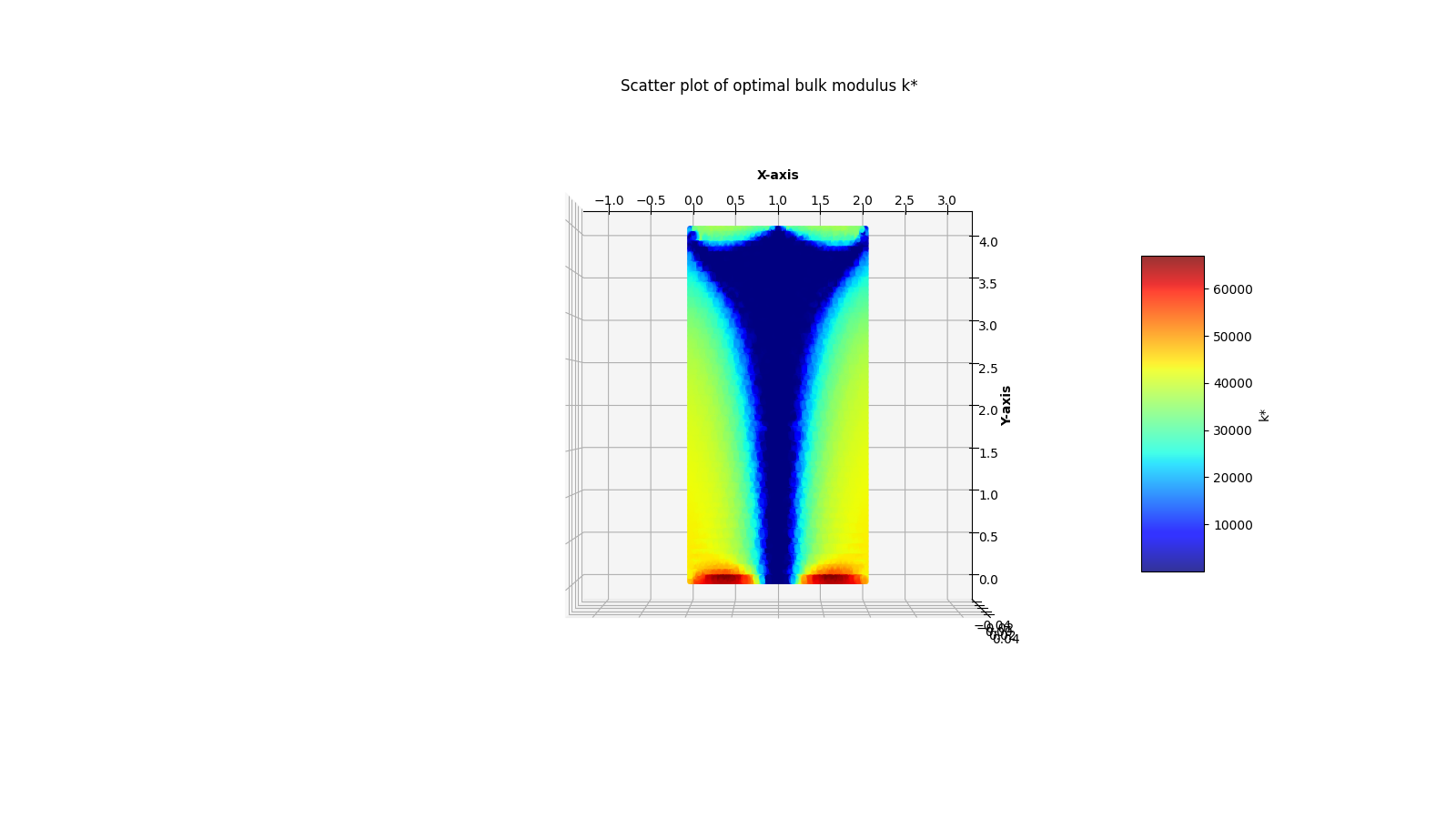}}
	\caption{Cantilever problem, optimal bulk modulus $\hat{k}$ for various $p\in[1,\infty)$.}
	\label{fig:k_canti}       
\end{figure}

\begin{figure}[h]
\captionsetup[subfloat]{labelformat=simple,farskip=3pt,captionskip=1pt}
	\centering
    \subfloat[IMD, $p=1$]{\includegraphics*[trim={16.3cm 6.2cm 5.4cm 4.7cm},clip,width=0.30\textwidth]{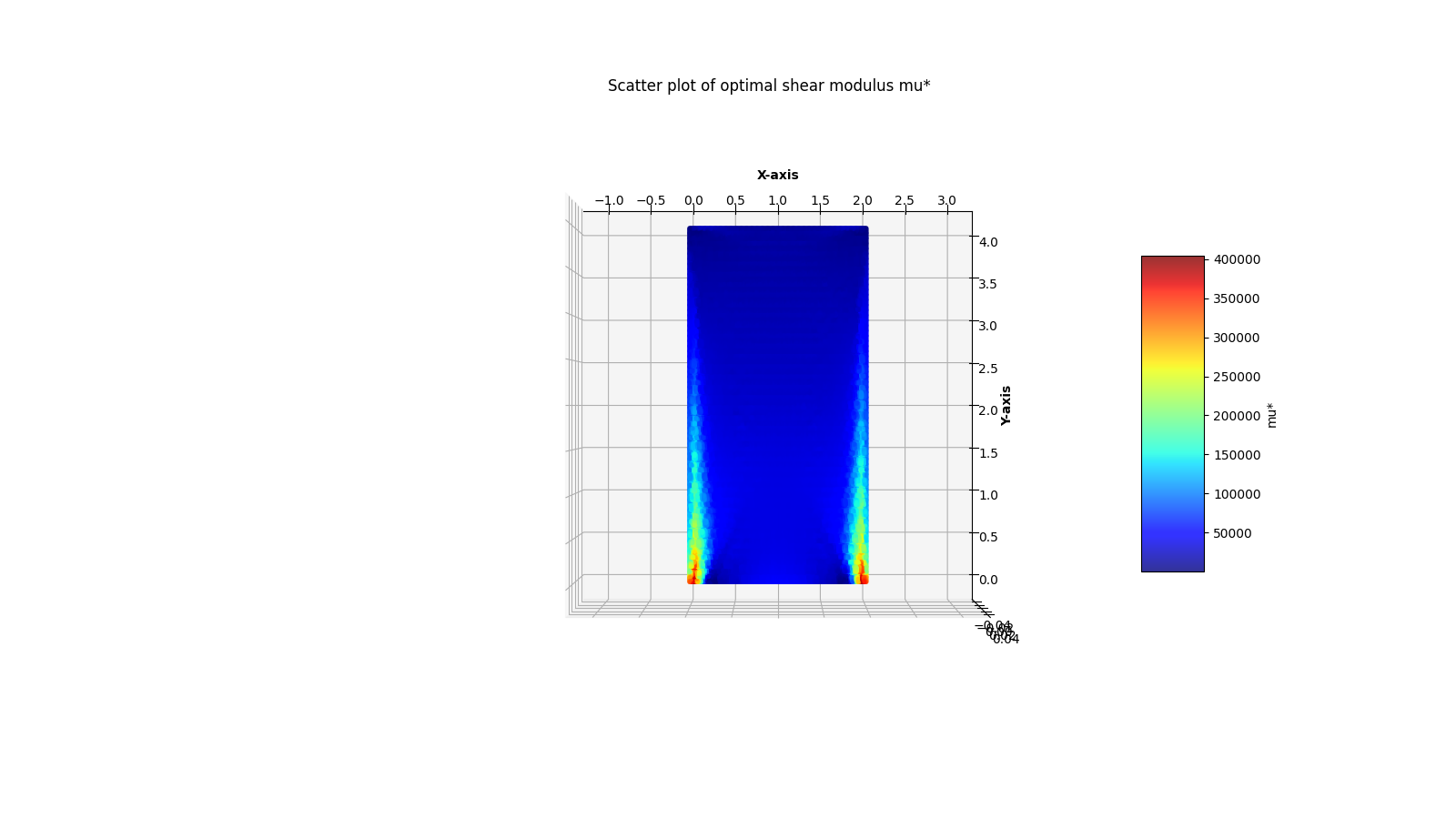}}\\
	\subfloat[vp-IMD, $p=2$]{\includegraphics*[trim={16.3cm 6.2cm 5.4cm 4.7cm},clip,width=0.30\textwidth]{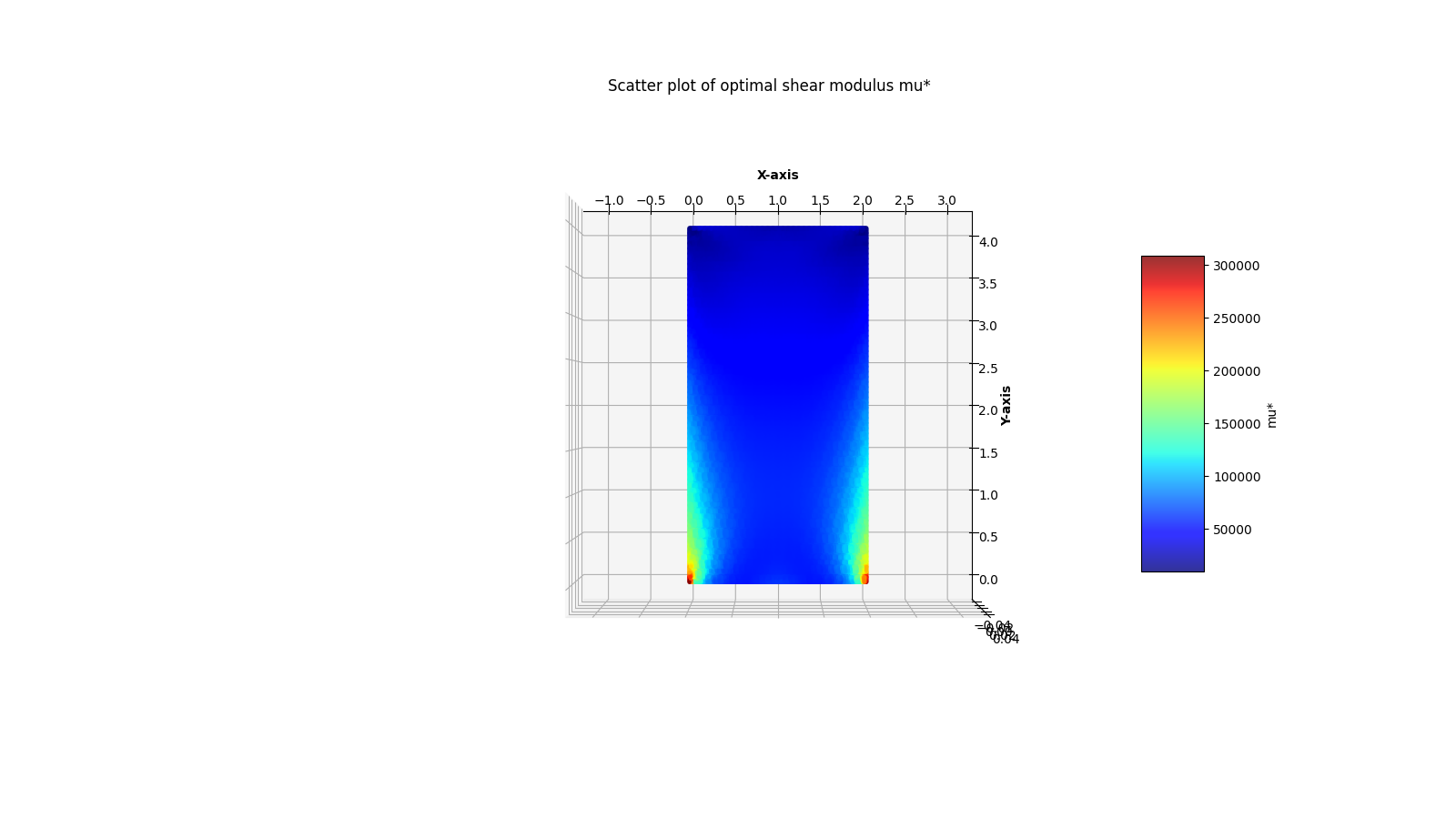}}\hspace{2cm}
	\subfloat[sp-IMD, $p=2$]{\includegraphics*[trim={16.3cm 6.2cm 5.4cm 4.7cm},clip,width=0.30\textwidth]{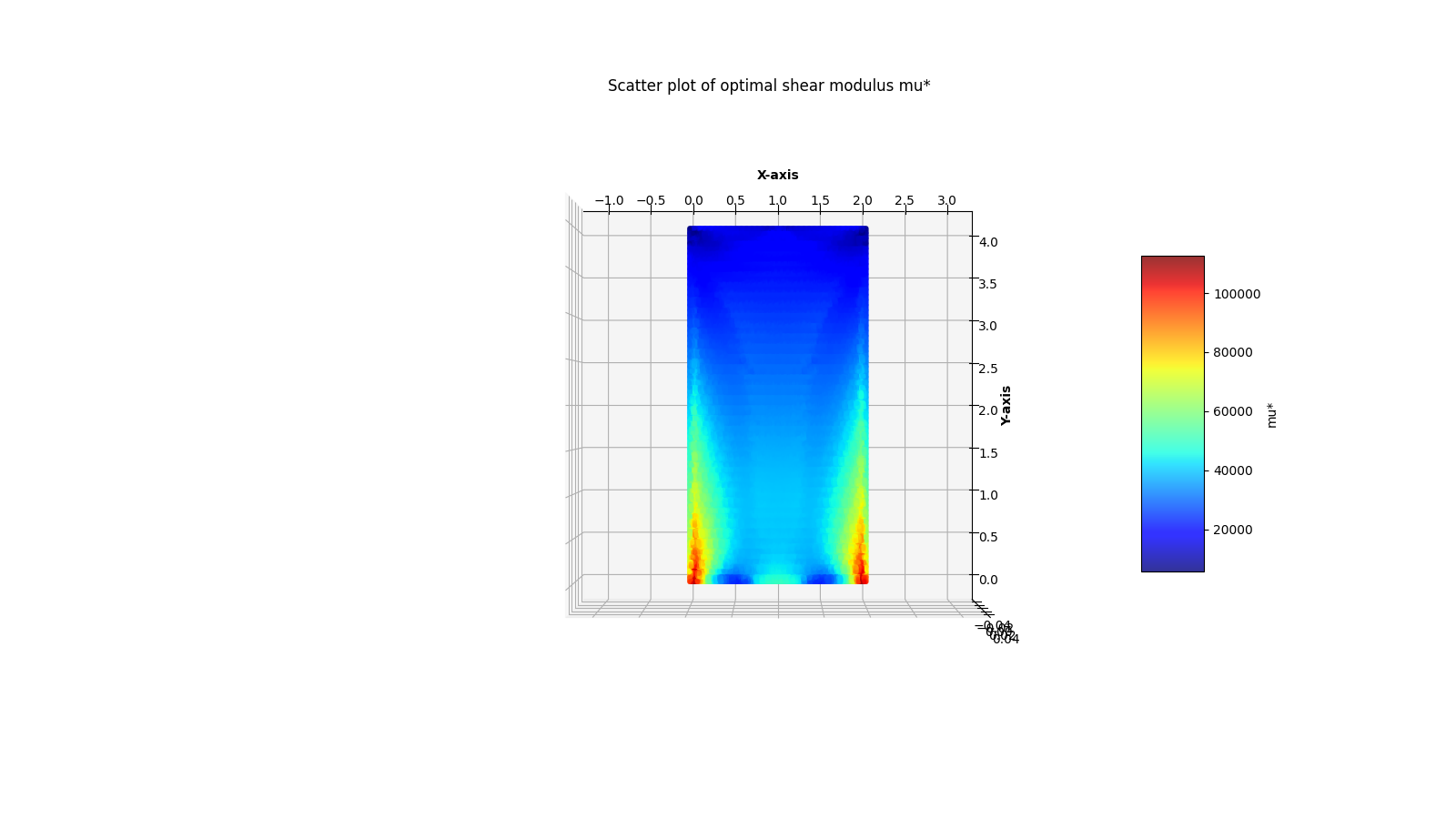}}\\
	\subfloat[vp-IMD, $p=3$]{\includegraphics*[trim={16.3cm 6.2cm 5.4cm 4.7cm},clip,width=0.30\textwidth]{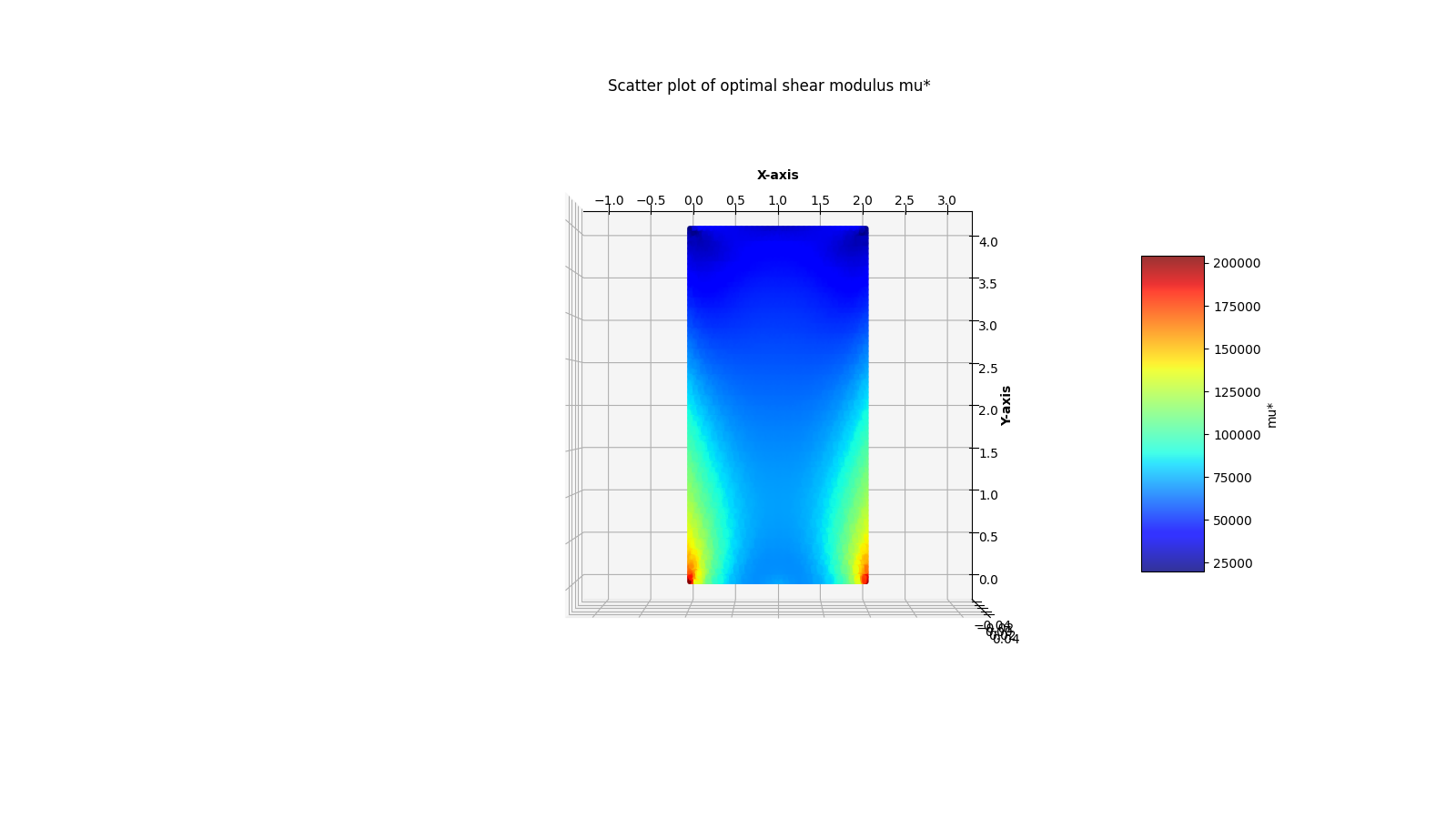}}\hspace{2cm}
	\subfloat[sp-IMD, $p=3$]{\includegraphics*[trim={16.3cm 6.2cm 5.4cm 4.7cm},clip,width=0.30\textwidth]{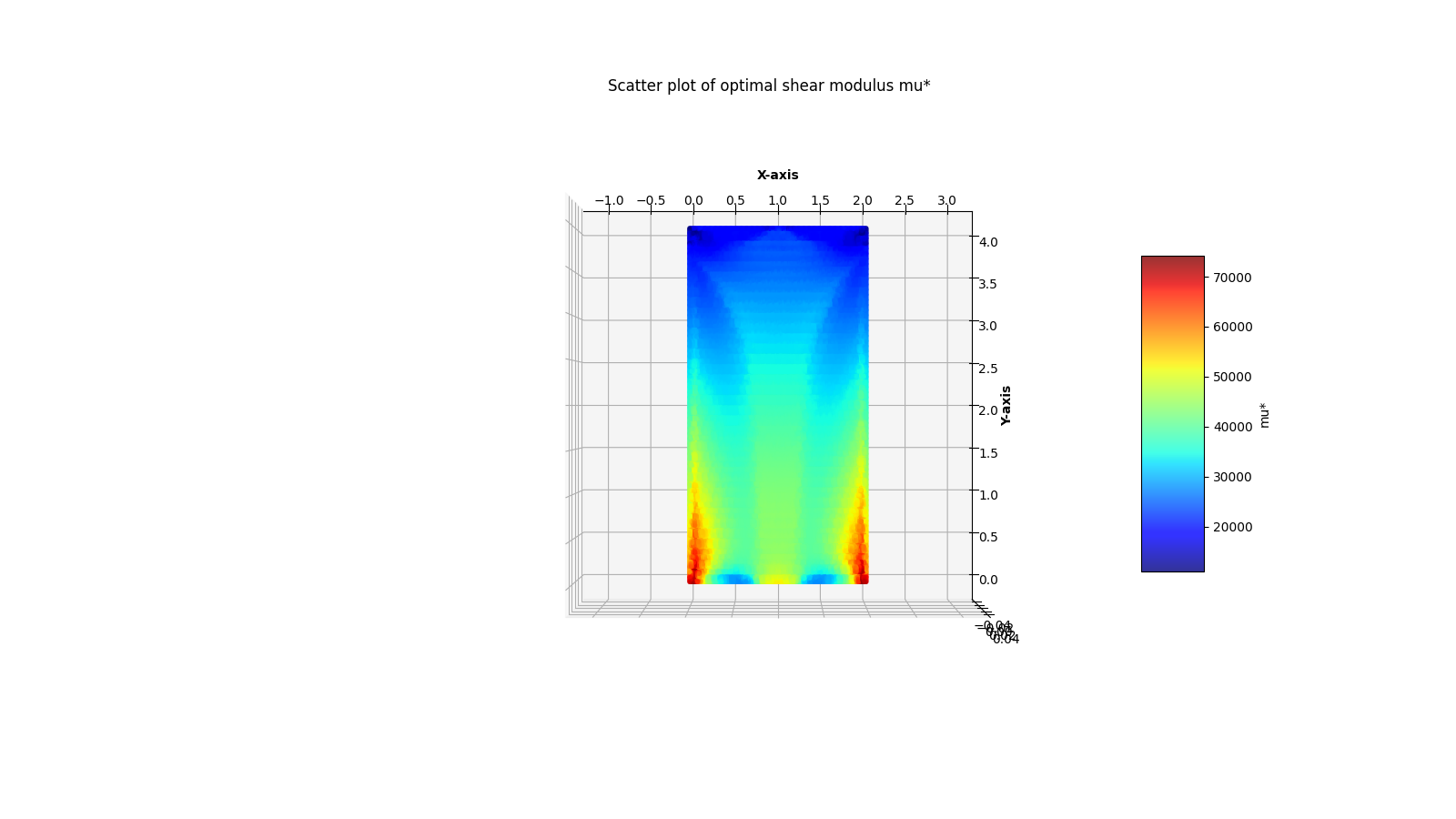}}\\
	\subfloat[vp-IMD, $p=100$]{\includegraphics*[trim={16.3cm 6.2cm 5.4cm 4.7cm},clip,width=0.30\textwidth]{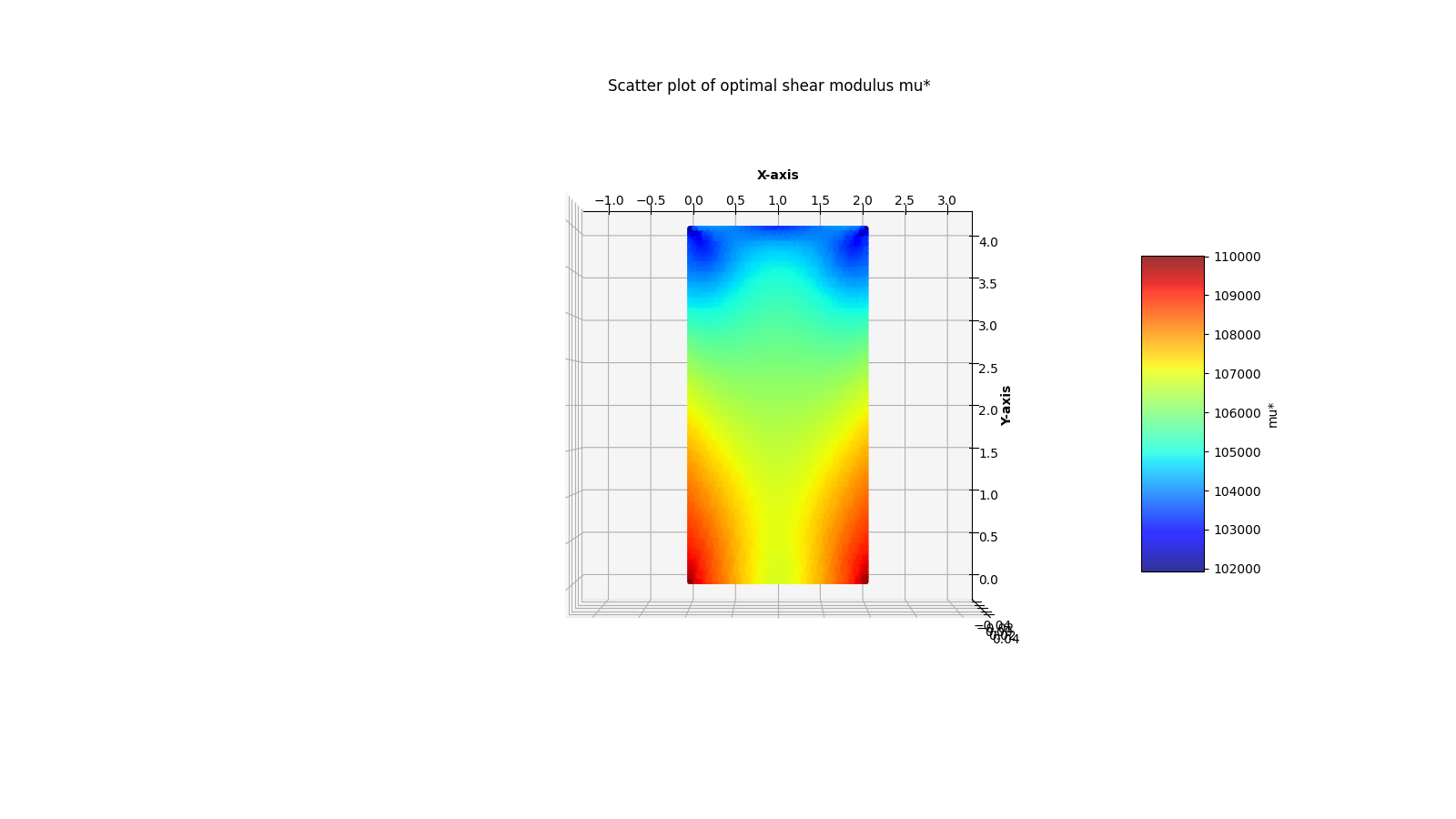}}\hspace{2cm}
	\subfloat[sp-IMD, $p=100$]{\includegraphics*[trim={16.3cm 6.2cm 5.4cm 4.7cm},clip,width=0.30\textwidth]{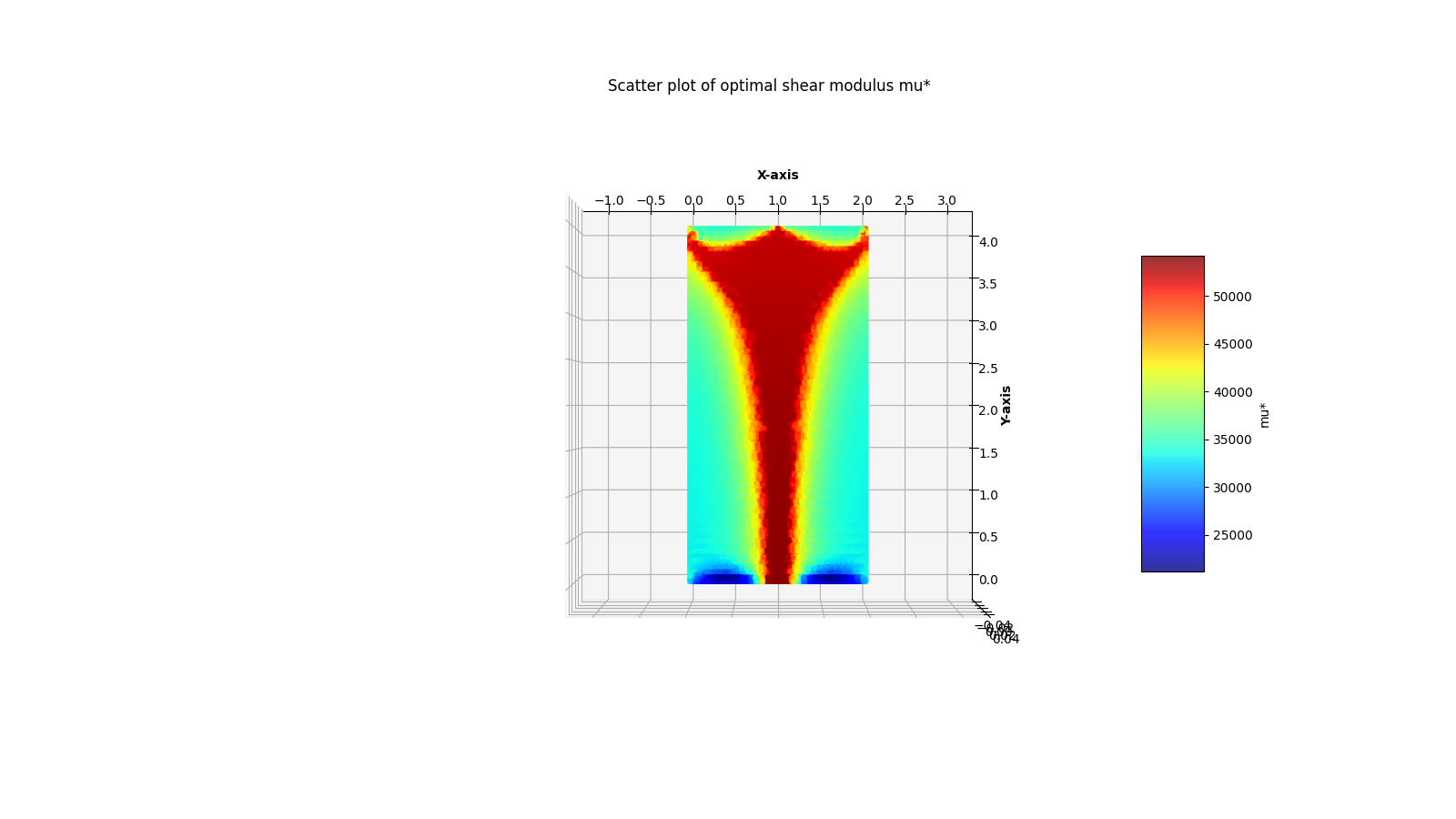}}
	\caption{Cantilever problem, optimal shear modulus $\hat{\mu}$ for various $p\in[1,\infty)$}
	\label{fig:mu_canti}       
\end{figure}

\end{example}

\section{Final remarks and discussion}

This paper puts forward a method of designing an isotropic 2D/3D body that minimizes compliance under the new class of constraints on the amount of the used material. The design variables are the bulk and shear moduli $k,\mu:\Omega \to \R_+$ whose cost is expressed by their $L^p$ norm for a chosen $p \in [1,\infty]$. Two variants of the methods are proposed, differing by the local norm of the pair $(k,\mu)\in \R^2$.

For both variants, a rigorous mathematical theory is developed for $p\in (1,\infty)$, with an easy extension to $p=\infty$. This involves reformulation to a pair of auxiliary convex variational problems -- stress-based formulation and its displacement-based dual -- together with the existence results. For the next step, another framework was found, which entails a system of equations (or inclusions) of non-linear elasticity. Finally, the numerical simulations have been carried out through tackling the stress-based formulations by using the original finite element approach leading to an unconstrained convex optimization problem.

To conclude this work, we shall reflect on: the significance of its contribution, the remaining challenging issues, and, finally, the possible developments to be potentially addressed in the future. 

\subsection{Advantages over the original IMD method}

The methods vp-IMD and sp-IMD are inspired by the Isotropic Material Design method introduced for the first time in \cite{czarnecki2015a}. The original IMD method poses a bunch of mathematical challenges that were summed up above in Section \ref{ssec:IMD_recovery}. Above all, to arrive at a well-posed problem, the search for both optimal moduli and the stresses has to be extended to the space of Radon measures, cf. \cite{bolbotowski2022b}.
It renders the analysis of the optimal design problem much more subtle. In particular, in order to write down the system of optimality conditions, the classical notion of Sobolev space does not suffice. A feasible approach is to employ the apparatus of the $\lambda$-tangential differentiation and the related Sobolev space with respect to measure $W^{1,q}_\lambda(\O;\Rd)$, see \cite{bouchitte1997}. This is a very technical topic that lies well outside the scope of the standard calculus of variations.

The present work presents the methods that are similar in spirit to the IMD method, yet they circumvent the delicate mathematical issues expounded above. Both optimal moduli and stresses are to be found in the space of integrable functions. Moreover, the optimal displacement functions $\hat{\mbf{v}}$ lie in the classical Sobolev space. Operating with the classical weak derivatives of $\hat{\mbf{v}}$ is sufficient for a rigorous formulation of the system of optimality conditions as in \eqref{eq:non-linear_system_vp} and \eqref{eq:non-linear_system_sp}. If one is to entrust Conjecture \ref{conj:p=1}, the solutions of the problems (vp-IMD) and (sp-IMD) found for $p$ close to $1$  should approximate well the solutions of the original (IMD) problem. Accordingly, the two new methods constitute a by-pass for the mathematical complexity of the original method.

\subsection{(vp-IMD) versus (sp-IMD)}

The mathematical advantages of the new methods can be taken even further in the case of the vp-IMD methods. As stated in Theorem \ref{thm:stress_vp} and Corollary \ref{cor:dual_vp}, the optimal moduli $\hat{k},\hat{\mu}$, the stress $\hat{\boldsymbol{\tau}}$, and the displacement $\hat{\mbf{v}}$ are all unique. This could not be asserted for the sp-IMD method. It is related to the non-smoothness of the cost condition for the latter method. This non-smoothness further translates to the auxiliary functional in the stress-based problem \eqref{eq:primal_sp}, which, in turn, results in a non-smooth optimization problem after the discretization. Its counterpart is smooth in the case of the vp-IMD method, see Section \ref{ssec:implementation}.

The non-smoothness of the cost in the sp-IMD method also manifests itself in the system of the optimality conditions. In \eqref{eq:non-linear_system_sp} we have seen that, for the (sp-IMD) problem, it involves the strain-stress relations:
\begin{equation}
	\label{eq:inclusions_p}
	\begin{cases}
		(3.a) & \mathrm{Tr}\,\boldsymbol{\eps}(\hat{\mbf{v}}) \in \big(\abs{\mathrm{Tr}\,\hat{\boldsymbol{\tau}}}+\beta\,\norm{\mathrm{dev}\,\hat{\boldsymbol{\tau}}}\big)^{r-1} \mathcal{N} (\mathrm{Tr}\,\hat{\boldsymbol{\tau}}),\\
		(3.b) & \mathrm{dev}\,\boldsymbol{\eps}(\hat{\mbf{v}}) \in  \big(\abs{\mathrm{Tr}\,\hat{\boldsymbol{\tau}}}+\beta\,\norm{\mathrm{dev}\,\hat{\boldsymbol{\tau}}}\big)^{r-1} \beta\, \mathcal{N}(\mathrm{dev}\,\hat{\boldsymbol{\tau}}).
	\end{cases}
\end{equation}
Above, $\mathcal{N}$ is a multi-valued operator, cf. \eqref{eq:N}, hence we are facing a pair of inclusions. As such, they are difficult to implement numerically. Contrarily, the smooth cost condition for the (vp-IMD) problems has lead us to the system \eqref{eq:non-linear_system_vp}, where these relations read:
\begin{equation}
	\label{eq:power-law}
	\begin{cases}
		(3.a) & \mathrm{Tr}\,\boldsymbol{\eps}(\hat{\mbf{v}}) = \abs{\mathrm{Tr}\,\hat{\boldsymbol{\tau}}}^{r-2} \mathrm{Tr}\,\hat{\boldsymbol{\tau}},\\
		(3.b) & \mathrm{dev}\,\boldsymbol{\eps}(\hat{\mbf{v}}) = \big(\beta\,\norm{\mathrm{dev}\,\hat{\boldsymbol{\tau}}}\big)^{r-2} \mathrm{dev}\,\hat{\boldsymbol{\tau}}.
	\end{cases}
\end{equation}
In the literature these are known as the \textit{power-law} constitutive equation, cf. \cite{castaneda1998}. Together with the equilibrium equation, the optimality conditions  \eqref{eq:non-linear_system_vp}  for the vp-IMD method is, virtually, a system of non-linear elasticity equations. This perspective opens up new alternative numerical methods. This system could be potentially tackled by an off-the-shelf finite element method software like, e.g., Abaqus\textsuperscript{\tiny\textregistered}. Accordingly, running the computations for $p \approx 1$ could furnish a method of solving the original (IMD) problem that would be widely available. Currently, however, no such simulations have been carried out by the present authors.

\subsection{Losing the cutting-out property}
\label{ssec:cutting-out}
By contrast, let us now discuss the drawbacks of the methods vp-IMD/sp-IMD (for $p>1$) comparing to IMD ($p=1$). First, since only the IMD method models the moduli $k,\mu$ as positive measures, this method exclusively provides us with the information on the dimensionality of the body at each point of the domain $\O$, at least on the theoretical level.  Namely, the IMD method tells the designer where lower dimensional stiffeners should be employed.

More importantly, the methods vp-IMD/sp-IMD lacks the crucial \textit{cutting-out property} when $p>1$. This is to say that, for a generic support and loading conditions, the optimal body has a full support in $\O$, i.e. there holds $(\hat{k},\hat{\mu}) \neq (0,0)$ a.e. in $\O$. This could be discerned in the conducted numerical simulations. Let us now convey an intuition behind this phenomenon. From Theorems \ref{thm:stress_vp}, \ref{thm:stress_sp} we can see that  $(\hat{k}(x),\hat{\mu}(x)) = (0,0)$ if and only if $\hat{\boldsymbol{\tau}}(x)=0$ for the unique optimal stress. In turn, from either of the stress-strain relations \eqref{eq:inclusions_p}, \eqref{eq:power-law}, we find that this can happen only at points where $\boldsymbol{\eps}(\hat{\mbf{v}})(x) =0$. Therefore, for the optimal moduli to vanish on a subregion of $\Omega$ with non-zero measure, the displacement $\hat{\mbf{v}}$ must be a rigid motion in that region, i.e. $\hat{\mbf{v}}(x) = \mbf{Q}x +\mbf{v}_0$, where $\mbf{v}_0\in \Rd$ and $\mbf{Q}$ is a skew-symmetric matrix. Next, from the displacement compatibility conditions on the interface between such region and a region where the material does occur, we deduce that the tangential component of the displacement also follows this formula. In general, this is a tough condition to meet. Take, for instance, the L-shaped domain problem and the left bottom corner where this interface should be a nearly circular arc, see the top plot in Fig. \ref{fig:nu}. The tangential strain would surely be non-zero in this case. In summary, cutting out a subdomain implies conditions that are, in general, difficult to satisfy when $p>1$.

Let us now explain how the IMD method ($p=1$) slips the foregoing limitations. The formulas \eqref{eq:opt_kmu_tau_IMD} again say that a subdomain with zero moduli implies that $\hat{\boldsymbol\tau} = 0$ in this subdomain. What differs, is the stress-strain relation, which (for $\hat{\mbf{v}} \in C^1$) read
\begin{equation}
	\begin{cases}
		(3.a) &  \mathrm{Tr}\,\boldsymbol{\eps}(\hat{\mbf{v}}) \in  \mathcal{N} (\mathrm{Tr}\,\hat{\boldsymbol{\tau}}),\\
		(3.b) &  \mathrm{dev}\,\boldsymbol{\eps}(\hat{\mbf{v}}) \in  \beta\, \mathcal{N}(\mathrm{dev}\,\hat{\boldsymbol{\tau}}),
	\end{cases}
\end{equation}
which is the relation \eqref{eq:inclusions_p} particularized for $r=1$, recall Section \ref{ssec:IMD_recovery}. From the definition \eqref{eq:N} of $\mathcal{N}$ we can see that if $\hat{\boldsymbol{\tau}}(x) =0$, then the above inclusions merely imply that: $\abs{\mathrm{Tr}\,\boldsymbol{\eps}(\hat{\mbf{v}})(x)} \leq 1$ and $\norm{\mathrm{dev}\,\boldsymbol{\eps}(\hat{\mbf{v}})(x)} \leq \beta$. This is, of course, nowhere near as stringent as $\boldsymbol{\eps}(\hat{\mbf{v}})(x)=0$. As a result, there is greater flexibility in forming $\hat{\mbf{v}}$ in the subregions where the material vanishes, thus paving a way to cutting the material out in the solutions of the (IMD) problem.

To sum up, the two new methods vp-IMD and sp-IMD both lack the feature of selecting a strict subdomain of $\Omega$ in which the material is concentrated when $p>1$ is chosen. In other words, they are neither the topology nor the shape optimization methods.

\subsection{Singularities by concentration of the optimal moduli}
\label{ssec:vanishing_mass}

This next issue is shared by both the original IMD method and the vp-IMD/sp-IMD method provided that $p<\infty$. As was observed for the numerical simulations, the optimal moduli $(\hat{k},\hat{\mu})$ have a tendency to blow up to infinity, e.g., in the 2D case, at reentrant corners and other places where e.g. the loadings are concentrated or where the boundary conditions change abruptly. Essentially, this arises from the mathematical formulation: the form of the cost condition enforces admitting the moduli functions in the $L^p$ space. On the other hand, such phenomenon seems to escape the realm of feasible materials: the stiffness tensor should be bounded. An easy fix to this issue could be to add a uniform local bound on the trace of the stiffness tensor, namely, for $t_{\max} >0$,
\begin{equation}
    \label{eq:local}
    dk(x)+\beta^2 2\mu(x) \leq t_{\max} \qquad \text{for a.e. $x \in\O$.}
\end{equation}
This strategy was undertaken for the free material design problem, e.g., in \cite{haslinger2010}. This extra local constraint would result in a well-posed problem with the optimal solutions $(k,\mu) \in L^\infty(\Omega;\R_+)^2$, whatever is $p\in[1,\infty]$. This modification, however, would spoil the passage to the pair of mutually dual auxiliary problems, which is the very core of this paper.

Below we will engage an informal discussion to show that our method does make physical sense after all, along with the singular solutions that it produces.  Assume that we do add the local constraint \eqref{eq:local} to our formulations vp-IMD/sp-IMD. Then, it is clear that the solutions will depend on the dimensionless ratio $\frac{E_0}{t_{\max}}$, recall the global constraints \eqref{eq:vpcost_E0}, \eqref{eq:spcost_E0}. With $t_{\max}$ fixed, let us consider a sequence of the modified (sp-IMD) (or (vp-IMD)) problems for $E_{0,n} = \frac{1}{n}E_{0} \to 0$, with $(\hat{k}_n,\hat{\mu}_n) \in L^\infty(\Omega;\R_+)^2$ being the respective solutions. Doing so, we simulate the \textit{vanishing mass regime}, well studied in the literature for different optimal design formulations, see e.g. 
\cite{bouchitte2020,bouchitte2007}. Naturally, the measure of the subset $\O_n \subset \O$ on which \eqref{eq:local} is saturated must converge to zero. In other words, in the limit the local bound $t_{\max}$ should be no longer binding. Clearly, however, $(\hat{k}_n,\hat{\mu}_n) \to 0$ in $L^p$. Thus, in order to have a meaningful asymptotic solution we must rescale the sequence of the moduli to $(n\hat{k}_n,n\hat{\mu}_n)$ which satisfy the global cost constraint with the bound $E_0$. In the limit of such a rescaled sequence we can expect to recover a solution to the original (vp-IMD) or (sp-IMD) problem, that is without the local constraint.

Naturally, the foregoing remarks are heuristic. In order to make them rigorous a $\Gamma$-convergence analysis similar to the one in \cite{bouchitte2007dimred} could be conducted. This is a delicate problem that once again goes beyond the scope of this paper. We leave the reader with the conjecture:
\begin{conjecture}
    Assume $p \in(1,\infty)$ and $t_{\max}>0$. For any natural $n$ consider $(\hat{k}_n,\hat{\mu}_n) \in L^\infty(\Omega;\R_+)^2$ to be solutions of the (vp-IMD) problem (resp. the (sp-IMD) problem) for the reference modulus $E_{0,n} = \frac{1}{n}E_0$ and with the additional local constraint \eqref{eq:local}.
    
    Then, up to choosing a subsequence, $(n\hat{k}_n,n\hat{\mu}_n) \rightharpoonup (\hat{k},\hat{\mu})$ weakly in $L^p(\O;\R^2)$ where $(\hat{k},\hat{\mu})$ solves the unmodified problem (vp-IMD) (resp. the problem (sp-IMD)).
\end{conjecture}
\noindent For $p =1$ a similar statement holds, yet with the weak-* convergence in the space of measures.

Assuming that the conjecture is true, the interpretation of the vp-IMD/sp-IMD method follows. It furnishes an optimal distribution of the moduli under the assumption that there is "very little material" to be utilized. In such a case, the material needs to be concentrated at the places where it is the most needed.
Such a premise is valid, for instance, in civil engineering, where usually only a small fraction of the design domain is occupied by the material. This observation resounds especially for the choice $p=1$, for which the present methods can predict 1D stiffeners. After all, it is typical to design the bearing system of the buildings as framed bar structures.

\smallskip

Nevertheless, there is a way to dispose of the material singularities in the optimal designs while maintaining the main ideas presented in this work. The paper  \cite{czarnecki2021isotropic} introduced the yield condition into the IMD modelling which alleviates all the components of statically admissible stress tensor and cuts extremes of all components of the stress field solving the problem (IMD) ($p=1$), hence making bounded all layouts of the optimal elastic moduli. In Fig. \ref{fig:mises}, the influence of introducing the yield condition at the yield stress value $\sigma _0 = 50\left[ {{\rm{MPa}}} \right]$				
on the distribution of optimal moduli is demonstrated; the L-shaped plate is supported and loaded as in Fig. \ref{fig:dom}(a).

\begin{figure}[h]
\captionsetup[subfloat]{labelformat=simple,farskip=3pt,captionskip=1pt}
	\centering
	\subfloat[(a)]{\includegraphics*[trim={16.3cm 6.2cm 5.4cm 4.7cm},clip,width=0.30\textwidth]{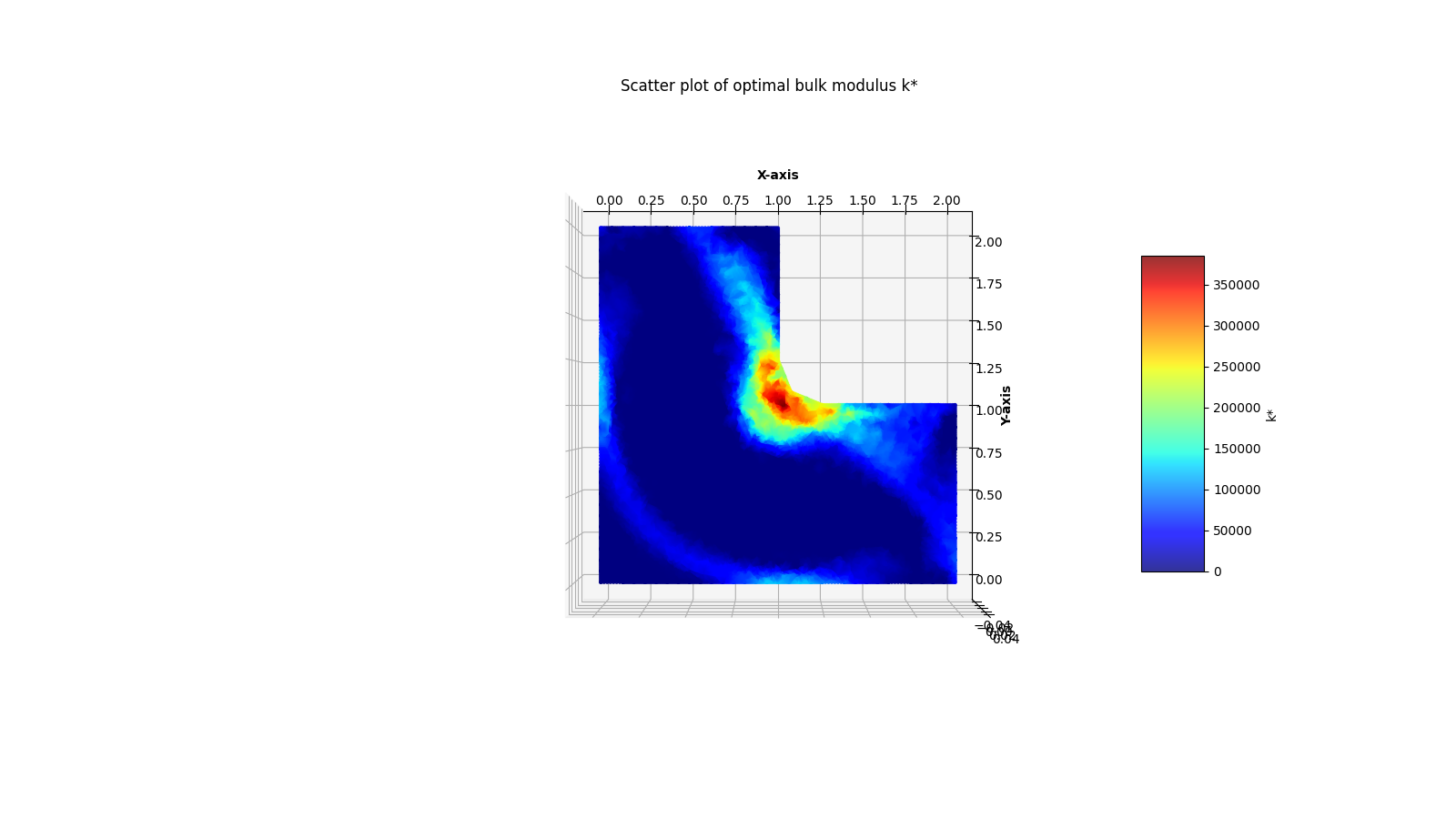}}\hspace{2cm}
	\subfloat[(b)]{\includegraphics*[trim={16.3cm 6.2cm 5.4cm 4.7cm},clip,width=0.30\textwidth]{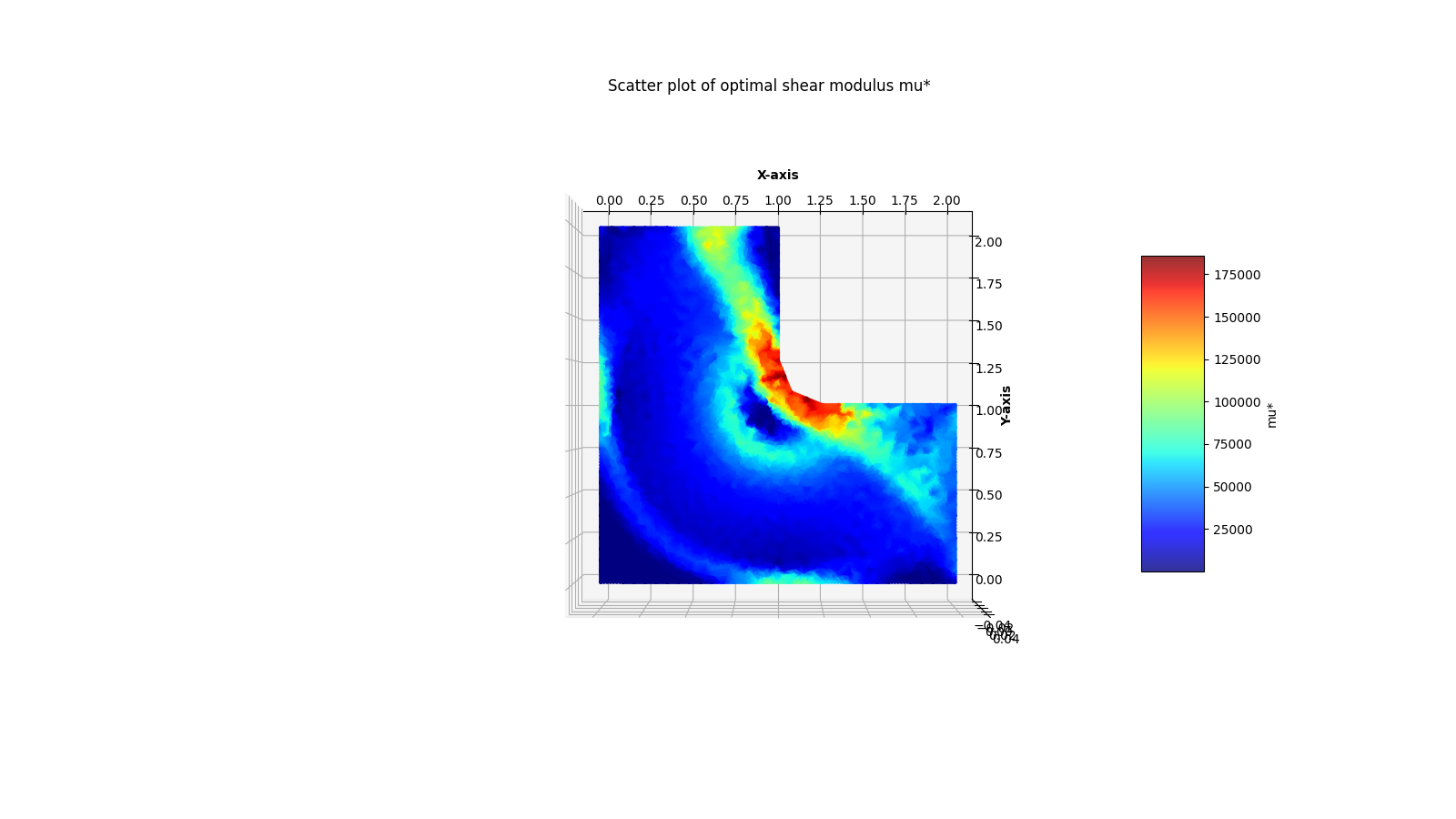}}\\
	\subfloat[(c)]{\includegraphics*[trim={16.3cm 6.2cm 5.4cm 4.7cm},clip,width=0.30\textwidth]{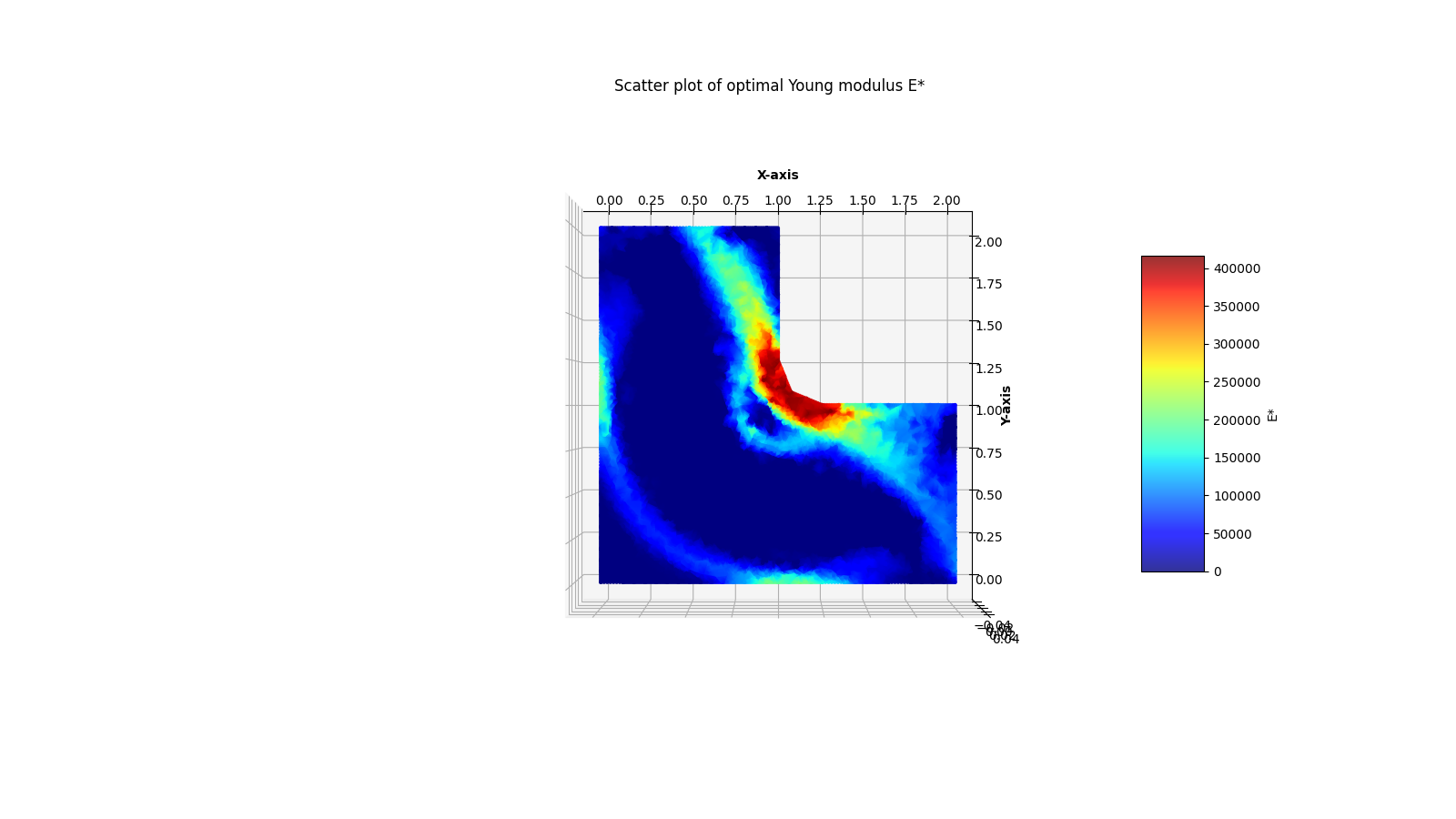}}\hspace{2cm}
	\subfloat[(d)]{\includegraphics*[trim={16.3cm 6.2cm 5.4cm 4.7cm},clip,width=0.30\textwidth]{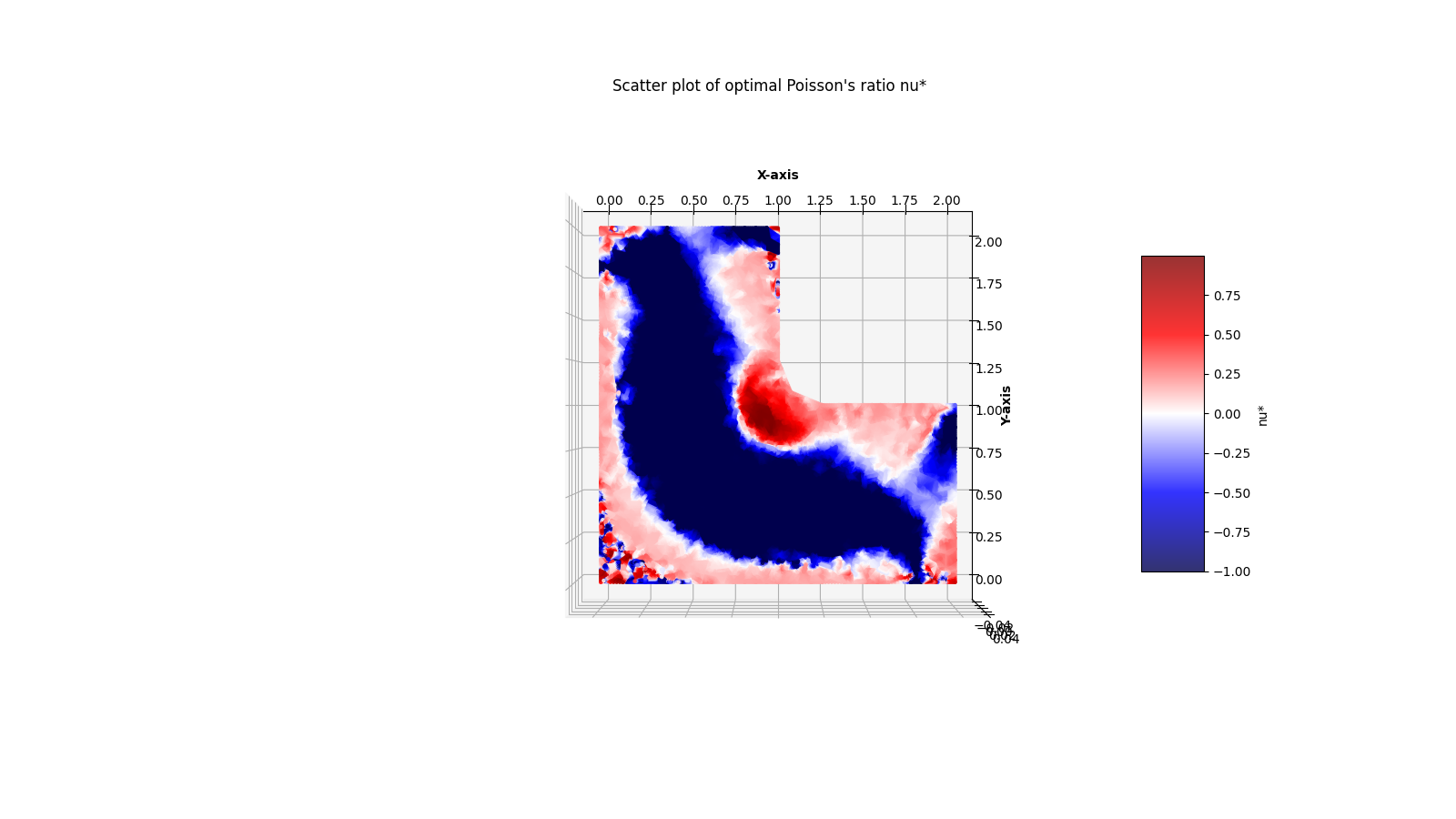}}
	\caption{L-shaped domain problem, solution of the (IMD) problem with the yield condition according to \cite{czarnecki2021isotropic}: (a) optimal bulk modulus $\hat{k}$; (b) optimal shear modulus $\hat{\mu}$; (c) optimal Young modulus $\hat{E}$; (d) optimal Poisson ratio $\hat{\nu}$.}
	\label{fig:mises}       
\end{figure}
The introduction of a similar procedure seems possible in the more general case of the cost conditions considered in this paper. On the other hand, maintaining a constant value of the maximum allowable yield stress $\sigma _0 = \mathrm{const}$ with spatially varying values of the optimal moduli $\hat{k},\hat{\mu}$ (especially when they reach a value equal to or close to zero in certain sub-domains) seems to be an incorrect assumption at the level of macro analysis. However, from the point of view of the analysis at the level of microstructure modelling, the yield condition in the adopted form can be considered correct.

\bigskip

\subsection*{Acknowledgment} This paper was  co-financed under the research
grant of the Warsaw University of Technology supporting the
scientific activity in the discipline of Civil Engineering, Geodesy
and Transport.

\bibliographystyle{plain}

\bigskip

\noindent
\small{Karol Bo{\l}botowski}:\\
Department of Structural Mechanics and Computer Aided Engineering\\
Faculty of Civil Engineering, Warsaw University of Technology\\
16 Armii Ludowej Street, 00-637 Warsaw - POLAND\\
and\\
Lagrange Mathematics and Computing Research Center\\
103 rue de Grenelle, Paris 75007 - FRANCE\\
{\tt karol.bolbotowski@pw.edu.pl}

\bigskip 
\noindent
S{\l}awomir Czarnecki:\\
Department of Structural Mechanics and Computer Aided Engineering\\
Faculty of Civil Engineering, Warsaw University of Technology\\
16 Armii Ludowej Street, 00-637 Warsaw - POLAND\\
{\tt slawomir.czarnecki@pw.edu.pl}

\bigskip 
\noindent
Tomasz Lewi\'{n}ski:\\
Department of Structural Mechanics and Computer Aided Engineering\\
Faculty of Civil Engineering, Warsaw University of Technology\\
16 Armii Ludowej Street, 00-637 Warsaw - POLAND\\
{\tt tomasz.lewinski@pw.edu.pl}

\end{document}